\documentclass[12pt,review]{amsart}
\pdfoutput=1
\usepackage{amsthm, amsmath,amsfonts,amssymb,euscript,hyperref,graphics,color,slashed}
\usepackage{graphicx}
\usepackage{mathrsfs}
\usepackage{comment}
\usepackage{import}
\usepackage{tikz}
\usepackage{latexsym}
\usepackage[makeroom]{cancel}
\usepackage[font=normalsize]{caption}
\usepackage{subfigure}
\usepackage[title]{appendix}
\usepackage{enumerate}
\usepackage{mathrsfs}
\usepackage{empheq}   
\usepackage{bm}
\usepackage{bookmark}
\usepackage{arydshln}
\usepackage{lineno,hyperref}
\usepackage{cite}
\usepackage{float}
\usepackage[utf8]{inputenc}
\usepackage{dashbox}
\usepackage{ulem}
\newcommand\dboxed[1]{\dbox{\ensuremath{#1}}}

\def\inte#1{
\displaystyle\mathop{#1\kern0pt}^\circ }

\def\virgp{\raise 2pt\hbox{,}}
\def\cdotpv{\raise 2pt\hbox{;}}

\def\C{\mathop{\mathbb C\kern 0pt}\nolimits}
\def\DD{\mathop{\mathbb D\kern 0pt}\nolimits}
\def\EE{\mathop{{\mathbb E \kern 0pt}}\nolimits}
\def\K{\mathop{\mathbb K\kern 0pt}\nolimits}
\def\N{\mathop{\mathbb N\kern 0pt}\nolimits}
\def\Q{\mathop{\mathbb Q\kern 0pt}\nolimits}
\def\R{\mathop{\mathbb R\kern 0pt}\nolimits}
\def\SS{\mathop{\mathbb S\kern 0pt}\nolimits}
\def\ZZ{\mathop{\mathbb Z\kern 0pt}\nolimits}
\def\TT{\mathop{\mathbb T\kern 0pt}\nolimits}
\def\P{\mathop{\mathbb P\kern 0pt}\nolimits}

\newcommand{\beq}{\begin{equation}}
\newcommand{\eeq}{\end{equation}}
\newcommand{\ben}{\begin{eqnarray}}
\newcommand{\een}{\end{eqnarray}}
\newcommand{\beno}{\begin{eqnarray*}}
\newcommand{\eeno}{\end{eqnarray*}}

\newtheorem{thm}{Theorem}[section]
\newtheorem{lem}{Lemma}[section]

\newtheorem{prop}{Proposition}[section]

\newtheorem*{Main Theorem}{Main Theorem}
\newtheorem{theorem}{Theorem}[section]

\newtheorem{remark}[theorem]{Remark}

\numberwithin{equation}{section}

\allowdisplaybreaks

\begin{document}
\title[Ill-posedness and shock formation of compressible MHD]{Low regularity ill-posedness and shock formation\\ for 3D ideal compressible MHD}

\author[Xinliang An]{Xinliang An}
\address{Department of Mathematics, National University of Singapore\\ Singapore}
\email{matax@nus.edu.sg}

\author[Haoyang Chen]{Haoyang Chen}
\address{Department of Mathematics, National University of Singapore\\ Singapore}
\email{hychen@nus.edu.sg}

\author[Silu Yin]{Silu Yin}
\address{Department of Mathematics, Hangzhou Normal University\\ Hangzhou, China}
\email{yins@hznu.edu.cn}

\begin{abstract}
The study of magnetohydrodynamics (MHD) significantly boosts the understanding and development of solar physics, planetary dynamics and controlled nuclear fusion. Dynamical properties of the MHD system involve nonlinear interactions of waves with multiple travelling speeds (the fast and slow magnetosonic waves, the Alfv\'{e}n wave and the entropy wave). One intriguing topic is the shock phenomena accompanied by the magnetic field, which have been affirmed by astronomical observations. However, permitting the residence of all above multi-speed waves, mathematically, whether one can prove shock formation for three dimensional (3D) MHD is still open. The multiple-speed nature of the MHD system makes it fascinating and challenging. In this paper, we report our recent progress in answering the above question. For 3D ideal compressible MHD, we construct planar symmetric examples of shock formation allowing the presence of all characteristic waves with multiple wave speeds. This also answers a question raised by Majda on conservation law in 1984. Building on our construction, we further prove that the Cauchy problem for 3D ideal MHD is $H^2$ ill-posed. And this is caused by the shock formation. In particular, when the magnetic field is absent, we also provide a desired low-regularity ill-posedness result for the 3D compressible Euler equations, and it is sharp with respect to the regularity of the fluid velocity. Our proof for 3D MHD is based on a coalition of a carefully designed algebraic approach and a geometric approach. To trace the nonlinear interactions of various waves, we algebraically decompose the 3D ideal MHD equations into a $7\times 7$ non-strictly hyperbolic system. Via detailed calculations, we reveal its hidden subtle structures. With them we give a complete description of MHD dynamics up to the earliest singular event, when a shock forms.
\end{abstract}

\maketitle

\tableofcontents

\section{Introduction}
Magnetohydrodynamics (MHD) plays a crucial role in astrophysics, planetary magnetism and controlled fusion reactor. It operates on every scales, from vast on stellar winds and solar magnetic activity to the small on controlled fusion plasmas.

Mathematically, the equations of MHD system describe the dynamics and magnetic properties of electrically conducting fluid or plasma. To study planetary dynamics, liquid metals and electrolyte, vibrant mathematical explorations of incompressible MHD equations are conducted, ranging from well-posedness theory to singularity formation mechanism and covering topics from Cauchy problems to free boundary problems.

Compared with the incompressible counterpart, there is less mathematical literature investigating the compressible MHD. On the contrary, the physical phenomena governed by compressible MHD are extremely rich, which include the jets from accretion disks at a galactic scale, stellar magnetism describing solar activity, and controlled nuclear fusions in a designed tokamak.
\begin{figure}[htbp]
\centering
\begin{minipage}[t]{0.48\textwidth}
\centering
\includegraphics[width=7.3cm,height=3.6cm]{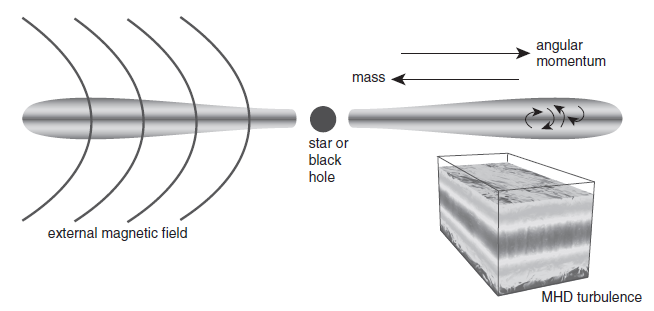}
\caption{\small{\bf Schematic of an accretion disk.} From \cite{davidson} by Davidson.}
\end{minipage}
\begin{minipage}[t]{0.48\textwidth}
\centering
\includegraphics[width=6cm]{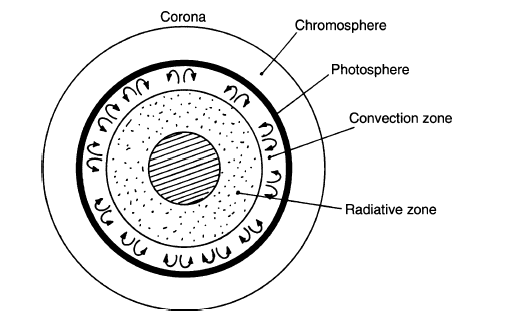}
\caption{\small{\bf The structure of the sun.} Convection zone is the source of MHD activity, including the initiation of coronal flux tubes and sunspots. From \cite{davidson} by Davidson.}
\end{minipage}
\end{figure}

Similarly to other compressible fluids with zero viscosity, shock waves are anticipated to play a significant role in the context of 3D ideal compressible MHD with negligible resistivity/viscosity. For instance, a shock wave is expected to form when the solar wind interacts with the earth's magnetic field.
\begin{figure}[htbp]
\centering
\includegraphics[width=0.6\linewidth]{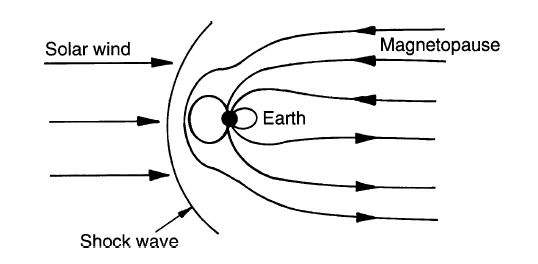}
\captionof{figure}{\small{\bf The interaction of the earth's magnetic field with the solar wind.} From \cite{davidson} by Davidson.}
\label{solar wind}
\end{figure}

Shock-wave-related phenomena of MHD raise broad interests among the physical and the numerical communities. The evolution of compressible MHD involves interactions of waves with multiple speeds: the fast and slow magnetosonic waves, the Alfv\'{e}n wave and the entropy wave. Mathematically, an \underline{open} question is to give a rigorous proof of shock formation allowing the residence of all above multi-speed waves.

In this paper, under planar symmetry we provide a \underline{confirmative} answer to the above question for the 3D ideal compressible MHD. We give a constructive proof including detailed analysis of the solutions' behaviours up to the (shock) singularity. The proof relies on a combination of a carefully designed algebraic approach and a geometric approach. And the shock forms due to the compression of characteristics. Furthermore, based on shock formation, we prove that the Cauchy problems of the 3D ideal compressible MHD equations are ill-posed in $H^2(\mathbb{R}^3)$. We show that the underlying mechanism of this ill-posedness is the shock formation. Our ill-posedness result for the 3D ideal compressible MHD allows nonlinear interactions of multiple-speed waves. In the absence of the magnetic field, we also provide the desired ill-posedness result for the 3D compressible Euler equations.

Our research toward low-regularity ill-posedness for MHD is inspired by remarkable works on wave equations. Under planar symmetry, Lindblad gave sharp counterexamples to the low-regularity local well-posedness for semilinear and certain quasilinear wave equations in \cite{lindblad93,Lind96,Lind98}. In \cite{lindblad17}, Ettinger-Lindblad generalized the above results to the Einstein's equations and constructed a sharp counterexample for local well-posedness of Einstein vacuum equations in wave coordinates. An exploration by Granowski \cite{Ross} later showed that Lindblad's $H^s$ ill-posedness in \cite{Lind98} is stable under general perturbations. Without any symmetry assumption, Alinhac \cite{Alinhac99,Alinhac99II} firstly proved shock formation for solutions to quasilinear wave equations in more than one spatial dimension. He employed a Nash-Moser iteration scheme to overcome the loss of derivatives in energy estimates. This approach relies heavily on a non-degeneracy condition posed on initial data and fails to reveal information beyond the first blow-up point. In 2009, a breakthrough \cite{christodoulou10} was made by Christodoulou. Based on a fully geometric approach, he sharpened Alinhac's results by giving a detailed understanding and a complete description of shock formation for 3D irrotational Euler equations. With the full use of the geometry therein, Christodoulou derived sharper energy estimates which avoid the potential loss of derivatives. This geometric approach was later extended to a larger class of wave equations and Euler equations with or without vorticity. See Christodoulou-Miao \cite{christodoulou-miao}, Luk-Speck \cite{Speck-luk,Speck-luk2}, Miao-Yu \cite{miao}, Speck \cite{speckbk,Speck18}, Speck-Holzegel-Luk-Wong \cite{Speck16}. Recently, applying a different approach via self-similar ansatz, Buckmaster-Shkoller-Vicol\cite{buckmaster1,buckmaster} constructed shock formation for compressible Euler equations in 2D and 3D allowing the presence of non-trivial vorticity.

Compared with the single wave equation, fewer results are known for quasilinear wave systems. Quasilinear wave systems with physical background (e.g. elasticity and MHD) typically describe phenomena involving interactions of waves with multiple travelling speeds. An important and challenging task is to study singularity formation for these multi-wave-speed hyperbolic systems, including elastic waves, compressible MHD, crystal optics, self-gravitating relativistic fluids, etc. In \cite{majda}, Majda addressed this as an open question in conservation law on proving shock formation for general $m \times m$ non-genuinely non-strictly hyperbolic systems. Motivated by this question, we are developing a \underline{new} approach in this direction and plan to apply it to the aforementioned physical systems. The first application of this approach was given by us in \cite{an}. In \cite{an}, we generalized Lindblad's work on a scalar wave equation in \cite{Lind98} by showing that the Cauchy problems for 3D elastic waves, a physical system with multiple wave speeds, are ill-posed in $H^3(\mathbb{R}^3)$. We also quantitatively demonstrated that the ill-posedness is driven by shock formation. Our estimates in \cite{an} rely on the smallness of certain coefficients of nonlinear terms in the elastic wave system. In this paper, we further advance our approach, equipped with a renormalization procedure when choosing right eigenvectors and more robust a priori estimates which applies to more general hyperbolic systems without the smallness coefficients as in the elastic wave system. Now we can prove a desired result for an important quasilinear non-strictly hyperbolic system with multiple wave-speed: the 3D ideal compressible MHD system. Compared with the elastic waves, the MHD system possesses more subtle and more complex structures, which imposes many new challenges to establish ill-posedness and shock formation. (See Subsection \ref{difficulty} for details.) The method introduced here would act as a firm step towards the study of general non-genuinely nonlinear non-strictly hyperbolic systems.

\subsection{Main theorems}\label{sec1.1}
We start from inducing the equations. The 3D ideal compressible MHD, composed of the Euler equations of fluid dynamics and the coupled Maxwell's equations for magnetic field, forms a quasilinear hyperbolic system:
\begin{equation}\label{MHD}
\left\{\begin{split}
&\partial_t\varrho+\nabla\cdot(\varrho u)=0,\\
&\varrho\{\partial_t+(u\cdot\nabla)\}u+\nabla p+\mu_0 H\times\text{rot}H=0,\\
&\partial_t H-\text{rot}(u\times H)=0,\\
&\partial_t S+(u\cdot\nabla)S=0,\\
&\nabla\cdot H=0.
\end{split}
\right.
\end{equation}
Here, $\mu_0$ is the magnetic permeability constant, $\varrho$ is the fluid density\footnote{In our analysis, the value of $\varrho$ is close to $1$.}, $u=(u_1,u_2,u_3) \in \mathbb{R}^3$ is the fluid velocity, $H=(H_1,H_2,H_3) \in \mathbb{R}^3$ is the magnetic field intensity, $S$ is the entropy and $p$ is the pressure satisfying the equation of state $p=p(\varrho,S).$ To describe the solar wind, following \cite{davidson} we use the polytropic equation of state
\begin{equation}\label{state}
p=A e^S\varrho^\gamma
\end{equation}
with $A$ being a positive constant and $\gamma>1$ being the adiabatic gas constant.

Following Lindblad's approach toward proving ill-posedness, we study \eqref{MHD} under planar symmetry. By Gauss's Law, $H_1$ is always a constant. We denote this constant parameter by $\kappa$. The constant parameter $\kappa$ satisfies the following condition (see also \eqref{h1}) 
\begin{equation} \label{h10}
	\kappa^2 \ll \min\{\mu_0^{-1}A\gamma,1\}.
\end{equation}	
Here,  we impose the smallness of $H_1$ to guarantee the strict separation of two characteristic speeds (see Section \ref{pre} for the details). See Section \ref{pre} for the details. Therefore, we actually work with the following system:
\begin{equation} \label{reduce}
\partial_t \Phi(x,t)+A(\Phi)\partial_{x} \Phi(x,t)=0,
\end{equation}
where $\Phi=(u_1,u_2,u_3,\varrho-1, H_2,H_3,S)^T$ and $A(\Phi)\in \mathbb{R}^{7\times 7}$ is a $7\times 7$ matrix. For notational simplicity, we denote $x_1$ by $x$ in this paper. Our proof relies on a carefully designed algebraic wave decomposition of this $7\times 7$ system. See Subsection \ref{step} and Section \ref{pre} for details. The delicate structures of the decomposed system play a crucial role in our proof. In the following theorem, these subtle structures are revealed for the \underline{first} time.
\begin{thm}[{\bf Structures of MHD equations}] \label{structurethm}
For the MHD system \eqref{reduce}: $\partial_t \Phi(x,t)+A(\Phi)\partial_{x} \Phi(x,t)=0$, with $H_1=\kappa$ being a constant parameter satisfying \eqref{h10}, one can rewrite it into the following form:
\begin{equation} \label{decomposed system}
\begin{cases}
  \partial_{s_i}\rho_i=c_{ii}^iv_i+\Big(\sum\limits_{m\neq i}c_{im}^iw_m\Big)\rho_i,\\
  \partial_{s_i}w_i=-c_{ii}^iw_i^2+\Big(\sum\limits_{m\neq i}(-c_{im}^i+\gamma_{im}^i)w_m\Big)w_i+\sum\limits_{m\neq i,k\neq i\atop m\neq k}\gamma_{km}^iw_kw_m,\\
  \partial_{s_i}v_i=\Big(\sum\limits_{m\neq i}\gamma_{im}^iw_m\Big)v_i+\sum\limits_{m\neq i,k\neq i\atop m\neq k}\gamma_{km}^iw_kw_m\rho_i.
  \end{cases}
\end{equation}
Here, $\partial_{s_i}=\lambda_i\partial_x+\partial_t$ with $\lambda_i$ being the $i^{\text{th}}$ eigenvalue of the $7 \times 7$ matrix $A(\Phi)$. And $\rho_i:=\partial_{z_i}X_i$ is the inverse density of the characteristics\footnote{We use $(X_i(z_i,s_i),s_i)$ to denote the trajectory of the characteristics with $(z_i,s_i)$ being the characteristic coordinates. See \eqref{flow} for detailed definitions.}. We further define $w_i:=l_i\partial_x\Phi$ and $v_i:=\rho_iw_i$, where $l_i$ is the $i^{\text{th}}$ left eigenvector of $A(\Phi)$. With a 17-page calculation, we prove that the coefficients $c^i_{ii},c^i_{im},\gamma^i_{im},\gamma^i_{km}$ being non-trivial functions of $\Phi$ are \underline{all} of order $1$. (See Section \ref{structure} and Appendix \ref{appendixB} for the details\footnote{These coefficients could potentially be singular due to the singular factors present in the eigenvectors.}.) This enables us to reveal the following structures for the system \eqref{decomposed system}. For $\rho_i$, we have
\small\begin{equation}\label{infmrho}
 \left\{ \begin{split}
    \partial_{s_1}\rho_1\thicksim&v_1+\rho_1\big(\sum\limits_{i\neq 1}w_i\big),\\ \partial_{s_2}\rho_2\thicksim&\cancel{v_2}+\boxed{\rho_2w_3}+\rho_2\big(\sum\limits_{i\neq 2,3}w_i\big),\\
    \partial_{s_3}\rho_3\thicksim& v_3+\cancel{\rho_3 w_2}+\rho_3\big(\sum\limits_{i\neq 2,3}w_i\big),\\
    \partial_{s_4}\rho_4\thicksim& \cancel{v_4}+\rho_4\big(\sum\limits_{i\neq 4}w_i\big),\\
    \partial_{s_5}\rho_5\thicksim& v_5+\boxed{\rho_5w_6}+\rho_5\big(\sum\limits_{i\neq 5,6}w_i\big),\\
    \partial_{s_6}\rho_6\thicksim&\cancel{v_6}+\cancel{\rho_6 w_5}+\rho_6\big(\sum\limits_{i\neq 5,6}w_i\big),\\
    \partial_{s_7}\rho_7\thicksim&v_7+\rho_7\big(\sum\limits_{i\neq 7}w_i\big).\\
  \end{split}\right.
\end{equation}
Here, the notation $\cancel{v_2}$ denotes the absence of the term $v_2$, $\boxed{\rho_2w_3}$ denotes the strong interaction term. For $w_i$, we have
\small{\begin{equation}\label{wsketch1}
 \left\{ \begin{split}
    \partial_{s_1} w_1\thicksim&(w_1)^2+\dboxed{(\lambda_2-\lambda_3)w_2w_3}+\dboxed{(\lambda_5-\lambda_6)w_5w_6}+{\sum_{m\neq 1,k\neq 1, m\neq k\atop (m,k)\neq(2,3),(m,k)\neq(5,6)}w_mw_k},\\
    \partial_{s_2} w_2\thicksim&\cancel{(w_2)^2}+\boxed{w_2w_3}+\dboxed{(\lambda_5-\lambda_6)w_5w_6}+\sum_{m\neq 2,k\neq 2, m\neq k\atop (m,k)\neq(5,6)}w_mw_k,\\
    \partial_{s_3} w_3\thicksim&(w_3)^2+\dboxed{(\lambda_2-\lambda_3)w_2w_3}+\dboxed{(\lambda_5-\lambda_6)w_5w_6}+\sum_{m\neq 3,k\neq 3, m\neq k\atop (m,k)\neq(5,6)}w_mw_k,\\
    \partial_{s_4} w_4\thicksim&\cancel{(w_4)^2}+\dboxed{(\lambda_2-\lambda_3)w_2w_3}+\dboxed{(\lambda_5-\lambda_6)w_5w_6}+\sum_{m\neq 4,k\neq 4, m\neq k\atop (m,k)\neq(2,3),(m,k)\neq(5,6)}w_mw_k,\\
    \partial_{s_5} w_5\thicksim&(w_5)^2+\dboxed{(\lambda_2-\lambda_3)w_2w_3}+\dboxed{(\lambda_5-\lambda_6)w_5w_6}+\sum_{m\neq 5,k\neq 5, m\neq k\atop (m,k)\neq(5,6)}w_mw_k,\\
    \partial_{s_6} w_6\thicksim&\cancel{(w_6)^2}+\dboxed{(\lambda_2-\lambda_3)w_2w_3}+\boxed{w_5w_6}+\sum_{m\neq 6,k\neq 6, m\neq k\atop (m,k)\neq(2,3)}w_mw_k,\\
    \partial_{s_7} w_7\thicksim&(w_7)^2+\dboxed{(\lambda_2-\lambda_3)w_2w_3}+\dboxed{(\lambda_5-\lambda_6)w_5w_6}+\sum_{m\neq 7,k\neq 7, m\neq k\atop (m,k)\neq(2,3),(m,k)\neq(5,6)}w_mw_k.\\
  \end{split}\right.
\end{equation}}
Here, the notation $\dboxed{(\lambda_2-\lambda_3)w_2w_3}$ denotes the weak interaction term. For $v_i$, we have
\small{\begin{equation}\label{infmv}
 \left\{ \begin{split}
    \partial_{s_1}v_1\thicksim&v_1\big(\sum\limits_{i\neq 1}w_i\big)+\rho_1\bigg(\dboxed{(\lambda_2-\lambda_3)w_2w_3}+\dboxed{(\lambda_5-\lambda_6)w_5w_6}+\sum_{m\neq 1,k\neq 1, m\neq k\atop (m,k)\neq(2,3),(m,k)\neq(5,6)}w_mw_k\bigg),\\
    \partial_{s_2}v_2\thicksim&\dboxed{(\lambda_2-\lambda_3)v_2 w_3}+\big(\sum\limits_{i\neq 2,3}w_i\big)v_2+\rho_2\bigg(\dboxed{(\lambda_5-\lambda_6)w_5w_6}+\sum_{m\neq 2,k\neq 2, m\neq k\atop (m,k)\neq(5,6)}w_mw_k\bigg),\\
    \partial_{s_3}v_3\thicksim& \dboxed{(\lambda_2-\lambda_3)v_3 w_2}+\big(\sum\limits_{i\neq 2,3}w_i\big)v_3+\rho_3\bigg(\dboxed{(\lambda_5-\lambda_6)w_5w_6}+\sum_{m\neq 3,k\neq 3, m\neq k\atop (m,k)\neq(5,6)}w_mw_k\bigg),\\
    \partial_{s_4}v_4\thicksim&v_4\big(\sum\limits_{i\neq 4}w_i\big)+\rho_4\bigg(\dboxed{(\lambda_2-\lambda_3)w_2w_3}+\dboxed{(\lambda_5-\lambda_6)w_5w_6}+\sum_{m\neq 4,k\neq 4, m\neq k\atop (m,k)\neq(2,3),(m,k)\neq(5,6)}w_mw_k\bigg),\\
    \partial_{s_5}v_5\thicksim&\dboxed{(\lambda_5-\lambda_6)v_5 w_6}+\big(\sum\limits_{i\neq 5,6}w_i\big)v_5+\rho_5\bigg(\dboxed{(\lambda_2-\lambda_3)w_2w_3}+\sum_{m\neq 5,k\neq 5, m\neq k\atop (m,k)\neq(2,3)}w_mw_k\bigg),\\
    \partial_{s_6}v_6\thicksim&\dboxed{(\lambda_5-\lambda_6)v_6 w_5}+\big(\sum\limits_{i\neq 5,6}w_i\big)v_6+\rho_6\bigg(\dboxed{(\lambda_2-\lambda_3)w_2w_3}+\sum_{m\neq 6,k\neq 6, m\neq k\atop (m,k)\neq(2,3)}w_mw_k\bigg),\\
    \partial_{s_7}v_7\thicksim&v_7\big(\sum\limits_{i\neq 7}w_i\big)+\rho_7\bigg(\dboxed{(\lambda_2-\lambda_3)w_2w_3}+\dboxed{(\lambda_5-\lambda_6)w_5w_6}+\sum_{m\neq 7,k\neq7, m\neq k\atop (m,k)\neq(2,3),(m,k)\neq(5,6)}w_mw_k\bigg).\\
  \end{split}\right.
\end{equation}}
\end{thm}
\begin{remark}[\bf Potentially harmful terms]
In \eqref{wsketch1}-\eqref{infmv}, some potentially harmful terms are present. We list them on the right hand side of the following equations
\begin{align}
&\partial_{s_2}w_2=\boxed{(\gamma_{23}^2-c_{23}^2)w_2w_3}+\cdots,\label{abw2}\\
&\partial_{s_6}w_6=\boxed{(\gamma_{65}^6-c_{65}^6)w_5w_6}+\cdots,\label{abw6}\\
&\partial_{s_2}\rho_2=\boxed{c_{23}^2\rho_2w_3}+\cdots,\label{abr2}\\
&\partial_{s_6}\rho_6=\boxed{c_{65}^6\rho_6w_5}+\cdots.\label{abr6}
\end{align}
These potentially harmful terms could result in strong nonlinear interactions between the two families of waves ($w_2$ and $w_3$, $w_5$ and $w_6$) with almost-the-same speeds. To handle these nonlinear terms is the main task of Section \ref{pfthm1.2}. Compared with elastic waves in \cite{an}, there the coefficients of these potentially harmful nonlinear terms have additional smallness, which makes the interactions weaker (and thus easier to control) than those in MHD.
\end{remark}
\begin{remark}
In Theorem \ref{structurethm}, our conclusion that all $c^i_{ii}(\Phi),c^i_{im}(\Phi),\gamma^i_{im}(\Phi),\gamma^i_{km}(\Phi)$ are of order $1$ is highly non-trivial. This is true because we carefully design the wave decomposition in this paper and \underline{all} the potentially singular factors in $c^i_{ii}(\Phi),c^i_{im}(\Phi),\gamma^i_{im}(\Phi),\gamma^i_{km}(\Phi)$ are \underline{cancelled}. To prove this, we conduct a 17-page (mainly in the appendix) calculation for 299 coefficients. See Proposition \ref{coeff} in Section \ref{structure} and Appendix \ref{appendixB} for the crucial discussions.
\end{remark}
John \cite{john74} proved the shock formation for $m\times m$ genuinely non-linear strictly hyperbolic system via introducing the decomposition of waves. For general non-genuinely nonlinear and non-strictly hyperbolic system, shock formation was addressed as an open question by Majda in \cite{majda}. In \cite{liu}, Liu generalized John's result to strictly hyperbolic system allowing certain characteristics to be linearly degenerate\footnote{For linearly degenerate system, we refer to \cite{Zhou} for low-regularity local well-posedness.}. There Liu imposed structural conditions that the coefficients of the decomposed system satisfy $\gamma^i_{jk}\equiv 0$ when $c^i_{ii}\neq 0$ and $c^j_{jj}=c^k_{kk}\equiv 0$.

Based on John's method, we are developing a new approach to study non-genuinely nonlinear, non-strictly hyperbolic systems, where all the characteristic waves could interact with each other, i.e., all the coefficients $\gamma^i_{jk}$ are allowed to be non-zero. In particular, our approach extends \cite{john74,liu} and applies to the elastic waves equations in \cite{an} and the physical MHD system in this paper.

We are now in a good position to answer Majda's question in \cite{majda} with the following result: we construct examples of planar symmetric shock formation for 3D ideal compressible MHD equations allowing the presence of nonlinear interactions of all multi-speed waves.
\begin{thm}[{\bf Shock formation for MHD}] \label{shock}
For the 3D ideal compressible MHD system, there exists a large class of planar symmetric compactly supported smooth Cauchy initial data that lead to finite-time shock formations.

More precisely, let $\varepsilon\in(0,\frac1{100})$ be a fixed parameter. There is a sufficiently small constant $\theta_0>0$ such that for any $0<\theta<\theta_0$, if the initial data of the decomposed system \eqref{decomposed system} verify $w_1(x,0)=O(\theta)$ and $w_i(x)=O(\theta^2)$ with $i\neq 1$, then the corresponding solution to \eqref{decomposed system} satisfies the following statements:
\begin{itemize}
\item[\textup{i)}] {\bf(Shock formation along characteristics)} There exists a $z_0$ lying in the support of the initial data and a corresponding finite time $T^*(z_0)$, such that, a shock forms at time $T^*(z_0)$, i.e.,
\begin{equation*}
  \lim_{t\rightarrow T^*(z_0)} \rho_1(z_0,t)=0\quad \text{and}\quad \lim_{t\rightarrow T^*(z_0)} w_1(z_0,t)=\infty.
\end{equation*}
The blow-up of $w_1$ implies $\lim_{t\rightarrow T^*(z_0)}|\partial_x \Phi|(z_0,t)=\infty$. In particular, we show the precise upper and lower bound estimates for $\rho_1(z_0,t)$ as below:
\begin{equation}
	\begin{split}
	\rho_1(z_0,t)\leq (1+\varepsilon)\Big(1-(1-\varepsilon)^4t|c_{11}^1(0)|w_1(z_0,0)\Big),\\
	\rho_1(z_0,t)\geq 	(1-\varepsilon)\Big(1-(1+\varepsilon)^3t|c_{11}^1(0)|w_1(z_0,0)\Big),
		\end{split}
\end{equation}
where $c_{11}^1(0)$ is a negative constant. The lifespan $T^*(z_0)$ of $w_1$ then satisfies
\begin{equation}
 \frac{1}{(1+\varepsilon)^3|c_{11}^1(0)|w_1(z_0,0)}\leq T^*(z_0)\leq\frac{1}{(1-\varepsilon)^3|c_{11}^1(0)|w_1(z_0,0)}.
\end{equation}

\item[\textup{ii)}] {\bf(Lower bound of other $\rho_i$)} For any $i\in\{2,\cdots,7\}$, we have $w_i(z_i,t)=O(\theta)$ and $\rho_i(z_i,t)>C$ with $C$ being a uniform positive constant for all $z_i\in\mathbb{R}$ and any $t<T^*(z_0)$.

\item[\textup{iii)}] {\bf(Boundedness of $\Phi$)} The original variable $\Phi$ is uniformly bounded. We have $\Phi(z_i,t)=O(\theta)$ for all $z_i\in\mathbb{R}$ and any $t<T^*(z_0)$.

\item[\textup{iv)}] {\bf(Boundedness of $\rho_i$ and $v_i$)} For any $i\in\{1,\cdots,7\}$, we have $\rho_i(z_i,t)=O(1)$ and $v_i(z_i,t)=O(\theta)$ for all $z_i\in\mathbb{R}$ and any $t<T^*(z_0)$.

\end{itemize}
\end{thm}
\begin{remark} \label{rkshock}
For the decomposed system \eqref{decomposed system}, to establish ill-posedness, we prescribe initial data which are compactly supported on $[\eta,2\eta]$ with $\eta>0$. Then, along different families of characteristics, there are seven strips $\mathcal{R}_{i}$, $i=1\cdots 7$ issuing from the support $[\eta,2\eta]$. To conduct our analysis, we use the following picture. The shock is constructed to happen at $t=T_\eta^*$ in the first characteristic strip $\mathcal{R}_1$ with the fastest propagation wave speed. For $t<T_\eta^*$, there is \underline{no} other singular point in the spacetime region and the solution is smooth. See Figure \ref{fig1}.
\end{remark}
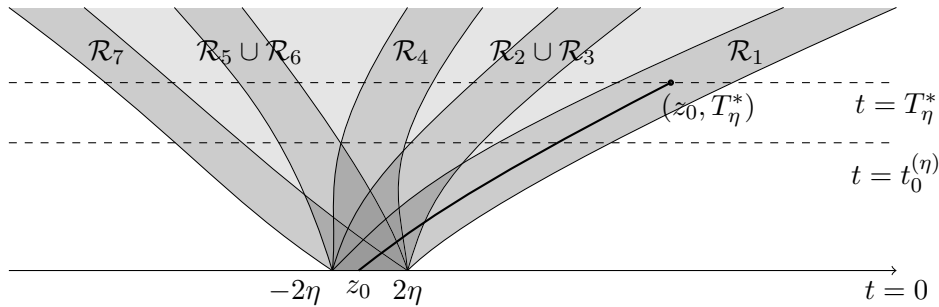
\begin{figure}[htb]
\centering
\begin{tikzpicture}[fill opacity=0.5, draw opacity=1, text opacity=1]
\node [below]at(3.5,0){$2\eta$};
\node [below]at(2.3,0){$\eta$};

\filldraw[white, fill=gray!40] (3.5,0)..controls (2,1) and (1,2)..(-0.8,3.5)--(0.4,3.5)..controls (1,2.8) and (1.9,2)..(2.5,0);
\filldraw[white, fill=gray!40](3.5,0)..controls (3.2,1) and (2,2.6)..(1.3,3.5)--(3.5,3.5)..controls (2.2,1.5) and (2.6,1)..(2.5,0);
\filldraw[white, fill=gray!40](3.5,0)..controls (3.4,1) and (3,1.5)..(4.5,3.5)--(5.5,3.5)..controls (3.2,1.5) and (2.8,1)..(2.5,0);
\filldraw[white, fill=gray!40](3.5,0)..controls (3.8,1) and (4,1.5)..(6.6,3.5)--(8.5,3.5)..controls (3.9,1.5) and (3.5,1)..(2.5,0);

\filldraw[white, fill=gray!80] (2.5,0)..controls (3.5,1) and (3.9,1.5)..(8.5,3.5)--(10,3.5)..controls (5.5,1.5) and (4,0.5)..(3.5,0)--(2.5,0);
\filldraw[white,fill=gray!80](2.5,0)..controls (2.8,1) and (3.2,1.5)..(5.5,3.5)--(6.6,3.5)..controls (4,1.5) and (3.8,1)..(3.5,0);
\filldraw[white,fill=gray!80](2.5,0)..controls (2.6,1) and (2.2,1.5)..(3.5,3.5)--(4.5,3.5)..controls (3,1.5) and (3.4,1)..(3.5,0);
\filldraw[white,fill=gray!80](2.5,0)..controls (1.9,2) and (1,2.8)..(0.4,3.5)--(1.3,3.5)..controls (2,2.6) and (3.2,1)..(3.5,0);
\filldraw[white,fill=gray!80](2.5,0)..controls (1,1) and (0.5,1.8)..(-1.8,3.5)--(-0.8,3.5)..controls (1,2) and (2,1)..(3.5,0);
\draw[->](-1.8,0)--(10,0)node[left,below]{$t=0$};
\draw[dashed](-1.8,1.7)--(10,1.7)node[right,below]{$t=t_0^{(\eta)}$};
\draw[dashed](-1.8,2.5)--(10,2.5)node[right,below]{$t=T_\eta^*$};

\filldraw [thick,black] (7,2.5) circle [radius=0.8pt];
\draw [thick,color=black](2.85,0)..controls (3.5,0.5) and (4.1,1)..(7,2.5);
\node [below]at(2.85,0){$z_0$};
\node [below]at(7.5,2.5){$(z_0,T_\eta^*)$};

\draw (3.5,0)..controls (4,0.5) and (5.5,1.5)..(10,3.5);

\draw (2.5,0)..controls (3.5,1) and (3.9,1.5)..(8.5,3.5);
\node [below] at(8,3.2){$\mathcal{R}_1$};

\draw [color=black](3.5,0)..controls (3.8,1) and (4,1.5)..(6.6,3.5);

\draw [color=black](2.5,0)..controls (2.8,1) and (3.2,1.5)..(5.5,3.5);
\node [below] at(5.3,3.2){$\mathcal{R}_2\cup\mathcal{R}_3$};

\draw [color=black](3.5,0)..controls (3.4,1) and (3,1.5)..(4.5,3.5);

\draw [color=black](2.5,0)..controls (2.6,1) and (2.2,1.5)..(3.5,3.5);
\node [below] at(3.55,3.2){$\mathcal{R}_4$};

\draw [color=black](3.5,0)..controls (3.2,1) and (2,2.6)..(1.3,3.5);
\node [below] at(1.4,3.2){$\mathcal{R}_5\cup\mathcal{R}_6$};

\draw [color=black] (2.5,0)..controls (1.9,2) and (1,2.8)..(0.4,3.5);

\draw [color=black](3.5,0)..controls (2,1) and (1,2)..(-0.8,3.5);
\node [below] at(-0.5,3.2){$\mathcal{R}_{7}$};

\draw [color=black] (2.5,0)..controls (1,1) and (0.5,1.8)..(-1.8,3.5);
\end{tikzpicture}
\caption{\small{\bf First shock forms at $(z_0,T_\eta^*)$ in the characteristic strip $\mathcal{R}_1$.} In this picture, the domain of our consideration is divided into three regions: (a) The grey region is called the characteristic strips. (b) The light grey region denotes the \underline{disjoint} domain between two separate strips. (c) The dark grey region denotes the domain where the strips overlap with each other. }
\label{fig1}
\end{figure}
\begin{remark}
We design the above strips $\{\mathcal{R}_1,\mathcal{R}_2\cup\mathcal{R}_3,\mathcal{R}_4,\mathcal{R}_5\cup\mathcal{R}_6,\mathcal{R}_{7}\}$ and regions with colors to conduct our analysis. We remark that the solution is not necessarily being zero outside of the characteristic strips. In particular, the solutions' dynamics in the light-grey region of Figure \ref{fig1} is far from being trivial. For our proof, it is crucial to investigate the different behaviours of the solutions in and out of these characteristic strips. We will derive a priori estimates for solutions in these colored regions in Subsection \ref{apes}.
\end{remark}
\begin{remark}
Our proof of shock formation here relies on combining the algebraic wave decomposition and a geometric understanding. In particular, we employ the geometric quantity--the inverse density of characteristics $\rho_i$ to quantify the compression of the characteristics. The vanishing of $\rho_i$ characterizes the shock formation. We prove that $\rho_1\to 0$ within $\mathcal{R}_1$. And all the other $\rho_i$ are bounded from below by a uniform positive constant. Our proof here is \underline{not} a qualitative proof by contradiction. By investigating the dynamics of $\rho_i$, we understand the shock formation and its associated precise quantitative behaviours.
\end{remark}
\begin{remark}
		Recall that, under plane symmetry, the parameter $H_1$ is a constant satisfying condition \eqref{h10}. Note that in \eqref{h10}, the smallness of $H_1$ is only imposed to ensure the strict separation of two characteristic speeds (see Section \ref{pre} for the details). This parameter is not relevant to the shock formation argument of Theorem \ref{shock}. 
\end{remark}

Moreover, based on our constructed shock formation, we further provide a counterexample to $H^2$ local well-posedness. It is caused by the shock formation. The ill-posedness theorem is stated as below.
\begin{thm}[{\bf $H^2$ ill-posedness for MHD}] \label{3D}
The Cauchy problems of the 3D ideal compressible MHD equations \eqref{MHD} are ill-posed in $H^2(\mathbb{R}^3)$ in the following sense:

There exists a family of compactly supported, smooth initial data satisfying
$${\|\Phi^{(\eta)}(\vec{x},0)\|}_{{{H}}^2(\mathbb{R}^3)} \to 0\quad \text{as} \quad \eta\to0,$$
where $\eta>0$ is a small parameter and it identifies the datum in this family. For each $\eta$, the Cauchy problem of the 3D ideal MHD \eqref{reduce} admits a solution that ceases to be regular in finite time. Let $T_\eta^*$ be the largest time such that the solution $\Phi^{(\eta)}\in C^\infty\big(\mathbb{R}^3\times[0,T_\eta^*) \big)$.

  Then, for each solution $\Phi^{(\eta)}$, evolving from the corresponding initial datum, a shock forms at $T_\eta^*$. Let $\Omega_{T_\eta^*}$ be a spatial neighborhood of the first (shock) singularity, the $H^1$-norm $\|\Phi^{(\eta)}\|_{H^1(\Omega_{T_\eta^*})}$ blows up at $T^*_{\eta}$. In particular, for any $\eta$, we have
\begin{equation} \label{h1blowup}
  \|\partial_x u_1^{(\eta)}(\cdot,T_\eta^*)\|_{L^2(\Omega_{T_\eta^*})}=+\infty,\ \|\partial_x \varrho^{(\eta)}(\cdot,T_\eta^*)\|_{L^2(\Omega_{T_\eta^*})}=+\infty.
\end{equation}

\end{thm}
\begin{remark} \label{stronger inflation}
    The $H^1$ blow-up mechanism in \eqref{h1blowup} is \emph{stronger} than the usual norm inflation: for each fixed $\eta$, the $H^1$-norm of the solution becomes infinite at a finite time $T_\eta^*$, rather than merely growing arbitrarily large (inflating) as $\eta \to 0$.
\end{remark}
\begin{remark} \label{rhoill}
  Note that the $L^\infty$ blow-up of the solutions' first derivatives in Theorem \ref{shock} is not adequate to imply the blow-up of the $H^1$ norm as in \eqref{h1blowup}. To prove \eqref{h1blowup}, we crucially employ the \underline{vanishing of the inverse foliation density $\rho$}, which is the geometric description of the shock formation. Our exploration of the characteristics' geometric behaviours via $\rho$ bridges the gap between the shock formation (pointwise blow-up of the first derivatives) and the blow-up of $H^1$-norms. See Subsection \ref{ill} for the details.
\end{remark}
\begin{remark} \label{new}
For 3D ideal compressible MHD (\ref{MHD}), both our examples of the shock formation allowing the presence of waves with multiple travelling speeds and the associated low-regularity ill-posedness are \underline{new}. With the subtle structures we found, for our constructed solutions, we give a \underline{complete} description of the MHD dynamics up to the time $T^*_{\eta}$ when the shock singularity happens.
\end{remark}
\begin{remark} \label{nonsym data}
To prove the $H^2$ ill-posedness, we employ our 1+1 dimensional shock-formation result in Theorem \ref{shock} within $\mathcal{D}_\eta$, that is the development region of an initial small ball $B_{\frac\eta2}^3$ at $t=0$ along the characteristics. In particular, the initial data we prescribe are planar symmetric inside $B_{\frac\eta2}^3$. Due to finite speed of propagation, the solutions to the 3D MHD system still obey plane symmetry in $\mathcal{D}_\eta$, which allows for the application of our Theorem \ref{shock}. It is worthwhile to note that the initial data outside of $B_{\frac\eta2}^3$ are extended to be compactly supported, and hence depend on $x_2,x_3$ as well. Therefore, the overall initial data are \emph{not planar symmetric} on the entire $\mathbb{R}^3$. See Subsection \ref{illdata} for further discussions.
\end{remark}
If the magnetic field vanishes, i.e., $H_1=H_2=H_3\equiv0$, our system \eqref{MHD} reduces to the 3D compressible Euler equations. Our results thus give
\begin{thm}[{\bf $H^2$ ill-posedness for Euler equations}] \label{sharpeuler}
Consider the following 3D compressible Euler equations allowing non-trivial entropy and vorticity
\begin{equation}\label{euler}
\left\{\begin{split}
&\partial_t\varrho+\nabla\cdot(\varrho u)=0,\\
&\varrho\{\partial_t+(u\cdot\nabla)\}u+\nabla p=0,\\
&\partial_t S+(u\cdot\nabla)S=0.
\end{split}
\right.
\end{equation}
Then the Cauchy problem of \eqref{euler} is $H^2$ ill-posed with respect to the fluid velocity and density. More precisely, there exits a family of compactly supported smooth initial data for \eqref{euler} satisfying
$${\|\varrho_0^{(\eta)}\|}_{{{H}}^2(\mathbb{R}^3)}+{\|u_0^{(\eta)}\|}_{{{H}}^2(\mathbb{R}^3)}+{\|S_0^{(\eta)}\|}_{{{H}}^2(\mathbb{R}^3)} \to 0\quad \text{as} \quad \eta\to0.$$
After a finite time $T_\eta^*$ a shock forms and the solution ceases to be smooth. Furthermore, tied to this family of initial data, the $H^1$-norm of the velocity $u_1$ and density $\varrho$ blow up at the shock formation time $T^*_{\eta}$:
\begin{equation*}
  \|u_{1}^{(\eta)}(\cdot,T_\eta^*)\|_{H^1(\Omega_{T_\eta^*})}=+\infty,\quad \|\varrho^{(\eta)}(\cdot,T_\eta^*)\|_{H^1(\Omega_{T_\eta^*})}=+\infty
\end{equation*}
where $\Omega_{T_\eta^*}$ is a spatial region around the first (shock) singularity.
\end{thm}
\begin{remark}
The above ill-posedness result for the 3D compressible Euler equations is sharp with respect to the regularity of the fluid velocity $u$ and density $\varrho$. The classical local well-posedness (LWP) result for the 3D compressible Euler equations holds in $H^s$ with $s>\frac52$ for initial data of $u$ and $\varrho$. The regularity for local well-posedness is lowered to $s>2$ for fluid velocity and density by Disconzi-Luo-Mazzone-Speck \cite{Disconzi}, Wang \cite{Wang19} and Zhang-Andersson \cite{zhang-andersson}. Our ill-posedness result thus acts as a sharp counterpart of these local well-posedness results. For the incompressible case, Bourgain-Li \cite{bourgain-li} proved the sharp low-regularity ill-posedness result with respect to the regularity of the vorticity $\omega$. Their initial data require fluid velocity $u$ and vorticity $\omega$ satisfying $(u,\omega) \in H^\frac52 \times H^{\frac32}$. With the picture below, one can see that our Theorem \ref{sharpeuler} serves as a sharp companion with respect to the fluid velocity $u$ and density $\varrho$. See Section \ref{euler ill} for more detailed discussions.
\end{remark}
\begin{figure}[H]
\centering
\begin{tikzpicture}[fill opacity=0.5, draw opacity=1, text opacity=1,scale=1.2]
\filldraw[white, fill=gray!40](1,0.5)--(1,4)--(5,4)--(1.5,0.5)--(1,0.5);
\filldraw[gray!40, fill=gray!80](1,0.5)--(1,1)--(1.5,1)--(1.5,0.5)--(1,0.5);
\draw[->](0,-0.3)--(5.5,-0.3);
\filldraw(5,-0.3)node[below]{\footnotesize $s$: regularity for the velocity};
\draw[->](0,-0.3)--(0,4)node[right,above]{\footnotesize $s'$: regularity for the vorticity};
\draw(1.5,0.5)--(5,4);
\draw[dashed](1.5,0.5)--(1.5,4);
\draw[dashed](1,1)--(1.5,1);
\draw[dashed](1,4)--(5,4);
\filldraw [thick,black] (1.5,0.5) circle [radius=0.8pt];
\filldraw [thick,black] (1,0.5) circle [radius=0.8pt];
\filldraw [thick,black] (1,-0.3) circle [radius=0.8pt]node[below]{2};
\filldraw [thick,black](1.5,-0.3) circle [radius=0.8pt]node[below]{$\frac{5}{2}$};

\filldraw [thick,black](0,0.5) circle [radius=0.8pt]node[left]{$\frac32$};
\filldraw [thick,black](0,1) circle [radius=0.8pt]node[left]{2};

\draw[thick,black](1,0.5)--(1.5,0.5);
\draw[very thick,black](1,0.5)--(1,4);
\filldraw(1.65,2.2)circle node[right]{\scriptsize LWP proved by Wang(19') (See also};
\filldraw(1.65,1.9)circle node[right]{\scriptsize Disconzi-Luo-Mazzone-Speck(19'),};
\draw[->](1.75,1.9)--(1.25,1.9);
\filldraw(1.65,1.6)circle node[right]{\scriptsize Zhang-Andersson(21'))};
\filldraw(1.65,3)circle node[right]{\scriptsize Classical LWP};
\filldraw(1.65,0.75)circle node[right]{\scriptsize LWP by Conjecture in Wang(19')};
\filldraw (2,2) circle;
\draw[->](1.75,0.75)--(1.25,0.75);
\draw[->](1.75,0.3)--(1.25,0.45);
\filldraw(1.65,0.3)node[right]{\scriptsize Bourgain-Li's ill-posedness(21')};
\filldraw(1.65,0)node[right]{\scriptsize (incompressible 3D Euler)};
\draw[->](-0.1,2.25)--(0.95,2.25);
\filldraw(0,2.5)node[left]{\footnotesize An-Chen-Yin's };
\filldraw(0,2.2)node[left]{\footnotesize ill-posedness};
\end{tikzpicture}
\end{figure}

\subsection{Difficulties and Strategies} \label{difficulty}
To prove the ill-posedness result, as in \cite{an,lindblad17,Lind96,Lind98} we work under planar symmetry, the MHD system then obeys
\begin{equation}\label{main}
\left\{\begin{split}
&\varrho\partial_tu_1+\varrho u_1\partial_x u_1+c^2\partial_x\varrho+\partial_S p\partial_xS+\mu_0H_2\partial_xH_2+\mu_0H_3\partial_xH_3=0,\\
&\varrho\partial_tu_2+\varrho u_1\partial_x u_2-\mu_0H_1\partial_xH_2=0,\\
&\varrho\partial_tu_3+\varrho u_1\partial_x u_3-\mu_0H_1\partial_xH_3=0,\\
&\partial_t\varrho+u_1\partial_x\varrho+\varrho\partial_x u_1=0,\\
&\partial_tH_2+u_1\partial_x H_2-H_1\partial_xu_2+H_2\partial_xu_1=0,\\
&\partial_tH_3+u_1\partial_x H_3-H_1\partial_xu_3+H_3\partial_xu_1=0,\\
&\partial_t S+u_1\partial_xS=0,
\end{split}
\right.
\end{equation}
where $c=\sqrt{\partial_\varrho p}$ denotes the sound speed. For \eqref{main}, by Gauss's Law, $H_1$ is a constant and hence \eqref{main} can be viewed as a $7\times 7$ first-order hyperbolic system:
\begin{equation}\label{1order}
        \partial_t\Phi+A(\Phi)\partial_x\Phi=0,
\end{equation}
where $\Phi=(u_1,u_2,u_3,\varrho-1, H_2,H_3,S)^T$ and $A(\Phi)$ is a $7\times7$ matrix. We then explore the structures of \eqref{1order} and proceed to do analysis, but several difficulties arise:

{\bf $\bullet$ Invalidity of Riemann invariants for a $7\times 7$ system:} For hyperbolic system with different travelling speeds, to prove shock formation, a key ingredient is to understand the nonlinear interaction of waves propagating in different speeds. For a $2\times 2$ first-order hyperbolic system, there exist two Riemann invariants, which can be constructed explicitly and verify two transport equations. These two Riemann invariants carry crucial information about the nonlinear structure of the $2\times 2$ first-order hyperbolic system. And with them other geometric quantities can be solved along the characteristics. However, for a large system, such as our $7\times 7$ system, it is almost impossible to find proper Riemann invariants with explicit formulas. Here we adopt a \underline{different} approach. We appeal to John's classic method \cite{john74}, i.e., decomposition of waves. Applying the following formula of decomposition of waves
\begin{equation*}
  \partial_x\Phi=\sum_{k=1}^7 w_kr_k,
\end{equation*}
we rewrite \eqref{1order} as a \underline{diagonal} Riccati-type system:
\begin{equation}\label{w}
(\partial_t+\lambda_i\partial_x)w_i=-c_{ii}^i(w_i)^2+\Big(\sum_{m\neq i}(-c_{im}^i+\gamma_{im}^i)w_m\Big)w_i+\sum_{m\neq i,k\neq i\atop m\neq k}\gamma_{km}^iw_kw_m,
\end{equation}
where $w_i:=l_i(\Phi)\partial_x\Phi$ and $l_i(\Phi)$ are the left eigenvectors of $A(\Phi)$. Here the coefficients $c_{ii}^i$, $c_{im}^i$, $\gamma_{km}^i$, $\gamma_{im}^i$ are functions of the unknowns $\Phi$. A \underline{fundamental} technical point in our work is to bound these coefficients. This is closely relevant to the choice of right eigenvectors $r_i$.

{\bf $\bullet$ Non-strict hyperbolicity:} For the hyperbolic system \eqref{1order}, via algebraic calculations, we notice that the MHD system (\ref{main}) is \underline{non-strictly} hyperbolic\footnote{For more discussions on the non-strictly hyperbolic system, we refer to Liu-Xin \cite{liu-xin}.}, with eigenvalues of the coefficient matrix $(A(\Phi))_{7\times7}$ satisfying
\begin{align*}
\begin{split}
    \lambda_7(\Phi)<\lambda_6(\Phi)\leq \lambda_5(\Phi)<\lambda_4(\Phi)<\lambda_3(\Phi)\leq\lambda_2(\Phi)<\lambda_1(\Phi).
\end{split}
\end{align*}
When $\Phi$ is sufficiently small, two pairs of characteristics have almost-the-same travelling speed $\lambda_2(\Phi)\thickapprox\lambda_3(\Phi),\ \lambda_5(\Phi)\thickapprox\lambda_6(\Phi).$ This yields the long-time overlapping of the corresponding characteristic strips $\mathcal{R}_2,\mathcal{R}_3$ (or $\mathcal{R}_5,\mathcal{R}_6$) mentioned in Remark \ref{rkshock}. Substantial nonlinear interactions will be carried out in the overlapped regions and they may effect other strips. We overcome this obstacle by taking unions of the overlapped strips and considering the wave propagations in and out of the following five strips:
\begin{equation} \label{strips}
\{\mathcal{R}_1,\mathcal{R}_2\cup\mathcal{R}_3,\mathcal{R}_4,\mathcal{R}_5\cup\mathcal{R}_6,\mathcal{R}_7\}.
  \end{equation}
Tracing the dynamics of MHD (a $7\times 7$ non-strictly hyperbolic system) with analyzing the solutions' behaviours in and out of these five characteristic strips $\{\mathcal{R}_1, \mathcal{R}_2\cup\mathcal{R}_3, \mathcal{R}_4,\mathcal{R}_5\cup\mathcal{R}_6, \mathcal{R}_7\}$ is \underline{new}. These five strips will completely separate from each other when $t>t_0^{(\eta)}$, where $t_0^{(\eta)}$ can be precisely calculated. See the following picture sketching the dynamics:
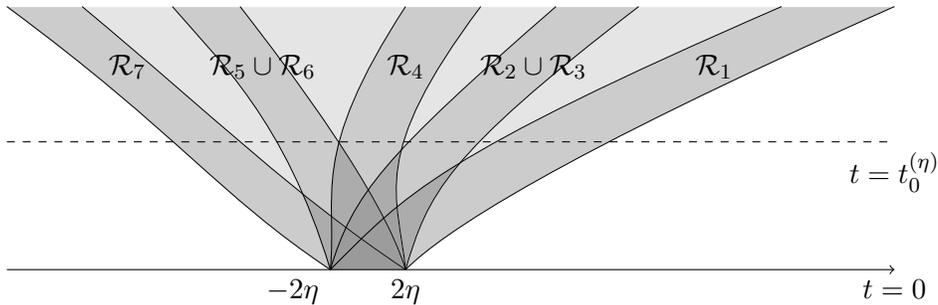
\begin{figure}[htbp]
\centering
\begin{tikzpicture}[fill opacity=0.5, draw opacity=1, text opacity=1]
\node [below]at(3.5,0){$2\eta$};
\node [below]at(2.3,0){$\eta$};

\filldraw[white, fill=gray!40] (3.5,0)..controls (2,1) and (1,2)..(-0.8,3.5)--(0.4,3.5)..controls (1,2.8) and (1.9,2)..(2.5,0);
\filldraw[white, fill=gray!40](3.5,0)..controls (3.2,1) and (2,2.6)..(1.3,3.5)--(3.5,3.5)..controls (2.2,1.5) and (2.6,1)..(2.5,0);
\filldraw[white, fill=gray!40](3.5,0)..controls (3.4,1) and (3,1.5)..(4.5,3.5)--(5.5,3.5)..controls (3.2,1.5) and (2.8,1)..(2.5,0);
\filldraw[white, fill=gray!40](3.5,0)..controls (3.8,1) and (4,1.5)..(6.6,3.5)--(8.5,3.5)..controls (3.9,1.5) and (3.5,1)..(2.5,0);

\filldraw[white, fill=gray!80] (2.5,0)..controls (3.5,1) and (3.9,1.5)..(8.5,3.5)--(10,3.5)..controls (5.5,1.5) and (4,0.5)..(3.5,0)--(2.5,0);
\filldraw[white,fill=gray!80](2.5,0)..controls (2.8,1) and (3.2,1.5)..(5.5,3.5)--(6.6,3.5)..controls (4,1.5) and (3.8,1)..(3.5,0);
\filldraw[white,fill=gray!80](2.5,0)..controls (2.6,1) and (2.2,1.5)..(3.5,3.5)--(4.5,3.5)..controls (3,1.5) and (3.4,1)..(3.5,0);
\filldraw[white,fill=gray!80](2.5,0)..controls (1.9,2) and (1,2.8)..(0.4,3.5)--(1.3,3.5)..controls (2,2.6) and (3.2,1)..(3.5,0);
\filldraw[white,fill=gray!80](2.5,0)..controls (1,1) and (0.5,1.8)..(-1.8,3.5)--(-0.8,3.5)..controls (1,2) and (2,1)..(3.5,0);
\draw[->](-1.8,0)--(10,0)node[left,below]{$t=0$};
\draw[dashed](-1.8,1.7)--(10,1.7)node[right,below]{$t=t_0^{(\eta)}$};

\draw (3.5,0)..controls (4,0.5) and (5.5,1.5)..(10,3.5);

\draw (2.5,0)..controls (3.5,1) and (3.9,1.5)..(8.5,3.5);
\node [below] at(7.6,3){$\mathcal{R}_1$};

\draw [color=black](3.5,0)..controls (3.8,1) and (4,1.5)..(6.6,3.5);

\draw [color=black](2.5,0)..controls (2.8,1) and (3.2,1.5)..(5.5,3.5);
\node [below] at(5.2,3){$\mathcal{R}_2\cup\mathcal{R}_3$};

\draw [color=black](3.5,0)..controls (3.4,1) and (3,1.5)..(4.5,3.5);

\draw [color=black](2.5,0)..controls (2.6,1) and (2.2,1.5)..(3.5,3.5);
\node [below] at(3.5,3){$\mathcal{R}_4$};

\draw [color=black](3.5,0)..controls (3.2,1) and (2,2.6)..(1.3,3.5);
\node [below] at(1.6,3){$\mathcal{R}_5\cup\mathcal{R}_6$};

\draw [color=black] (2.5,0)..controls (1.9,2) and (1,2.8)..(0.4,3.5);

\draw [color=black](3.5,0)..controls (2,1) and (1,2)..(-0.8,3.5);
\node [below] at(-0.2,3){$\mathcal{R}_{7}$};

\draw [color=black] (2.5,0)..controls (1,1) and (0.5,1.8)..(-1.8,3.5);
\end{tikzpicture}
\caption{\small\bf Separation of five characteristic strips.}
\end{figure}
Estimates before $t_0^{(\eta)}$ can be carried out similarly as in \cite{christodoulou} by Christodoulou-Perez, where they studied the propagation of electromagnetic waves in nonlinear crystals and solved a first-order strictly hyperbolic system. However, after $t_0^{(\eta)}$, the non-strict hyperbolicity dramatically amplifies the nonlinear interactions. $\mathcal{R}_2$ and $\mathcal{R}_3$ (or $\mathcal{R}_5$ and $\mathcal{R}_6$) could overlap with each other for a long time. This overlapping feature results in strong nonlinear interaction. Here we illustrate the intuition with the interactions in $\mathcal{R}_2$ and $\mathcal{R}_3$. In strictly hyperbolic case as in \cite{christodoulou}, since $\mathcal{R}_2$ and $\mathcal{R}_3$ will separate after $t_0^{(\eta)}$, the interaction of corresponding characteristic waves is weak. However, for our non-strictly hyperbolic system, due to the overlapping feature, the characteristics $\mathcal{C}_{2},\mathcal{C}_{3}$ may still intersect in $\mathcal{R}_{2}$ (or $\mathcal{R}_{3}$) after a long time. This forces us to estimate in $\mathcal{R}_2\cup\mathcal{R}_3$ where the interaction is conceivably stronger. Nonetheless, we show that the strong interaction will not jeopardize the desired blow-up mechanism. To guarantee that the interaction terms are still controllable, we need to employ the subtle structures of the equations. This will be discussed in the following paragraph.

{\bf $\bullet$ Subtle structures in MHD system have to be employed:} The subtleties lying in the structures of \eqref{w} are of paramount importance in our proof. A thorough analysis of the structures is essential for tracing the geometric behaviours of the characteristics in and out of the above five strips \eqref{strips} up to shock formation time $T^*_{\eta}$. We list some crucial structures for $w_i$ here:
\begin{equation}\label{wsketch}
 \left\{ \begin{split}
    (\partial_t+\lambda_1\partial_x)w_1\thicksim&(w_1)^2+(\lambda_2-\lambda_3)w_2w_3+(\lambda_5-\lambda_6)w_5w_6+\cdots,\\
    (\partial_t+\lambda_2\partial_x)w_2\thicksim&\cancel{(w_2)^2}+\boxed{w_2w_3}+(\lambda_5-\lambda_6)w_5w_6+\cdots,\\
    (\partial_t+\lambda_3\partial_x)w_3\thicksim&(w_3)^2+(\lambda_2-\lambda_3)w_2w_3+(\lambda_5-\lambda_6)w_5w_6+\cdots,\\
    (\partial_t+\lambda_4\partial_x)w_4\thicksim&\cancel{(w_4)^2}+(\lambda_2-\lambda_3)w_2w_3+(\lambda_5-\lambda_6)w_5w_6+\cdots,\\
    (\partial_t+\lambda_5\partial_x)w_5\thicksim&(w_5)^2+(\lambda_2-\lambda_3)w_2w_3+(\lambda_5-\lambda_6)w_5w_6+\cdots,\\
    (\partial_t+\lambda_6\partial_x)w_6\thicksim&\cancel{(w_6)^2}+(\lambda_2-\lambda_3)w_2w_3+\boxed{w_5w_6}+\cdots,\\
    (\partial_t+\lambda_7\partial_x)w_7\thicksim&(w_7)^2+(\lambda_2-\lambda_3)w_2w_3+(\lambda_5-\lambda_6)w_5w_6+\cdots.\\
  \end{split}\right.
\end{equation}
The coefficients of the deleted terms in \eqref{wsketch} are proved to be \underline{zero}. In other words, there is no Riccati-type terms appearing in the equations of $\{w_i\}_{i=2,4,6}$. The vanishing of these terms indicates the linear degeneracy of the corresponding characteristics. In contrast, the non-vanishing of Riccati-type terms such as ${(w_1)}^2,{(w_3)}^2,{(w_5)}^2, {(w_7)}^2$ demonstrate the genuine nonlinearity of the corresponding characteristics and they may drive the blow-ups of the solutions.

Now we compare with the circumstance for the elastic waves in \cite{an}. The system of $w_i$ in \cite{an} reads
\begin{equation}\label{wsketchelaticity}
 \left\{ \begin{split}
    (\partial_t+\lambda_1\partial_x)w_1\thicksim&(w_1)^2+\cancel{w_2w_3}+\cancel{w_4w_5}+\cdots,\\
    (\partial_t+\lambda_2\partial_x)w_2\thicksim&\cancel{(w_2)^2}+(\lambda_2-\lambda_3)w_2w_3+(\lambda_4-\lambda_5)w_4w_5+\cdots,\\
    (\partial_t+\lambda_3\partial_x)w_3\thicksim&(w_3)^2+\cancel{w_2w_3}+\cancel{w_4w_5}+\cdots,\\
    (\partial_t+\lambda_4\partial_x)w_4\thicksim&(w_4)^2+\cancel{w_2w_3}+\cancel{w_4w_5}+\cdots,\\
    (\partial_t+\lambda_5\partial_x)w_5\thicksim&\cancel{(w_5)^2}+(\lambda_2-\lambda_3)w_2w_3+(\lambda_4-\lambda_5)w_4w_5+\cdots,\\
    (\partial_t+\lambda_6\partial_x)w_6\thicksim&(w_6)^2+\cancel{w_2w_3}+\cancel{w_4w_5}+\cdots
  \end{split}\right.
\end{equation}
where $\lambda_2\thickapprox\lambda_3$ and $\lambda_5\thickapprox\lambda_6$.

In the equations for $w_2$ and $w_5$ in \eqref{wsketchelaticity}, observe that the coefficients of $w_2w_3$ and $w_4w_5$ contain small factors $\lambda_2-\lambda_3$ and $\lambda_4-\lambda_5$ owing to the non-strict hyperbolicity. This suggests that the corresponding related interactions are weak. We were able to employ the bi-characteristic coordinates transformation $(x,t)=\big(X_j(y_j,t'(y_i,y_j)),t'(y_i,y_j)\big)$ to control these terms. For the equation of $w_2$, we integrate along $\mathcal{C}_2$ inside $\mathcal{R}_3$ to deduce
\begin{equation*}
  \begin{split}
    &\int_{\mathcal{C}_2}\big|(\lambda_2-\lambda_3)w_2\big(X_2(z_2,t'),t'\big)w_3\big(X_2(z_2,t'),t'\big)\big|dt'\\
    \leq&\int_{y_3\in[\eta,2\eta]}\Big|w_2w_3\big(y_3,t'(y_2,y_3)\big)\cancel{(\lambda_2-\lambda_3)}\frac{\rho_3\big(y_3,t'(y_2,y_3)\big)}{\cancel{\lambda_2-\lambda_3}}\Big|dy_3.
  \end{split}
\end{equation*}
The only singular factor in the above integral is $\frac{1}{\lambda_2-\lambda_3}$, which comes from the coordinate transformation. As we can see, it can be exactly cancelled by the small factor $\lambda_2-\lambda_3$ in the coefficient.

However, for the ideal compressible MHD, equations of $w_2$ and $w_6$ in \eqref{wsketch} do not include small factor $\lambda_2-\lambda_3$ or $\lambda_5-\lambda_6$ in front of $w_2 w_3$ or $w_5 w_6$, which renders our approach in \cite{an} via the bi-characteristic coordinates to fail. A more delicate choice of the initial data for $w_2, w_3, w_5, w_6$ and additional a priori estimates for
\begin{equation*}
\check{W}(t):=\max\Big\{\sup_{(x',t')\in\mathcal{R}_2\cup\mathcal{R}_3,\atop 0\leq t'\leq t}\{w_2,w_3\},\sup_{(x',t')\in\mathcal{R}_5\cup\mathcal{R}_6,\atop 0\leq t'\leq t}\{w_5,w_6\}\Big\}.
\end{equation*}
are employed to handle these terms.

Moreover, for \eqref{wsketch}, the interaction terms $w_2 w_3$ and $w_5 w_6$ also appear in the equations of $w_1$, $w_3$, $w_4$, $w_5$, $w_7$. While the corresponding terms vanish in \eqref{wsketchelaticity} since their coefficients are all zero. Nevertheless, in \eqref{wsketch} we can prove these nonlinear interactions are weak and negligible compared with the ones in the equations of $w_2$ and $w_6$. In particular, the interaction terms $w_2 w_3, w_5 w_6$ are proved to be small compared with the dominant Riccati terms since they have small coefficients $\lambda_2-\lambda_3,\lambda_5-\lambda_6$. And the smallness here enables us to use bi-characteristic coordinates to control $w_2 w_3, w_5 w_6$ with desired bounds.

{\bf $\bullet$ Delicate choice of right eigenvectors:} To explore the subtle structures of our system, we need to control the following \underline{non-trivial} coefficients $c_{im}^i$, $\gamma_{im}^i$ and $\gamma_{km}^i$ in the decomposed system \eqref{decomposed system}, where
\begin{align*}
&c_{im}^i:=\nabla_\Phi\lambda_i\cdot r_m,\\
  &\gamma_{im}^i:=-(\lambda_i-\lambda_m)l_i \cdot(\nabla_\Phi r_i \cdot r_m-\nabla_\Phi r_m \cdot r_i),\quad m\neq i,\\
  &\gamma_{km}^i:=-(\lambda_k-\lambda_m)l_i \cdot (\nabla_\Phi r_k \cdot r_m), \qquad\qquad\qquad k\neq i,\  m\neq i.
\end{align*}
Here, $\Phi=(u_1,u_2,u_3,\varrho-1, H_2,H_3,S)^T$. And $\lambda_i$, $l_i$, $r_i$ are the eigenvalues, left and right eigenvectors of $A(\Phi)$, respectively. If all these coefficients are of $O(1)$, the system will then take the schematic form presented in \eqref{wsketch} and in Theorem \ref{structurethm}. As we mentioned before, the realization of these estimates for $c_{im}^i$, $\gamma_{im}^i$, $\gamma_{km}^i$ relies heavily on carefully \underline{designed} right eigenvectors $r_i$. For example, to evaluate $\gamma_{64}^2$, one attempt is to take the following right eigenvectors
\begin{equation} \label{badr}
\begin{split}
&\tilde{r}_2=\left(\begin{array}{cc} 0\\ \frac{C_a H_3}{H_1 H_2}\\ -\frac{C_a}{H_1}\\0\\-\frac{H_3}{H_2}\\1\\0\end{array}\right),
\ \tilde{r}_4=\left(\begin{array}{cc} 0\\ 0\\ 0\\-\frac{\varrho}{\gamma}\\0\\0\\1\end{array}\right),
\ \tilde{r}_6=\left(\begin{array}{cc} 0\\ \frac{C_a}{H_1}\\ -\frac{C_a H_2}{H_1 H_3}\\0\\1\\-\frac{H_2}{H_3}\\0\end{array}\right),
\end{split}
\end{equation}
where $C_a$ is the Alfv\'en wave speed (See Section \ref{pre}), and the left eigenvector
\begin{equation} \label{badl}
\tilde{l}_2=(0,\frac{H_1H_2H_3}{2C_a(H_2^2+H_3^2)},-\frac{H_2^2H_1}{2C_a(H_2^2+H_3^2)},0,-\frac{H_2H_3}{2(H_2^2+H_3^2)},\frac{H_2^2}{2(H_2^2+H_3^2)},0).
\end{equation}
Here, $H_1=\kappa \neq 0$ and the elements in $\tilde{l}_2$ are of order $1$. And it follows
\begin{equation} \label{badgamma}
\gamma_{64}^2=-(\lambda_6-\lambda_4)\tilde{l}_2 \cdot (\nabla_\Phi \tilde{r}_6 \cdot \tilde{r}_4)=-\frac{(\lambda_6-\lambda_4)H_2}{4\gamma H_3}.
\end{equation}
However, since both $H_2$ and $H_3$ are unknowns with values close to $0$, the term $\frac{H_2}{H_3}$ is very problematic and may exhibit unexpected singularity. It may render the failure of the whole argument. To overcome this problem, we renormalize and design the right eigenvectors and left eigenvectors in the following form
\begin{equation*}
\begin{split}
&r_2=\left(\begin{array}{cc} 0\\ \frac{C_a H_3}{H_1}\\ -\frac{C_a H_2}{H_1}\\0\\-H_3\\H_2\\0\end{array}\right),\ r_4=\left(\begin{array}{cc} 0\\ 0\\ 0\\-\frac{\varrho}{\gamma}\\0\\0\\1\end{array}\right),\ r_6=\left(\begin{array}{cc} 0\\ \frac{C_a H_3}{H_1}\\ -\frac{C_a H_2}{H_1}\\0\\H_3\\-H_2\\0\end{array}\right),
\end{split}
\end{equation*}
\begin{equation*}
l_2=(0,\frac{H_1H_3}{2C_a(\text{\dashuline{$H_2^2+H_3^2$}})},-\frac{H_2H_1}{2C_a(\text{\dashuline{$H_2^2+H_3^2$}})},0,-\frac{H_3}{2(\text{\dashuline{$H_2^2+H_3^2$}})},\frac{H_2}{2(\text{\dashuline{$H_2^2+H_3^2$}})},0).
\end{equation*}
Note that the right eigenvectors are more regular, but the left eigenvectors are more singular, since they always satisfy $l_ir_j=\delta_{ij}$ and $\tilde{l}_i\tilde{r}_j=\delta_{ij}$. With the new set of $r_2,r_4,r_6$ and $l_2$, we have
\begin{eqnarray*}
\gamma_{64}^2&=&-(\lambda_6-\lambda_4)l_2 \cdot (\nabla_\Phi r_6 \cdot r_4)\\
&=&-(\lambda_6-\lambda_4)l_2 \cdot \Big(0,\frac{C_a H_3}{2H_1\gamma},-\frac{C_a H_2}{2H_1\gamma
},0,0,0,0\Big)^\top \\
&=&-(\lambda_6-\lambda_4)\Big(\frac{H_3^2}{4\gamma(H_2^2+H_3^2)}+\frac{H_2^2}{4\gamma(H_2^2+H_3^2)}\Big)\\
&=&-\frac{\lambda_6-\lambda_4}{4\gamma}=O(1).
\end{eqnarray*}
Compared with \eqref{badgamma}, the singular term is \underline{cancelled} in a specific way. Note that in $l_2$, a potentially singular factor $\frac{1}{H_2^2+H_3^2}$ appears since $H_2$ and $H_3$ are close to zero. But with a detailed matrix multiplication, the singular denominator is cancelled. Not only for $\gamma^2_{64}$, we have checked 299 terms with 17 pages calculations and find that with our designed right and left eigenvectors in \eqref{regv} and \eqref{legv}, all these coefficients $c^i_{im}$, $\gamma^i_{im}$ and $\gamma^i_{km}$ are uniformly bounded.

{\bf $\bullet$ Control of nonlinear interactions of multiple waves:} Due to the multiple-wave-speed trait of the MHD, the decomposed system \eqref{decomposed system} exhibits nonlinear interactions among different characteristic strips. The high complexity of the coupled nonlinear terms poses a major technical difficulty, especially after the separating time $t_0^{(\eta)}$. To disclose how different waves interact with each other, we must distinguish its influence to other characteristic waves in its own strips and in the exterior. Firstly, outside of the strips, we estimate the following quantities
\begin{equation*}
\begin{split}
  & V_i(t):=\sup_{(x',t')\notin\mathcal{R}_i,\atop 0\leq t'\leq t}|w_i(x',t')|,\quad \text{for} \quad i=1,4,7,\\
  &V_{\bar{2}}(t):=\max\sup_{(x',t')\notin\mathcal{R}_2\cup\mathcal{R}_3,\atop 0\leq t'\leq t}\{|w_2(x',t')|,|w_3(x',t')|\},\\
 &V_{\bar{5}}(t):=\max\sup_{(x',t')\notin\mathcal{R}_5\cup\mathcal{R}_6,\atop 0\leq t'\leq t}\{|w_5(x',t')|,|w_6(x',t')|\}, \\
&V(t):=\max_{i}V_i(t)\quad \text{for} \ i=1,\bar{2},4,\bar{5},7.
\end{split}
\end{equation*}
Note that $V_{\bar{2}}$ and $V_{\bar{5}}$ are defined in a form different from that of $V_1$, $V_4$, $V_7$. This is because, as we emphasized before, for our non-strictly hyperbolic system, the almost-repeated characteristic waves $w_2,w_3$ (or $w_5,w_6$) are treated together in $\mathcal{R}_2\cup\mathcal{R}_3$ (or $\mathcal{R}_5\cup\mathcal{R}_6$). We then prescribe initial data such that only $w_1$ will blow up. This is feasible since the equation of $w_1$ is of Riccati type. Yet it is still hard to solve the Riccati equation directly in the presence of other nonlinear terms. To overcome this difficulty, we introduce the \underline{geometric quantities}---the inverse densities $\rho_i$ to characterize the compression of the $i^{\text{th}}$ characteristics. To control $w_i$, we appeal to investigating dynamics of $\rho_i$ and $v_i:=\rho_iw_i$ instead by estimating
\begin{align*}
  S_i(t):=&\sup_{(z'_i,s'_i)\atop z'_i\in[\eta,2\eta],\ 0\leq s'_i\leq t}\rho_i(z'_i,s'_i),&J_i(t):=&\sup_{(z'_i,s'_i)\atop z'_i\in[\eta,2\eta]\ 0\leq s'_i\leq t}|v_i(z'_i,s'_i)|.
\end{align*}
For the purpose of demonstration, we use $\rho_1$ as an example. To bound $\rho_1$, we estimate $S_1$ via integrating \eqref{eqrho}
\begin{equation*}
  \frac{\partial \rho_1}{\partial s_1}=O(J_1+ V S_1).
\end{equation*}
The first step then is to estimate $J_1$. We integrate \eqref{eqv}
\begin{equation*}
  \frac{\partial v_1}{\partial s_1}=O(VJ_1+ V^2 S_1).
\end{equation*}
Here, it is important to analyze the behaviours of the nonlinear terms on the right. The next step is to investigate $V_i$. Since a priori we do not have the estimates for the solution's life span $T$, we employ the bi-characteristic coordinates to control $V_i$. Note that to control $\rho_1$ in $\mathcal{R}_1$, we carefully estimate $J_1$ and all $V_i$ simultaneously. Compared with controlling $\rho_1$ in $\mathcal{R}_1$, the situation in the overlapped strips is more intricate. We refer the readers to Section \ref{apes} for more details.

{\bf $\bullet$ Meaning of ill-posedness:} For ideal compressible MHD (\ref{MHD}), we further prove the $H^2$ ill-posedness in 3D and we also demonstrate that the ill-posedness is driven by shock formation. In particular, there exists a family of initial data $\Phi_0^{(\eta)}$ with their $H^2$ norm tends to $0$, and the $H^2$ norms of their corresponding solutions in 3D blow up at the shock formation time $T^*_{\eta}$.  

Based on our decomposition of waves, we trace the evolution of inverse densities $\rho_i$ ($i=1,\cdots,7$). In $\mathcal{R}_1$, we prove that $\rho_1(z_0,t)\to 0$ as $t\to T_\eta^*$ which renders $T_\eta^*$ to be the first shock formation time. And $\rho_1$ stays positive outside $\mathcal{R}_1$. Then, with a uniform bound of $\partial_{z_1}\rho_1$ obtained in Subsection \ref{supbound}, in a constructed spatial region $\Omega_{T^*_\eta}$, we prove that
\begin{equation*}
\begin{split}
\|\Phi^{(\eta)}(\cdot,T^*_\eta)\|_{H^1(\Omega_{T^*_\eta})}^2&\geq C_\eta{\int_{z_0}^{z_0^*}\frac{1}{\rho_1(z,T_\eta^*)}dz}\overset{(*)}{=}C_\eta{\int_{z_0}^{z_0^*}\frac{1}{\rho_1(z,T_\eta^*)-\rho_1(z_0,T_\eta^*)}dz}\\ &\geq C_\eta\int_{z_0}^{z_0^*}\frac{1}{(z-z_0)\sup|\partial_z\rho_1|}dz\\
&\geq C_\eta\int_{z_0}^{z_0^*}\frac{1}{z-z_0}dz=+\infty.
\end{split}
\end{equation*}
As we mentioned in Remark \ref{rhoill}, one can see from equality $(*)$ that the vanishing of $\rho_1$ crucially drives the above blow-up. These conclusions reveal a hidden mechanism: the shock formation underlies the $H^2$ ill-posedness.

\subsection{Main steps in the proof} \label{step}
We outline our proof of Theorem \ref{structurethm}-Theorem \ref{3D} in this subsection.

\textbf{Step 1: Decomposition of waves (Section \ref{pre}).} Our first step is to algebraically decompose the non-strictly hyperbolic $7\times 7$ system \eqref{1order}. In this paper, we will use both characteristic coordinates $(z_i,s_i)$ and bi-characteristic coordinates $(y_i,y_j)$. Let $\mathcal{C}_i(z_i)$ be the $i^{\text{th}}$ characteristic with propagation speed $\lambda_i$ starting at $z_i$. Denote the corresponding $i^{\text{th}}$ characteristic strip to be $\mathcal{R}_i:=\cup_{z_i\in I_0}\mathcal{C}_i(z_i)$, where $I_0$ is the support of initial data. Define the inverse density of the $i^{\text{th}}$ characteristics
$$\rho_i:=\partial_{z_i}X_i.$$
For fixed $i$, let
\begin{equation*}
  w_i:=l_i\partial_x\Phi, \quad \text{and}\quad
  v_i:=\rho_iw_i.
\end{equation*}
Then, with the left and right eigenvectors satisfying $l_i r_j=\delta_{ij}$, one can decompose the original system in the following way
$$\partial_x\Phi=\sum_k w_kr_k.$$
One can verify that these quantities satisfy
\begin{align}
  \partial_{s_i}\rho_i=&c_{ii}^iv_i+\Big(\sum_{m\neq i}c_{im}^iw_m\Big)\rho_i, \label{122}\\
  \partial_{s_i}w_i=&-c_{ii}^i(w_i)^2+\Big(\sum_{m\neq i}(-c_{im}^i+\gamma_{im}^i)w_m\Big)w_i+\sum_{m\neq i,k\neq i\atop m\neq k}\gamma_{km}^iw_kw_m,\label{123}\\
  \partial_{s_i}v_i=&\Big(\sum_{m\neq i}\gamma_{im}^iw_m\Big)v_i+\sum_{m\neq i,k\neq i\atop m\neq k}\gamma_{km}^iw_kw_m\rho_i.\label{124}
\end{align}

\textbf{Step 2: Analysis of structures (Section \ref{structure}).} We explore the subtle structures of the above decomposed system \eqref{122}-\eqref{124} by carrying out a thorough calculation of all the coefficients $c^i_{ii},c^i_{im},\gamma^i_{im},\gamma^i_{km}$. Firstly, recall that $c_{ii}^i$ being non-vanishing corresponds to the genuine non-linearity. We show that $c_{11}^1<0$ in Lemma \ref{gn}. Hence the genuinely nonlinear condition is verified and $w_1$ satisfies a Riccati-type equation. In fact, the blow-up of $w_1$ builds on the genuine non-linearity. Then, we provide an explicit detailed computation of all coefficients. Based on our calculation, the decomposed system admits a schematic structure \eqref{infmrho}-\eqref{infmv}. We present the key ideas in Section \ref{structure}. The Appendix \ref{appendixB}, a 14-page concrete calculation of the system's structure, supplies all ingredients to complete the proof of Theorem \ref{structurethm}.

\textbf{Step 3: A priori estimates (Subsection \ref{apes}).} For the decomposed system \eqref{decomposed system}, characteristic waves in different $\mathcal{R}_i$ interact with each other. In particular, two pairs of characteristic strips ($\mathcal{R}_2$ and $\mathcal{R}_3$, $\mathcal{R}_5$ and $\mathcal{R}_6$) may overlap for a long time which yields strong nonlinear interactions. Our strategy is to study the dynamics of waves in the following five strips:
\begin{equation*}
\{\mathcal{R}_1,\mathcal{R}_2\cup\mathcal{R}_3,\mathcal{R}_4,\mathcal{R}_5\cup\mathcal{R}_6,\mathcal{R}_7\}.
  \end{equation*}
Then we will prove that a shock happens (and only happens) in $\mathcal{R}_1$. Initially, we prepare compactly supported smooth data $ \{w_i^{(\eta)}(z_i,0)\}$ satisfying \eqref{Wr}. Observe that after a finite time $t_0^{(\eta)}$ (see \eqref{t0}) these five strips are separated. Our a priori estimates are thus divided into two parts: before and after $t_0^{(\eta)}$.

\begin{itemize}
\item{\bf Estimates for $\mathbf{t\in[0,t_0^{(\eta)}]}$:} Before the separating time $t_0^{(\eta)}$, all the characteristic strips $\{\mathcal{R}_i\}$ overlap with each other. We define
\begin{align*}
  W(t):=\max_{i\in\{1,2,3,4,5,6,7 \}}\sup_{(x',t')\atop 0\leq t'\leq t}|w_i(x',t')|.
\end{align*}
In Section \ref{apes}, via a Gr\"{o}nwall-type argument, we firstly obtain
\begin{equation*}
  W(t)=O(W_0^{(\eta)})
\end{equation*}
with $W_0^{(\eta)}:=\max\limits_i\sup\limits_{z_i}|w_i(z_i,0)|$. Next, we introduce the upper bounds $\rho_i,v_i$ inside the strips
\begin{align*}
  S_i(t):=&\sup_{(z'_i,s'_i)\atop z'_i\in[\eta,2\eta],\ 0\leq s'_i\leq t}\rho_i(z'_i,s'_i),&S(t):=&\max_{i\in\{1,2,3,4,5,6,7 \}}S_i(t),\\
  J_i(t):=&\sup_{(z'_i,s'_i)\atop z'_i\in[\eta,2\eta]\ 0\leq s'_i\leq t}|v_i(z'_i,s'_i)|,&J(t):=&\max_{i\in\{1,2,3,4,5,6,7 \}}J_i(t).
\end{align*}
With $t_0^{(\eta)}$ explicitly given in \eqref{t0}, it is straight forward to estimate these quantities before $t_0^{(\eta)}$. Integrating the transport equations under the characteristic coordinates, we prove that when $t<t_0^{(\eta)}$
\begin{align*}
   S(t)=O(1),\quad\quad J(t)=&O(W_0^{(\eta)}).
\end{align*}

\item{\bf Estimates for $\mathbf{t\in[t_0^{(\eta)},T]}$:} For $t$ larger than $t_0^{(\eta)}$, $\mathcal{R}_2$ and $\mathcal{R}_3$ (or $\mathcal{R}_5$ and $\mathcal{R}_6$) is not completely separated due to the non-strict hyperbolicity. Nonetheless, the aforementioned five characteristic strips $\mathcal{R}_1,\mathcal{R}_2\cup\mathcal{R}_3,\mathcal{R}_4,\mathcal{R}_5\cup\mathcal{R}_6,\mathcal{R}_7$ are well separated. Adapted to this peculiarity, to replace $W$ in this part we introduce the following quantity
\begin{equation*}
\check{W}(t):=\max\Big\{\sup_{(x',t')\in\mathcal{R}_2\cup\mathcal{R}_3,\atop 0\leq t'\leq t}\{w_2,w_3\},\sup_{(x',t')\in\mathcal{R}_5\cup\mathcal{R}_6,\atop 0\leq t'\leq t}\{w_5,w_6\}\Big\}.
\end{equation*}
Furthermore, to give a thorough description of the dynamics after the separating time $t_0^{(\eta)}$, we need to estimate $w_i$ outside of the strips and we define
\begin{align*}
 &V_1(t):=  \sup_{(x',t')\notin\mathcal{R}_1,\atop 0\leq t'\leq t}|w_1(x',t')|,\quad  V_{\bar{2}}(t):=\sup_{(x',t')\notin\mathcal{R}_2\cup\mathcal{R}_3,\atop 0\leq t'\leq t}\{|w_2(x',t')|,|w_3(x',t')|\},\\
&V_4(t):=\sup_{(x',t')\notin\mathcal{R}_4,\atop 0\leq t'\leq t}|w_4(x',t')|,\quad V_{\bar{5}}(t):=\sup_{(x',t')\notin\mathcal{R}_5\cup\mathcal{R}_6,\atop 0\leq t'\leq t}\{|w_5(x',t')|,|w_6(x',t')|\},\\
&V_7(t):=\sup_{(x',t')\notin\mathcal{R}_7,\atop 0\leq t'\leq t}|w_7(x',t')|, \quad V(t):=\max_{i=1,\bar{2},4,\bar{5},7}V_i(t).
\end{align*}
Estimates of $V$ before $t_0^{(\eta)}$ can be similarly obtained as we just discussed
\begin{equation*}
   V(t)=O(\eta [W_0^{(\eta)}]^2).
\end{equation*}
After $t_0^{(\eta)}$, as we emphasized in Subsection \ref{difficulty}, we apply the bi-characteristic coordinates to bound certain quantities. Specifically, we use this technique to estimate $V_i$. In summary, for $t\in[t_0^{(\eta)},T]$, we obtain the following a priori estimates in the characteristic strips:
\begin{equation*}
\begin{split}
    S=&O(1+tVS+tJ+tS\check{W}),\\
     J=&O(W_0^{(\eta)}+t VJ+tV^2 S+t\check{W}J+tVS\check{W}),\\
    V=&O(\eta [W_0^{(\eta)}]^2+tV^2+\eta VJ+tV\check{W}+t(\check{W})^2), \\
    \check{W}=&O(\eta[W_0^{(\eta)}]^2+t\check{W}^2+tV\check{W}+tV^2).
    \end{split}
\end{equation*}

\end{itemize}

\textbf{Step 4: $L^\infty$-estimates and bootstrap argument (Subsection \ref{bootargu}).} We setup bootstrap assumptions in \eqref{y33}\eqref{wca}\eqref{Ja}. Based on the above estimates in Subsection \ref{apes}, via a bootstrap argument we obtain the desired $L^\infty$-estimates of the solution:
\begin{equation*}
J=O(W_0^{(\eta)}),\, S=O(1),\, V=O\Big(\eta [W_0^{(\eta)}]^2\Big),\,\check{W}=O(\eta[W_0^{(\eta)}]^2).
\end{equation*}
We depict the detailed behaviour of the solution up to the shock formation time. See Subsection \ref{bootargu} for the details.

\textbf{Step 5: Shock formation (Subsection \ref{pf1.2}).} We shall prove that a shock forms in $\mathcal{R}_1$ (Theorem \ref{shock}). With the estimates obtained in Section \ref{bootargu}, we derive the following sharp control of $\rho_1$:
\begin{equation*}
\rho_1 \sim 1-t W_0^{(\eta)}.
\end{equation*}
This implies that $\rho_1$ would become $0$ (a shock forms) at a finite time $T_{\eta}^* $ satisfying
\begin{equation*}
  T_{\eta}^*  \sim \frac{1}{W_0^{(\eta)}}.
\end{equation*}
Moreover, we deduce that $v_1$ has a positive lower bound $(1-\varepsilon) W_0^{(\eta)}$. Since $v_1=\rho_1 w_1$, it follows that $w_1$ will blow up at $T_{\eta}^*$.

\textbf{Step 6: $H^2$ ill-posedness (Section \ref{pfill}).} Furthermore, we prove the $H^2$ ill-posedness (Theorem \ref{3D}). We first choose a family of appropriate $H^2$-initial data \eqref{data}\eqref{data2} for $w_i^{(\eta)}$ in Subsection \ref{illdata}. For each $\eta$, in Subsection \ref{ill} we prove that the $H^1$-norm of the solution to \eqref{MHD} tends to infinity as $t \to T_\eta^*$. Here, we apply an upper bound estimate of $\partial_{z_1}\rho_1$ obtained in Subsection \ref{supbound}. The underlying reason for ill-posedness, embedded in our proof, is indeed the shock formation. This is the desired ill-posedness result for 3D ideal compressible MHD in $H^2(\mathbb{R}^3)$.

Finally, in Section \ref{appendix} and Section \ref{euler ill}, we prove Theorem \ref{shock}-\ref{3D} for the case when $H_1=0$ (we treat the case with non-trivial $H_1$ in all other sections) and for the 3D compressible Euler equations (i.e., $H=0$), respectively.

\subsection{Other related works}
In this subsection, we refer to related works in MHD. For 3D ideal compressible MHD, Ohno-Shirota \cite{ohno} proved that the characteristic fixed-boundary linearized problem satisfying perfect conducting wall condition is ill-posed in the standard Sobolev spaces $H^l$ with $l \geq 2$. We also refer to \cite{chen} by Chen-Young-Zhang for a conditional\footnote{In \cite{chen}, the authors considered the case of $H_1=0$ and imposed the condition that the fluid density $\varrho$ has global lower and upper bounds.} shock formation result for 1D compressible MHD. Another active topic for this system is the free-boundary problem. We refer to Chen-Wang \cite{chen-wang}, Lindblad-Zhang \cite{lindblad-zhang}, Trakhinin-Wang \cite{trakhinin1} for related results. The readers can also read Hu-Wang \cite{hu-wang}, Jiang-Ju-Li-Xin \cite{jiang}, Li-Li \cite{li} and Wang-Xin \cite{wang-xin} for other mathematical topics about compressible MHD.

For incompressible MHD, the global well-posedness of Cauchy problem has been studied by Cai-Lei \cite{cai}, Deng-Zhang \cite{deng2}, He-Xu-Yu \cite{he-xu-yu}, Hu-Lin \cite{hu}, Lin-Zhang \cite{lin2}, Wei-Zhang \cite{wei1,wei2,wei3}, Xu-Zhang \cite{Xu-Zhang}, Zhang \cite{Zhang}. Jeong-Oh proved ill-posedness results for the Hall- and electron-MHD in \cite{oh}. For free-boundary problems, we refer to Gu-Luo-Zhang \cite{gu}, Gu-Wang \cite{gu-wang}, Hao-Luo \cite{hao}, Sun-Wang-Zhang \cite{Sun1,Sun2}, Wang-Xin \cite{wang2-xin} and references therein.

\subsection{Acknowledgements}
The authors would like to thank Prof. Zhouping Xin for an enlightening discussion.

\section{Decomposition of waves} \label{pre}
In this section, we rewrite the MHD system with the wave decomposition method. We will compute the eigenvalues and eigenvectors of the coefficient matrix for the decomposed system. We then define the corresponding characteristics and introduce the characteristic (and bi-characteristic) coordinates. With them we will trace the dynamics of characteristic waves of different travelling speeds.

For the 3D ideal compressible MHD \eqref{MHD}, we collect the unknowns as a $8\times 1$ vector-valued function $\mathcal{U}$ with variables $x_1,x_2,x_3$ and $t$. We let
$$\mathcal{U}(x_1,x_2,x_3,t)=(u_1,u_2,u_3,\varrho, H_1,H_2,H_3,S)^T(x_1,x_2,x_3,t).$$
To prove ill-posedness for \eqref{MHD}, as in \cite{an,lindblad17,Lind96,Lind98}, we work under planar symmetry. And the solution satisfies $\mathcal{U}(x_1,x_2,x_3,t)=U(x_1,t).$ For notational simplicity, we denote $x_1$ by $x$ hereafter.

By the fifth equation of \eqref{MHD}, i.e., Gauss's Law, we obtain
\begin{equation*}
\nabla\cdot H=\frac{\partial H_1}{\partial x_1}+\frac{\partial H_2}{\partial x_2}+\frac{\partial H_3}{\partial x_3}=\frac{\partial H_1}{\partial x}=0, \quad \text{hence}\,\, H_1(x,t)=H_1(t).
\end{equation*}
And from the third equation of \eqref{MHD}, we further have $\frac{dH_1}{dt}=0.$ Thus, $H_1(x,t)$ equals to a constant $\kappa$ during the evolution. We require that the constant $\kappa$ is small in the following sense
\begin{equation} \label{h1}
\kappa^2 \ll \min\{\mu_0^{-1}A\gamma,1\}.
\end{equation}

We then collect the remaining unknowns $(u_1,u_2,u_3,\varrho,H_2,H_3,S)^T$ and they satisfy the following equations
\begin{equation}\label{ppMHD}
\left\{\begin{split}
&\varrho\partial_tu_1+\varrho u_1\partial_x u_1+c^2\partial_x\varrho+\partial_S p\partial_xS+\mu_0H_2\partial_xH_2+\mu_0H_3\partial_xH_3=0,\\
&\varrho\partial_tu_2+\varrho u_1\partial_x u_2-\mu_0H_1\partial_xH_2=0,\\
&\varrho\partial_tu_3+\varrho u_1\partial_x u_3-\mu_0H_1\partial_xH_3=0,\\
&\partial_t\varrho+u_1\partial_x\varrho+\varrho\partial_x u_1=0,\\
&\partial_tH_2+u_1\partial_x H_2-H_1\partial_xu_2+H_2\partial_xu_1=0,\\
&\partial_tH_3+u_1\partial_x H_3-H_1\partial_xu_3+H_3\partial_xu_1=0,\\
&\partial_t S+u_1\partial_xS=0.
\end{split}
\right.
\end{equation}
Here, $H_1=\kappa$ and $c=\sqrt{\partial_\varrho p}$ is the sound speed and the pressure $p$ satisfies the equation of state \eqref{state}.

We start to rewrite \eqref{ppMHD} as a $7\times 7$ hyperbolic system. Let $$\Phi:=(u_1,u_2,u_3,\varrho-1,H_2,H_3,S)^T(x,t).$$
From \eqref{ppMHD}, $\Phi$ obeys
\begin{equation}\label{pMHD}
\partial_t\Phi+A(\Phi)\partial_x\Phi=0
\end{equation}
with the coefficient matrix
\begin{equation} \label{matrix}
A(\Phi)=\begin{pmatrix}
u_1 & 0 & 0& c^2/\varrho & \mu_0H_2/\varrho & \mu_0H_3/\varrho & \partial_Sp/\varrho \\
0 & u_1 & 0& 0 & -\mu_0H_1/\varrho & 0 & 0 \\
0 & 0  & u_1 & 0 & 0 & -\mu_0H_1/\varrho& 0 \\
\varrho & 0 & 0& u_1 & 0 & 0 & 0 \\
H_2 & -H_1 & 0& 0 &u_1 & 0 & 0 \\
H_3& 0 & -H_1& 0 & 0 & u_1 & 0 \\
0 & 0 & 0& 0 & 0 & 0 &  u_1
\end{pmatrix}.
\end{equation}
The eigenvalues of \eqref{matrix} are
\begin{equation}\label{egvl}
\begin{split}
&\lambda_1=u_1+C_f,\quad \lambda_2=u_1+C_a,\quad \lambda_3=u_1+C_s,\quad \lambda_4=u_1,\\
&\lambda_5=u_1-C_s,\quad \lambda_6=u_1-C_a,\quad \lambda_7=u_1-C_f,
\end{split}
\end{equation}
where
\begin{equation} \label{charac speeds}
\begin{split}
&C_f=\Big\{\frac{\mu_0}{2\varrho}(H_1^2+H_2^2+H_3^2)+\frac{c^2}{2}+\frac12\sqrt{[\frac{\mu_0}{\varrho}(H_1^2+H_2^2+H_3^2)+c^2]^2-\frac{4\mu_0}{\varrho}H_1^2c^2}\Big\}^{1/2},\\
&C_s=\Big\{\frac{\mu_0}{2\varrho}(H_1^2+H_2^2+H_3^2)+\frac{c^2}{2}-\frac12\sqrt{[\frac{\mu_0}{\varrho}(H_1^2+H_2^2+H_3^2)+c^2]^2-\frac{4\mu_0}{\varrho}H_1^2c^2}\Big\}^{1/2},\\
&C_a=\sqrt{\frac{\mu_0}{\varrho}}H_1
\end{split}
\end{equation}
are called the fast, the slow and the Alfv\'en wave speeds, respectively. Here, the characteristic speed $\lambda_4$ corresponds to the propagation speed of the entropy disturbance that just follows the stream-line of the flow. The pair $\lambda_2$ and $\lambda_6$ correspond respectively to the speeds of the Alfv\'en disturbances moving forward and backward along the flow. And $\lambda_3,\lambda_5$ are speeds of forward and backward slow-magneto-acoustic disturbances, whereas $\lambda_1,\lambda_7$ constitute speeds of forward and backward fast-magneto-acoustic disturbances. Noting that $H_1$ satisfies \eqref{h1}, when $|\Phi|$ is small enough, one can verify that
\begin{equation} \label{3speeds}
C_s\leq C_a< C_f.
\end{equation}
The bulk of the paper is devoted to studying the case of non-trivial $\kappa$. The case of $\kappa=0$ will be studied in Section \ref{appendix}. With non-trivial $\kappa$, by \eqref{3speeds}, the corresponding eigenvalues obtained in \eqref{egvl} satisfy:
\begin{equation}\label{negv}
\begin{split}
\lambda_7<\lambda_6\leq\lambda_5<\lambda_4<\lambda_3\leq\lambda_2<\lambda_1.
\end{split}
\end{equation}
In \eqref{negv}, the equalities hold if and only if $H_2^2+H_3^2=0$. Since $\lambda_2,\lambda_3$ (or $\lambda_5,\lambda_6$) could take the same value, the system \eqref{pMHD} is \underline{not} strictly hyperbolic.

A very crucial step is to take proper right eigenvectors of $A(\Phi)$. By carefully designing, we take the right eigenvectors of $A(\Phi)$ as follows:
\begin{equation}\label{regv}
\begin{split}
&r_1=\left(\begin{array}{cc}\frac{C_a^2}{C_f}-C_f\\
\frac{C_a^2 H_2}{C_f H_1 }\\
\frac{C_a^2 H_3}{C_f H_1 }\\
\varrho(\frac{C_a^2}{C_f^2}-1)\\
-H_2\\
-H_3\\
0\end{array}\right),
\ r_2=\left(\begin{array}{cc} 0\\ \frac{C_a H_3}{H_1 }\\ -\frac{C_a H_2}{H_1}\\0\\-H_3\\H_2\\0\end{array}\right),
\ r_3=\left(\begin{array}{cc}\frac{C_a^2}{C_s}-C_s\\
\frac{C_a^2 H_2}{C_s H_1 }\\
\frac{C_a^2 H_3}{C_s H_1 }\\
\varrho(\frac{C_a^2}{C_s^2}-1)\\
-H_2\\
-H_3\\
0\end{array}\right),
\ r_4=\left(\begin{array}{cc} 0\\ 0\\ 0\\-\frac{\varrho}{\gamma}\\0\\0\\1\end{array}\right), \\
& r_5=\left(\begin{array}{cc}\frac{C_a^2}{C_s}-C_s\\
\frac{C_a^2 H_2}{C_s H_1 }\\
\frac{C_a^2 H_3}{C_s H_1 }\\
-\varrho(\frac{C_a^2}{C_s^2}-1)\\
H_2\\
H_3\\
0\end{array}\right),
\ r_6=\left(\begin{array}{cc} 0\\ \frac{C_a H_3}{H_1}\\ -\frac{C_a H_2}{H_1 }\\0\\H_3\\-H_2\\0\end{array}\right),
\ r_7=\left(\begin{array}{cc}\frac{C_a^2}{C_f}-C_f\\
\frac{C_a^2 H_2}{C_f H_1 }\\
\frac{C_a^2 H_3}{C_f H_1 }\\
-\varrho(\frac{C_a^2}{C_f^2}-1)\\
H_2\\
H_3\\
0\end{array}\right).
\end{split}
\end{equation}
As emphasized in Subsection \ref{difficulty}, this selection of right eigenvectors is a key point to our proof. The corresponding left eigenvectors $l_i$ are set dual to the right ones, i.e.,
\begin{equation}\label{lr}
l_ir_j=\delta_{ij}\quad \text{for}\ i,j=1,\cdots,7.
\end{equation}
Here $\delta_{ij}$ is the Kronecker symbol. Specifically, we list these $l_i$:
\begin{equation}\label{legv}
\begin{split}
l_1=&\bigg(-\frac{C_f}{2(C_f^2-C_s^2)},\frac{H_1H_2C_f(C_a^2-C_s^2)}{2(\text{\dashuline{$H_2^2+H_3^2$}})C_a^2(C_f^2-C_s^2)},\frac{H_1H_3C_f(C_a^2-C_s^2)}{2(\text{\dashuline{$H_2^2+H_3^2$}})C_a^2(C_f^2-C_s^2)},-\frac{C_f^2C_s^2}{2\varrho C_a^2(C_f^2-C_s^2)},\\
&-\frac{H_2C_f^2(C_a^2-C_s^2)}{2(\text{\dashuline{$H_2^2+H_3^2$}})C_a^2 (C_f^2-C_s^2)},-\frac{H_3C_f^2(C_a^2-C_s^2)}{2(\text{\dashuline{$H_2^2+H_3^2$}})C_a^2 (C_f^2-C_s^2)},-\frac{C_f^2C_s^2}{2\gamma C_a^2 (C_f^2-C_s^2)}\bigg),\\
l_2=&\bigg(0,\frac{H_1H_3}{2C_a(\text{\dashuline{$H_2^2+H_3^2$}})},-\frac{H_1 H_2}{2C_a(\text{\dashuline{$H_2^2+H_3^2$}})},0,-\frac{H_3}{2(\text{\dashuline{$H_2^2+H_3^2$}})},\frac{H_2}{2(\text{\dashuline{$H_2^2+H_3^2$}})},0\bigg),\\
l_3=&\bigg(\frac{C_s}{2(C_f^2-C_s^2)},-\frac{H_1H_2C_s(C_a^2-C_f^2)}{2(\text{\dashuline{$H_2^2+H_3^2$}})C_a^2(C_f^2-C_s^2)},-\frac{H_1H_3C_s(C_a^2-C_f^2)}{2(\text{\dashuline{$H_2^2+H_3^2$}})C_a^2(C_f^2-C_s^2)},\frac{C_f^2C_s^2}{2\varrho C_a^2(C_f^2-C_s^2)},\\
&\frac{H_2C_s^2(C_a^2-C_f^2)}{2(\text{\dashuline{$H_2^2+H_3^2$}})C_a^2 (C_f^2-C_s^2)},\frac{H_3C_s^2(C_a^2-C_f^2)}{2(\text{\dashuline{$H_2^2+H_3^2$}})C_a^2 (C_f^2-C_s^2)},\frac{C_f^2C_s^2}{2\gamma C_a^2 (C_f^2-C_s^2)}\bigg),\\
l_4=&\big(0,0,0,0,0,0,1\big),\\
l_5=&\bigg(\frac{C_s}{2(C_f^2-C_s^2)},-\frac{H_1H_2C_s(C_a^2-C_f^2)}{2(\text{\dashuline{$H_2^2+H_3^2$}})C_a^2(C_f^2-C_s^2)},-\frac{H_1H_3C_s(C_a^2-C_f^2)}{2(\text{\dashuline{$H_2^2+H_3^2$}})C_a^2(C_f^2-C_s^2)},-\frac{C_f^2C_s^2}{2\varrho C_a^2(C_f^2-C_s^2)},\\
&-\frac{H_2C_s^2(C_a^2-C_f^2)}{2(\text{\dashuline{$H_2^2+H_3^2$}})C_a^2 (C_f^2-C_s^2)},-\frac{H_3C_s^2(C_a^2-C_f^2)}{2(\text{\dashuline{$H_2^2+H_3^2$}})C_a^2 (C_f^2-C_s^2)},-\frac{C_f^2C_s^2}{2\gamma C_a^2 (C_f^2-C_s^2)}\bigg),\\
l_6=&\bigg(0,\frac{H_1H_3}{2C_a(\text{\dashuline{$H_2^2+H_3^2$}})},-\frac{H_1H_2}{2C_a(\text{\dashuline{$H_2^2+H_3^2$}})},0,\frac{H_3}{2(\text{\dashuline{$H_2^2+H_3^2$}})},-\frac{H_2}{2(\text{\dashuline{$H_2^2+H_3^2$}})},0\bigg),\\
l_7=&\bigg(-\frac{C_f}{2(C_f^2-C_s^2)},\frac{H_1H_2C_f(C_a^2-C_s^2)}{2(\text{\dashuline{$H_2^2+H_3^2$}})C_a^2(C_f^2-C_s^2)},\frac{H_1H_3C_f(C_a^2-C_s^2)}{2(\text{\dashuline{$H_2^2+H_3^2$}})C_a^2(C_f^2-C_s^2)},\frac{C_f^2C_s^2}{2\varrho C_a^2(C_f^2-C_s^2)},\\
&\frac{H_2C_f^2(C_a^2-C_s^2)}{2(\text{\dashuline{$H_2^2+H_3^2$}})C_a^2 (C_f^2-C_s^2)},\frac{H_3C_f^2(C_a^2-C_s^2)}{2(\text{\dashuline{$H_2^2+H_3^2$}})C_a^2 (C_f^2-C_s^2)},\frac{C_f^2C_s^2}{2\gamma C_a^2 (C_f^2-C_s^2)}\bigg)
\end{split}
\end{equation}
with $C_f,C_s,C_a$ defined in \eqref{charac speeds}.

\begin{remark}
Note that there is the singular factor $\frac{1}{H_2^2+H_3^2}$ present in the $2^{\text{nd}}$,$3^{\text{rd}}$,$5^{\text{th}}$,$6^{\text{th}}$ components of all $l_j$ with $j \neq 4$. Nevertheless, with our carefully designed right eigenvectors \eqref{regv}, all these $\frac{1}{H_2^2+H_3^2}$ factors are cancelled when calculating the coefficients $\gamma^i_{km}$ in Theorem \ref{structurethm}. The details of these cancellations will be demonstrated in Section \ref{structure} and Appendix \ref{appendixB}.
\end{remark}

We now introduce the characteristics as well as the associated characteristic and bi-characteristic coordinates. We define {\bf\textit{the $i^{\text{th}}$ characteristic $\mathcal{C}_i(z_i)$}}, which originates from $z_i$, to be the image $(X_i(z_i,t),t)$ of solutions to the following ODE:
\begin{align}\label{flow}
  \left\{\begin{array}{ll}
  \frac{\partial}{\partial t}X_i(z_i,t)=\lambda_i\big(\Phi(X_i(z_i,t),t)\big),\quad t\in[0,T],\\
  X_i(z_i,0)=z_i.
  \end{array}\right.
\end{align}
For any $(x,t)\in \mathbb{R}\times [0,T]$, there is a unique $(z_i,s_i)\in \mathbb{R}\times [0,T]$ such that $(x,t)=\big(X_i(z_i,s_i),s_i\big)$ with $X_i$ satisfying \eqref{flow}. We define this $(z_i,s_i)$ to be {\bf\textit{the characteristic coordinates}} of $(x,t)$. Moreover, we define bi-characteristic coordinates: consider two transversal characteristics $\mathcal{C}_i(z_i)$ and $\mathcal{C}_j(z_j)$ with $i\neq j$. Note that the characteristic $\mathcal{C}_i(z_i)$ can be parameterized by $z_j$. Thus given $(z_i,z_j)$ one can uniquely locate a point where the two characteristics $\mathcal{C}_i(z_i)$ and $\mathcal{C}_j(z_j)$ intersect. To distinguish with the characteristic coordinates above, we denote $(z_{i},z_{j})$ by $(y_{i},y_{j})$ in this new coordinate system. Then, given $(x,t)\in \mathbb{R}\times [0,T]$, {\bf\textit{the bi-characteristic coordinates}} of $(x,t)$ is defined to be the unique $(y_i,y_j)\in\mathbb{R}^2$ such that with $t=t'(y_i, y_j)$, we have
\begin{equation}\label{bichar}
  (x,t)=\big(X_i(y_i,t'(y_i,y_j)),t'(y_i,y_j)\big)=\big(X_j(y_j,t'(y_i,y_j)),t'(y_i,y_j)\big).
\end{equation}

To describe the compression among the $i^{\text{th}}$ characteristics, we introduce a geometric quantity $\rho_i$ (the inverse density of the $i^{\text{th}}$ characteristics)
\begin{equation}\label{dense}
\rho_i:=\partial_{z_i}X_i.
\end{equation}
Following from \eqref{flow}, for the initial data we have
\begin{equation}\label{319}
  \rho_i(z_i,0)=1,
\end{equation}
The coordinate transformations comply with the following rules:
\begin{equation}
  \partial_{z_i}=\rho_i\partial_x,\quad \partial_{s_i}=\lambda_i\partial_x+\partial_t,
\quad
  \partial_{y_i}t'=\frac{\rho_i}{\lambda_j-\lambda_i},\quad\partial_{y_j}t'=\frac{\rho_j}{\lambda_i-\lambda_j},
\end{equation}
\begin{equation}
  \partial_{y_i}=\frac{\rho_i}{\lambda_j-\lambda_i}\partial_{s_j}=\partial_{z_i}+\frac{\rho_i}{\lambda_j-\lambda_i}\partial_{s_i},
\end{equation}
and
\begin{equation}
  dx=\rho_idz_i+\lambda_ids_i,\quad dt=ds_i,\quad   dz_i=dy_i,\quad dz_j=dy_j,
\end{equation}
\begin{equation}
  dx=\frac{\rho_i\lambda_j}{\lambda_j-\lambda_i}dy_i+\frac{\rho_j\lambda_i}{\lambda_i-\lambda_j}dy_j,\quad dt=\frac{\rho_i}{\lambda_j-\lambda_i}dy_i+\frac{\rho_j}{\lambda_i-\lambda_j}dy_j.
\end{equation}

Now we are ready to decompose system \eqref{pMHD} via the decomposition of waves. See \cite{an,christodoulou,john74} for more applications of this method. For fixed $i$, setting
\begin{equation}\label{wi}
  w_i:=l_i\partial_x\Phi,
  \end{equation}
  and
  \begin{equation}
    v_i:=\rho_iw_i,
\end{equation}
then by \eqref{lr}, we have
\begin{equation}\label{phix}
  \partial_x\Phi=\sum_{k=1}^7 w_kr_k.
\end{equation}
The above equation serves as a \underline{decomposition} of $\partial_x \Phi$. By calculations similar to John \cite{john74}, one can verify that $\{\rho_i,w_i,v_i\}_{i=1,\cdots,7}$ satisfy the following transport equations
\begin{align}
  \partial_{s_i}\rho_i=&c_{ii}^iv_i+\Big(\sum_{m\neq i}c_{im}^iw_m\Big)\rho_i,\label{eqrho}\\
  \partial_{s_i}w_i=&-c_{ii}^iw_i^2+\Big(\sum_{m\neq i}(-c_{im}^i+\gamma_{im}^i)w_m\Big)w_i+\sum_{m\neq i,k\neq i\atop m\neq k}\gamma_{km}^iw_kw_m,\label{eqw}\\
  \partial_{s_i}v_i=&\Big(\sum_{m\neq i}\gamma_{im}^iw_m\Big)v_i+\sum_{m\neq i,k\neq i\atop m\neq k}\gamma_{km}^iw_kw_m\rho_i\label{eqv}
\end{align}
where $\partial_{s_i}=\lambda_i\partial_x+\partial_t$ and
\begin{align}
&c_{im}^i=\nabla_\Phi\lambda_i\cdot r_m,\label{coec}\\
  &\gamma_{im}^i=-(\lambda_i-\lambda_m)l_i \cdot(\nabla_\Phi r_i \cdot r_m-\nabla_\Phi r_m \cdot r_i),\quad m\neq i,\label{coeg1}\\
  &\gamma_{km}^i=-(\lambda_k-\lambda_m)l_i \cdot (\nabla_\Phi r_k \cdot r_m), \qquad\qquad\qquad k\neq i,\  m\neq i.\label{coeg2}
\end{align}
\begin{remark}
We call the $i^{\text{th}}$ characteristic {\bf \textit{genuinely nonlinear}} if $c_{ii}^i \neq 0$ for all $(x,t)$ and on the contrary {\bf \textit{linearly degenerate}} when $c_{ii}^i\equiv 0$.
\end{remark}
\begin{remark}
As will be proved in Theorem \ref{structurethm}, all the coefficients $c^i_{im}$, $\gamma^i_{im}$, $\gamma^i_{km}$ of the decomposed system \eqref{eqrho}-\eqref{eqv} are uniformly bounded of order $1$. The proof will be given in Section \ref{structure} and Appendix \ref{appendixB}. Moreover, important schematic structures of 3D compressible MHD are also exhibited via system \eqref{eqrho}-\eqref{eqv}.
\end{remark}

\section{Proof of Theorem \ref{structurethm}} \label{structure}

In this section, we explore the structures of the decomposed system \eqref{eqrho}-\eqref{eqv} and aim to prove Theorem \ref{structurethm}. These structures play a crucial role in proving the shock formation of \eqref{pMHD}.

We start from analyzing $w_1$. Our ultimate goal is to show that $w_1$ blows up in a finite time. Observe that in \eqref{eqw} the equation of $w_1$ is of Ricatti-type. In particular, with the following lemma we show that the $1^{\text{st}}$ characteristic is genuinely nonlinear.
\begin{lem}[{\bf Genuine nonlinearity of $\mathcal{C}_1$}] \label{gn}
For $|H_1| \ll 1$ and $|\Phi|\leq 2\delta \ll 1$, there holds $c_{11}^1(\Phi)<0$.
\end{lem}
\begin{proof}
Plugging the formula of $r_1$ \eqref{regv}, $\lambda_1$ \eqref{egvl}, the fast and the Alfv\'en wave speeds $C_f$, $C_a$ \eqref{charac speeds} into the definition of $c^1_{11}(\Phi)$ \eqref{coec}, one can verify that $c^1_{11}(\Phi)$ is continuous with respect to variable $\Phi$ and parameter $H_1$. Since $|H_1| \ll 1$ and $|\Phi|\leq 2\delta \ll 1$, it then suffices to show that $c_{11}^1(0)<0$ for the case of $H_1=0$. With $H_1=0$, it holds that
\begin{equation}
\lambda_1=u_1+\big\{\frac{\mu_0}{\varrho}(H_2^2+H_3^2)+A\gamma\varrho^{\gamma-1}\exp S\big\}^{1/2}.
\end{equation}
We hence obtain
\begin{equation}
\begin{split}
c^1_{11(0)}(0)=&\nabla_\Phi\lambda_1(0)\cdot r_1(0)\\
=&(1,0,0,\frac{\gamma-1}{2}\sqrt{A\gamma},0,0,\frac{\sqrt{A\gamma}}{2})\cdot (-\sqrt{A\gamma},0,0,-1,0,0,0)^T\\
=&-\frac{\sqrt{A\gamma}(\gamma+1)}{2}<0.
\end{split}
\end{equation}
Using continuity, we also conclude that $c_{11}^1(\Phi)<0$.
\end{proof}
Note that $c_{11}^1(\Phi)<0$ is important for the Riccati-type equation \eqref{eqw} of $w_1$ to develop a potential singularity.

For $w_2,w_6$, the following lemma indicates the vanishing of the Riccati-type terms.
\begin{lem} \label{vanishing}
For $c_{i2}^i$ and $c_{i6}^i$ defined in \eqref{coec}, it holds that
\begin{equation}
c_{i2}^i=0,\quad c_{i6}^i=0\quad \text{for}\quad i=1,\cdots,7.
\end{equation}
\end{lem}
\begin{proof}
Recall the definition \eqref{coec}
$$c_{i2}^i=\nabla_{\Phi} \lambda_i\cdot r_2.$$
Invoking the right eigenvector $r_2$ in Section \ref{pre}, we have
\begin{equation} \label{cii2}
c_{i2}^i=\sqrt{\frac{\mu_0}{\varrho}}\frac{H_3}{H_2}\partial_{u_2}\lambda_i-\sqrt{\frac{\mu_0}{\varrho}}\partial_{u_3}\lambda_i-\frac{H_3}{H_2}\partial_{H_2}\lambda_i+\partial_{H_3}\lambda_i.
\end{equation}
Since the eigenvalues $\lambda_i$ are independent of $u_2$ and $u_3$, it holds that
\begin{equation}
c_{i2}^i=-\frac{H_3}{H_2}\partial_{H_2}\lambda_i+\partial_{H_3}\lambda_i.
\end{equation}
Observe that $\lambda_i$ in \eqref{egvl} can be written as
$$\lambda_i=f_i(u_1,\rho,H_2^2+H_3^2,S).$$
Substituting it back in \eqref{cii2}, we then prove that the coefficients $c_{i2}^i$ satisfy
\begin{equation}
c_{i2}^i=-\frac{H_3}{H_2}\partial_{H_2^2+H_3^2} f_i\cdot 2H_2+\partial_{H_2^2+H_3^2} f_i\cdot2H_3=0.
\end{equation}
In the same fashion, we deduce $c_{i6}^i=0$.
\end{proof}
The above lemma immediately implies that $c^2_{22}=0$ and $c^6_{66}=0$. Moreover, by a straightforward calculation, one can easily see that $c^4_{44}=0$. Hence the $2^{\text{nd}},4^{\text{th}},6^{\text{th}}$ characteristics are linearly degenerate.
\begin{remark}
For $c^i_{ik}$, $\gamma^i_{im}$, $\gamma^i_{km}$ with other indices, we need to prove that they are $O(1)$. This is achieved via conducting a heavy calculation of checking all the scenarios and exploring the delicate cancellations among them. A 14-page detailed calculation is attached in Appendix \ref{appendixB}.
\end{remark}
The main proposition in this section is as follows.
\begin{prop}[{\bf Estimates of coefficients}] \label{coeff}
For $|\Phi|\leq 2\delta\ll 1$, the decomposed system \eqref{eqrho}-\eqref{eqv} admits uniformly bounded non-trivial coefficients $c^i_{ii}$, $c^i_{im}$, $\gamma^i_{im}$, $\gamma^i_{km}$, i.e.,
\begin{equation}
\max\limits_{i,k,m}\sup\limits_{|\Phi|\leq 2\delta}\big\{|c^i_{ii}|,|c^i_{im}|,|\gamma^i_{im}|,|\gamma^i_{km}|\big\}\leq C,
\end{equation}
where $C$ is a uniform constant.
\end{prop}
\begin{proof}
Let $\Gamma(\Phi)$ be the maximum of all the coefficients $\{|c_{im}^i|\}$ and $\{|\gamma_{km}^i|\}$ in \eqref{coec}-\eqref{coeg2}. With the eigenvalues and eigenvectors in Section \ref{pre}, we will show
\begin{equation}\label{order1}
  \Gamma(\Phi)=O(1) \qquad \text{for} \qquad |\Phi|\leq 2\delta\ll 1.
\end{equation}
This is clearly true for $c^i_{im}$, $\gamma^4_{4m}$ and $\gamma^4_{km}$ as $\nabla_\Phi \lambda_i$, $r_i$ and $l_4$ contain only $O(1)$ terms. But for $\gamma^i_{im},\gamma^i_{km}$ with $i \neq 4$, they might be singular since $\frac{1}{H_2^2+H_3^2}$ appears in the corresponding left eigenvectors. Then these potentially singular coefficients could destroy all the arguments we construct. We must calculate \underline{all} these coefficients concretely to eliminate this danger. After a thorough 14-page calculation, we find that all the potentially singular denominators are cancelled in the matrix multiplication. Whence, \eqref{order1} holds for all the coefficients (See Appendix \ref{appendixB}).

We begin with a concrete calculations for $\gamma_{64}^2$. With
\begin{equation*}
\begin{split}
&r_2=\left(\begin{array}{cc} 0\\ \frac{C_a H_3}{H_1}\\ -\frac{C_a H_2}{H_1}\\0\\-H_3\\H_2\\0\end{array}\right),\ r_4=\left(\begin{array}{cc} 0\\ 0\\ 0\\-\frac{\varrho}{\gamma}\\0\\0\\1\end{array}\right),\ r_6=\left(\begin{array}{cc} 0\\ \frac{C_a H_3}{H_1}\\ -\frac{C_a H_2}{H_1}\\0\\H_3\\-H_2\\0\end{array}\right)\,\,\text{and}
\end{split}
\end{equation*}
\begin{equation*}
l_2=\Big(0,\frac{H_1H_3}{2C_a(\text{\dashuline{$H_2^2+H_3^2$}})},-\frac{H_2H_1}{2C_a(\text{\dashuline{$H_2^2+H_3^2$}})},0,-\frac{H_3}{2(\text{\dashuline{$H_2^2+H_3^2$}})},\frac{H_2}{2(\text{\dashuline{$H_2^2+H_3^2$}})},0\Big),\\
\end{equation*}
we have
\begin{equation*}
  \nabla_\Phi r_6 \cdot r_4=\Big(0,\frac{C_a H_3}{2\gamma H_1},-\frac{C_a H_2}{2\gamma H_1},0,0,0,0\Big)^\top,
\end{equation*}
and hence
\begin{eqnarray*}
\gamma_{64}^2&=&-(\lambda_6-\lambda_4)l^2 \cdot (\nabla_\Phi r_6 \cdot r_4)\\
&=&-(\lambda_6-\lambda_4)\Big(\frac{H_3^2}{4\gamma(H_2^2+H_3^2)}+\frac{H_2^2}{4\gamma(H_2^2+H_3^2)}\Big)\\
&\overset{(*)}{=}&-\frac{\lambda_6-\lambda_4}{4\gamma}=O(1).
\end{eqnarray*}
Note the specific cancellation in equality $(*)$.

We use $\gamma^2_{26}$ to demonstrate another type of cancellation. Proceed as above and it follows
\begin{equation*}
  \nabla_\Phi r_2 \cdot r_6-\nabla_\Phi r_6 \cdot r_2=\Big(0,-\frac{2C_a H_2}{H_1},-\frac{2C_a H_3}{H_1},0,0,0,0\Big)^\top.
\end{equation*}
Thus
\begin{eqnarray*}
\gamma_{26}^2&=&-(\lambda_2-\lambda_6)l^2 \cdot (\nabla_\Phi r_2 \cdot r_6-\nabla_\Phi r_6 \cdot r_2)\\
&=&(\lambda_2-\lambda_6)\Big(\frac{H_3 H_2}{H_2^2+H_3^2}-\frac{H_2 H_3}{H_2^2+H_3^2}\Big){\overset{(**)}{=}}0.
\end{eqnarray*}
Note the crucial cancellation in equality $(**)$. The computation of other coefficients, i.e., $\gamma^i_{im},\gamma^i_{km}$ with $i \neq 4$, can be conducted similarly but with more involved expressions. Yet for any $\gamma^i_{km}$, the \underline{cancellation} of singular factor $\frac{1}{H_2^2+H_3^2}$ is either in the form of $(*)$ or in the form of $(**)$. The detailed calculations can be found in a 14-page Appendix \ref{appendixB}. These estimates critically reveal the subtle structures of our system.
\end{proof}

Now with all the uniformly bounded coefficients at our disposal, we obtain the structures of the system as listed in \eqref{infmrho}-\eqref{infmv}. This concludes the proof of Theorem \ref{structurethm}.

\section{Proof of shock formation (Theorem \ref{shock})} \label{pfthm1.2}
We prove shock formation for the 3D compressible MHD in this section. Firstly, we deduce proper a priori estimates (Subsection \ref{apes}). These delicate estimates act as a powerful tool to cope with the strong nonlinear interaction terms. Next, employing the a priori estimates, we set up and close a bootstrap argument (Subsection \ref{bootargu}). Finally, the shock formation follows from the $L^\infty$-estimates obtained via the bootstrap argument (Subsection \ref{pf1.2}).

\subsection{A priori estimates}\label{apes}
In this subsection, we derive the main a priori estimates. We firstly define the characteristic strips and a separating time $t_0^{(\eta)}$. Beyond $t_0^{(\eta)}$ the strips under consideration will be completely separated. The a priori estimates will also be divided into two parts: before and after $t_0^{(\eta)}$. These estimates provide a manifestation of how to control the nonlinear interactions of waves with different or almost-the-same characteristic speeds.

Let $\mathcal{C}_i(z_i)$ be the $i^{\text{th}}$ characteristic issuing from $z_i$ with propagation speed $\lambda_i$. Define the corresponding $i^{\text{th}}$ characteristic strip
\begin{equation} \label{strip}
\mathcal{R}_i:=\cup_{z_i\in I_0}\mathcal{C}_i(z_i),
\end{equation}
where $I_0=\{x:\eta\leq x\leq 2\eta\}$ is the compact support of the initial data and $\eta$ is a chosen positive parameter. We will require $\eta$ to be small to drive the ill-posedness. See Section \ref{illdata}-\ref{ill}. Yet for the purpose of shock formation in this section, $\eta$ is \underline{not} necessary to be small.

For the planar symmetric 3D ideal MHD system \eqref{pMHD}, two pairs of characteristics are of almost the same travelling speed. This is because $H_2^2+H_3^2$ is sufficiently small, and from \eqref{egvl}\eqref{charac speeds} it follows
      \begin{equation*}
   \lambda_2\thickapprox\lambda_3,\ \lambda_5\thickapprox\lambda_6.
      \end{equation*}
 This is illustrated by the long-time overlapping of the corresponding characteristic strips $\mathcal{R}_2,\mathcal{R}_3$ or $\mathcal{R}_5,\mathcal{R}_6$ in the picture below. Our strategy is to take unions of these overlapped strips and consider the propagation of waves within the following five characteristic strips:
 \begin{equation}\label{5strip}
\{\mathcal{R}_1,\mathcal{R}_{\bar{2}},\mathcal{R}_4,\mathcal{R}_{\bar{5}},\mathcal{R}_7\},
  \end{equation}
where $\mathcal{R}_{\bar{2}}:=\mathcal{R}_2\bigcup\mathcal{R}_3$ and $\mathcal{R}_{\bar{5}}:=\mathcal{R}_5\bigcup\mathcal{R}_6$. These five strips will be completely separated after a finite separating time $t_0^{(\eta)}$.
\begin{figure}[H]
\centering
\begin{tikzpicture}[fill opacity=0.5, draw opacity=1, text opacity=1]
\node [below]at(3.5,0){$2\eta$};
\node [below]at(2.3,0){$\eta$};

\filldraw[white, fill=gray!40] (3.5,0)..controls (2,1) and (1,2)..(-0.8,3.5)--(0.4,3.5)..controls (1,2.8) and (1.9,2)..(2.5,0);
\filldraw[white, fill=gray!40](3.5,0)..controls (3.2,1) and (2,2.6)..(1.3,3.5)--(3.5,3.5)..controls (2.2,1.5) and (2.6,1)..(2.5,0);
\filldraw[white, fill=gray!40](3.5,0)..controls (3.4,1) and (3,1.5)..(4.5,3.5)--(5.5,3.5)..controls (3.2,1.5) and (2.8,1)..(2.5,0);
\filldraw[white, fill=gray!40](3.5,0)..controls (3.8,1) and (4,1.5)..(6.6,3.5)--(8.5,3.5)..controls (3.9,1.5) and (3.5,1)..(2.5,0);

\filldraw[white, fill=gray!80] (2.5,0)..controls (3.5,1) and (3.9,1.5)..(8.5,3.5)--(10,3.5)..controls (5.5,1.5) and (4,0.5)..(3.5,0)--(2.5,0);
\filldraw[white,fill=gray!80](2.5,0)..controls (2.8,1) and (3.2,1.5)..(5.5,3.5)--(6.6,3.5)..controls (4,1.5) and (3.8,1)..(3.5,0);
\filldraw[white,fill=gray!80](2.5,0)..controls (2.6,1) and (2.2,1.5)..(3.5,3.5)--(4.5,3.5)..controls (3,1.5) and (3.4,1)..(3.5,0);
\filldraw[white,fill=gray!80](2.5,0)..controls (1.9,2) and (1,2.8)..(0.4,3.5)--(1.3,3.5)..controls (2,2.6) and (3.2,1)..(3.5,0);
\filldraw[white,fill=gray!80](2.5,0)..controls (1,1) and (0.5,1.8)..(-1.8,3.5)--(-0.8,3.5)..controls (1,2) and (2,1)..(3.5,0);
\draw[->](-1.8,0)--(10,0)node[left,below]{$t=0$};
\draw[dashed](-1.8,1.7)--(10,1.7)node[right,below]{$t=t_0^{(\eta)}$};

\draw (3.5,0)..controls (4,0.5) and (5.5,1.5)..(10,3.5);

\draw (2.5,0)..controls (3.5,1) and (3.9,1.5)..(8.5,3.5);
\node [below] at(7.6,3){$\mathcal{R}_1$};

\draw [color=black](3.5,0)..controls (3.8,1) and (4,1.5)..(6.6,3.5);

\draw [color=black](2.5,0)..controls (2.8,1) and (3.2,1.5)..(5.5,3.5);
\node [below] at(5.2,3){$\mathcal{R}_2\cup\mathcal{R}_3$};

\draw [color=black](3.5,0)..controls (3.4,1) and (3,1.5)..(4.5,3.5);

\draw [color=black](2.5,0)..controls (2.6,1) and (2.2,1.5)..(3.5,3.5);
\node [below] at(3.5,3){$\mathcal{R}_4$};

\draw [color=black](3.5,0)..controls (3.2,1) and (2,2.6)..(1.3,3.5);
\node [below] at(1.6,3){$\mathcal{R}_5\cup\mathcal{R}_6$};

\draw [color=black] (2.5,0)..controls (1.9,2) and (1,2.8)..(0.4,3.5);

\draw [color=black](3.5,0)..controls (2,1) and (1,2)..(-0.8,3.5);
\node [below] at(-0.2,3){$\mathcal{R}_{7}$};

\draw [color=black] (2.5,0)..controls (1,1) and (0.5,1.8)..(-1.8,3.5);
\end{tikzpicture}
\end{figure}
Now via estimates we first determine the separating time $t_0^{(\eta)}$. Take the supremum and infimum of the eigenvalues
\begin{align*}
  &\bar{\lambda}_i:=\sup_{\Phi\in B_{2\delta}^7 (0)}\lambda_i(\Phi),\quad \underline{\lambda}_i:=\inf_{\Phi\in B_{2\delta}^7(0)}\lambda_i(\Phi),\quad \text{for}\quad i=1,4,7,\\
&\bar{\lambda}_{\bar{2}}:=\sup_{\Phi\in B_{2\delta}^7(0)}\{\lambda_2(\Phi),\lambda_3(\Phi)\},\quad \underline{\lambda}_{\bar{2}}:=\inf_{\Phi\in B_{2\delta}^7(0)}\{\lambda_2(\Phi),\lambda_3(\Phi)\},\\
&\bar{\lambda}_{\bar{5}}:=\sup_{\Phi\in B_{2\delta}^7(0)}\{\lambda_5(\Phi),\lambda_6(\Phi)\},\quad \underline{\lambda}_{\bar{5}}:=\inf_{\Phi\in B_{2\delta}^7(0)}\{\lambda_5(\Phi),\lambda_6(\Phi)\},
\end{align*}
and define
\begin{equation*}
  \sigma:=\min_{\alpha<\beta\atop \alpha,\beta\in\{1,\bar{2},4,\bar{5},7\}}(\underline{\lambda}_\alpha-\bar{\lambda}_\beta).
\end{equation*}
According to \eqref{egvl}\eqref{charac speeds}\eqref{negv}, when $\delta$ is sufficiently small, $\sigma$ has a uniform positive lower bound.
For $\alpha\in\{1,\bar{2},4,\bar{5},7\}$, $z\in I_0$, it follows from \eqref{flow} that
\begin{equation*}
  z+\underline{\lambda}_\alpha t\leq X_\alpha(z,t)\leq z+\bar{\lambda}_\alpha t.
\end{equation*}
Moreover, for all $\alpha<\beta$, with $\alpha,\beta\in\{1,\bar{2},4,\bar{5},7\}$, there holds
\begin{equation*}
   X_\alpha(\eta,t)-X_\beta(2\eta,t) \geq (\eta+\underline{\lambda}_\alpha t)-(2\eta+\bar{\lambda}_\beta t)=-\eta+(\underline{\lambda}_\alpha-\bar{\lambda}_\beta)t\geq-\eta+\sigma t.
\end{equation*}
Note that the above difference is strictly positive when
\begin{equation}\label{t0}
  t>t_0^{(\eta)}:=\frac{\eta}\sigma>0.
\end{equation}
This implies that the five characteristic strips are well separated after $t_0^{(\eta)}$.

Let $\theta$ be a small parameter obeying $0<\theta\ll1$. We will prescribe initial data of order $O(\theta)$ and define
\begin{equation}\label{W0}
  W_0^{(\eta)}:=\max_{i=1,\cdots,7}\sup_{z_i}|w_i^{(\eta)}(z_i,0)|.
\end{equation}
We require the initial data satisfying the following conditions
\begin{equation}\label{Wr}
\begin{split}
&W_0^{(\eta)}=\sup_{z_1}|w_1^{(\eta)}(z_1,0)|=O(\theta),\\
&\max_{i=2,\cdots,7}\sup_{z_i}|w_i^{(\eta)}(z_i,0)|\leq C\eta( W_0^{(\eta)})^2.
\end{split}
\end{equation}
Next, we introduce the main quantities\footnote{For notational simplicity, we suppress the index $\eta$ in the solutions.}  to be estimated in our proof:
\begin{align}
  S_i(t):=&\sup_{(z'_i,s'_i)\atop z'_i\in[\eta,2\eta],\ 0\leq s'_i\leq t}\rho_i(z'_i,s'_i),&S(t):=&\max_{i}S_i(t),\label{norm1}\\
  J_i(t):=&\sup_{(z'_i,s'_i)\atop z'_i\in[\eta,2\eta]\ 0\leq s'_i\leq t}|v_i(z'_i,s'_i)|,&J(t):=&\max_{i}J_i(t),\label{norm2}\\
  W(t):=&\max_i\sup_{(x',t')\atop 0\leq t'\leq t}|w_i(x',t')|,&\bar{U}(t):=&\sup_{(x',t')\atop 0\leq t'\leq t}|\Phi(x',t')|,\label{norm3}
\end{align}
\begin{align*}
  & V_i(t):=\sup_{(x',t')\notin\mathcal{R}_i,\atop 0\leq t'\leq t}|w_i(x',t')|\quad \text{for} \quad i=1,4,7,\\
  &V_{\bar{2}}(t):=\max\sup_{(x',t')\notin\mathcal{R}_2\cup\mathcal{R}_3,\atop 0\leq t'\leq t}\{|w_2(x',t')|,|w_3(x',t')|\},\\
 &V_{\bar{5}}(t):=\max\sup_{(x',t')\notin\mathcal{R}_5\cup\mathcal{R}_6,\atop 0\leq t'\leq t}\{|w_5(x',t')|,|w_6(x',t')|\},
\end{align*}
\begin{equation}\label{normv}
V(t):=\max_{i}V_i(t)\quad \text{for} \ i=1,\bar{2},4,\bar{5},7.
\end{equation}
Furthermore, in order to deal with the potentially harmful terms in \eqref{abw2}-\eqref{abr6}, we need to estimate the following upper bound of $w_2,w_3$ (or $w_5,w_6$) in the unions of the overlapped characteristic strips $\mathcal{R}_2\cup\mathcal{R}_3$ (or $\mathcal{R}_5\cup\mathcal{R}_6$),
\begin{equation} \label{new norm}
\check{W}(t):=\max\Big\{\sup_{(x',t')\in\mathcal{R}_2\cup\mathcal{R}_3,\atop 0\leq t'\leq t}\{w_2,w_3\},\sup_{(x',t')\in\mathcal{R}_5\cup\mathcal{R}_6,\atop 0\leq t'\leq t}\{w_5,w_6\}\Big\}.
\end{equation}

{\bf $\bullet$ Estimates in the region $[0,t_0^{(\eta)}]$:} In the region $t\leq t_0^{(\eta)}$, the characteristic strips $\mathcal{R}_i$ are overlapped and are not separated. Nevertheless, the inverse densities of characteristics still have positive lower bounds in this region. Thus, no shock happens. In below, for $W(t)$, $V(t)$, $S(t)$, $J(t)$ and $\bar{U}(t)$ defined in \eqref{norm1}-\eqref{normv}, we derive a priori estimates for them before $t_0^{(\eta)}$.

{{\bf Step 1.} \textit{Estimate of $W(t)$.}} With $t\in[0,t_0^{(\eta)}]$ by \eqref{eqw} we have
\begin{equation*}
  \frac{\partial}{\partial s_i}|w_i|\leq \Gamma W^2.
\end{equation*}
As in \cite{christodoulou}, via comparing $w_i$ with solutions to the following ODE
\begin{equation*}
  \left\{\begin{array}{ll}
  \frac{d}{dt}Y=\Gamma Y^2,\\
  Y(0)=W_0^{(\eta)},
  \end{array}
  \right.
\end{equation*}
we obtain
\begin{equation}\label{wt0}
  |w_i|\leq Y(t)=\frac{W_0^{(\eta)}}{1-\Gamma W_0^{(\eta)} t}\quad \text{for}\quad t<\min\Big\{\frac1{\Gamma W_0^{(\eta)}},t_0^{(\eta)}\Big\}.
\end{equation}
And by \eqref{t0}\eqref{order1}, we have
\begin{equation}\label{gwt}
  \Gamma W_0^{(\eta)} t_0^{(\eta)}=O(\eta W_0^{(\eta)})=O(\theta).
\end{equation}
Applying \eqref{gwt} to \eqref{wt0}, for a small parameter $\varepsilon\in(0,\frac1{100}]$ it holds that
\begin{equation*}
  |w_i(x,t)|\leq (1+\varepsilon)W_0^{(\eta)}\quad \text{for any}\ x\in\mathbb{R}\ \text{and}\ t\in[0,t_0^{(\eta)}].
\end{equation*}
According to \eqref{norm3}, this implies that
\begin{equation}\label{y15}
  |W(t)|\leq (1+\varepsilon)W_0^{(\eta)}\quad \text{for}\  t\in[0,t_0^{(\eta)}].
\end{equation}

{{\bf Step 2.} \textit{Estimate of $S(t)$.}} From \eqref{eqrho} and \eqref{order1}, $\rho_i$ satisfies
\begin{equation*}
  \frac{\partial\rho_i}{\partial s_i}=O(\rho_iW).
\end{equation*}
Integrating the above equation along $\mathcal{C}_i$, we deduce
\begin{equation}\label{rho i}
  \rho_i(z_i,t)=\rho_i(z_i,0)\exp\big(O(tW(t))\big).
\end{equation}
Then plugging \eqref{319} into \eqref{rho i}, one can see that $\rho_i(z_i,t)$ is always positive. Moreover, invoking \eqref{t0} and \eqref{y15}, it follows that
\begin{equation}\label{320}
  \rho_i(z_i,t)=\exp\big(O(\eta W_0^{(\eta)})\big),\quad\text{for}\ t\in[0,t_0^{(\eta)}].
\end{equation}
We can further choose sufficiently small $\theta$ such that
\begin{equation}\label{321}
  1-\varepsilon\leq\exp\big(O(\eta W_0^{(\eta)})\big)\leq1+\varepsilon.
\end{equation}
Thus, for $t \in[0,t_0^{(\eta)}]$, we have that $\rho_i$ obeys the following lower and upper bounds
\begin{equation}\label{y16}
 1-\varepsilon\leq\rho_i(z_i,t)\leq1+\varepsilon\quad \text{and}\quad S(t)=O(1).
\end{equation}

{{\bf Step 3.} \textit{Estimate of $V(t)$.}} First note that the initial data are compactly supported in $I_0$. Formulated in characteristic coordinates $(z_i',s_i')$, this means $w_i^{(\eta)}(z_i,0)=0$ for $z'_i\notin I_0$. We then integrate \eqref{eqw} along the characteristic $C_i$ and obtain
\begin{equation*}
  V(t)=O(\eta [W(t)]^2)=O(\eta [W_0^{(\eta)}]^2).
\end{equation*}

{{\bf Step 4.} \textit{Estimate of $J(t)$.}} From \eqref{eqv} it follows
\begin{equation*}
\frac{\partial v_i}{\partial s_i}=O(S(t)[W(t)]^2).
\end{equation*}
Using \eqref{y15} and \eqref{y16}, for $t\in[0,t_0^{(\eta)}]$, we have
\begin{equation}
  J(t)=O(W_0^{(\eta)}+t[W(t)]^2)=O(W_0^{(\eta)}+\eta [W_0^{(\eta)}]^2)=O(W_0^{(\eta)}).
\end{equation}

{{\bf Step 5.} \textit{Estimate of $\bar{U}(t)$.}} The estimate for $\bar{U}$ can be derived via the expression of wave decomposition. We integrate the formula \eqref{phix}
\begin{equation}\label{phi}
  \Phi(x,t)=\int_{X_7(\eta,t)}^x\frac{\partial \Phi(x',t)}{\partial x} dx'=\int_{X_7(\eta,t)}^x\sum_{k}w_kr_k(x',t)dx'.
\end{equation}
Since $r_k(\Phi)=O(1)$, \eqref{phi} implies
\begin{equation}\label{y19}
  |\Phi(x,t)|=O\Big(\sum_{k}\int_{X_7(\eta,t)}^{X_1(2\eta,t)|}|w_k(x',t)|dx'\Big).
\end{equation}
Employing \eqref{y15} and the definition of $\bar{U}$ in \eqref{norm3}, we get
\begin{equation*}
   \bar{U}(t)=O\big(W(t)[\eta+(\bar{\lambda}_1-\underline{\lambda}_7)t]\big)=O(\eta W_0^{(\eta)})\quad \text{for any}\ t\in[0,t_0^{(\eta)}].
\end{equation*}

In summary, before the separating time $t_0^{(\eta)}$, we derive the following a priori estimates
\begin{align}
  &W(t)=O(W_0^{(\eta)}),\quad V(t)=O(\eta [W_0^{(\eta)}]^2),\quad S(t)=O(1), \\ \nonumber
  &J(t)=O(W_0^{(\eta)}),\quad \bar{U}(t)=O(\eta W_0^{(\eta)}).
\end{align}

{\bf $\bullet$ Estimates in the region $[t_0^{(\eta)},T]$:} This part encodes the key estimates of this paper. After $t_0^{(\eta)}$, $\mathcal{R}_2$ and $\mathcal{R}_3$ (or $\mathcal{R}_5$ and $\mathcal{R}_6$) may \underline{overlap} for a long time due to the non-strict hyperbolicity. Therefore, we let $\mathcal{R}_{\bar{2}}:=\mathcal{R}_2\bigcup\mathcal{R}_3$ and $\mathcal{R}_{\bar{5}}:=\mathcal{R}_5\bigcup\mathcal{R}_6$ and investigate the well-separated five characteristic strips $\{\mathcal{R}_1,\mathcal{R}_{\bar{2}},\mathcal{R}_4,\mathcal{R}_{\bar{5}},\mathcal{R}_7\}$ in this part.

{{\bf Step 1.} \textit{Estimate of $S(t)$.}} Based on the estimates we obtain in the region $[0,t_0^{(\eta)}]$, we move forward to derive estimates in the region $t\in[t_0^{(\eta)},T]$ with $T$ a potential shock formation time. Let us start with estimating $S(t)$, i.e., the supremum of inverse densities. For $\alpha=1,4,7$, if $(x,t)\in \mathcal{R}_\alpha$, by \eqref{eqrho}, we have
\begin{equation}\label{rho147}
  \frac{\partial \rho_\alpha}{\partial s_\alpha}=O(J_\alpha+ V S_\alpha).
\end{equation}
Thus, integrating \eqref{rho147} along the characteristic $\mathcal{C}_\alpha$, one can obtain that
\begin{equation}\label{rho148}
  S_\alpha(t)=O(1+tJ+ tV S)\quad \text{with}\ \alpha=1,4,7.
\end{equation}
Employing \eqref{infmrho}, we note that $\rho_3w_2$ is absent in the equation of $\rho_3$ in \eqref{eqrho}, neither is $\rho_5w_6$ in the equation of $\rho_5$. Therefore, the estimate \eqref{rho148} also holds for $\alpha=3,5$. Thus,
\begin{equation} \label{S13457}
  S_\alpha(t)=O(1+tJ+ tV S)\quad \text{for}\quad \alpha=1,3,4,5,7.
  \end{equation}

We then deal with $\rho_2$. In equation \eqref{eqrho}
 \begin{equation*}
  \frac{\partial \rho_2}{\partial s_2}=O(J_2+ V S_2+\rho_2w_3),
\end{equation*}
 a potentially harmful nonlinear term $\rho_2w_3$ pops up.  When $(x,t)\in \mathcal{R}_2$, it is possible that the point $(x,t)$ also belongs to $\mathcal{R}_3$. So we can not use $V$ to bound $w_3$ directly as in \eqref{rho147}. Our strategy is to estimate $w_3$ in $ \mathcal{R}_2 \cup \mathcal{R}_3$ via $\check{W}$. With this approach, we obtain
 \begin{equation}\label{bds2}
  S_2(t)=O(1+tJ+ tV S+t\check{W}S).
  \end{equation}
 Likewise, $S_6(t)$ shares the same bound as in \eqref{bds2}. Together with \eqref{S13457}, consequently, $S(t)$ obeys
 \begin{equation}\label{bds}
  S(t)=O(1+tJ+ tV S+t\check{W}S).
  \end{equation}

{{\bf Step 2.} \textit{Estimate of $J(t)$.}} For $(x,t)\in \mathcal{R}_\alpha$ with $\alpha=1,4,7$, it follows from \eqref{eqv} that
\begin{equation*}
  \frac{\partial v_\alpha}{\partial s_\alpha}=O(VJ_\alpha+ V^2 S_\alpha).
\end{equation*}
Via integration along the characteristics, we derive the following estimates:
\begin{equation}\label{v147}
  J_\alpha(t)=O(W_0^{(\eta)}+tVJ_\alpha+ tV^2 S_\alpha)\quad \text{for}\ \alpha=1,4,7.
\end{equation}

For $\{v_\alpha(x,t)\}_{\alpha=2,3,5,6}$ and $(x,t)\in \mathcal{R}_\alpha$, it holds that
\begin{align}
 & \frac{\partial v_2}{\partial s_2}=O\Big(VJ_2+S_2 V^2+J_2w_3+VS_2w_3\Big),\\
 &\frac{\partial v_3}{\partial s_3}=O\Big(VJ_3+S_3V^2+J_3w_2+VS_3w_2\Big),\\
&\frac{\partial v_5}{\partial s_5}=O\Big(VJ_5+S_5 V^2+J_5w_6+VS_5w_6\Big),\\
&\frac{\partial v_6}{\partial s_6}=O\Big(VJ_6+S_6 V^2+J_6w_5+VS_6w_5\Big).
\end{align}
In the same manner as to prove \eqref{bds2}, we have
\begin{equation}\label{v2356}
  \{J_\alpha\}_{\alpha=2,3,5,6}=O(W_0^{(\eta)}+tVJ+ tV^2 S+tJ\check{W}+tVS\check{W}).
\end{equation}
Combining \eqref{v147} and \eqref{v2356}, we conclude that
\begin{equation}\label{v}
  J=O(W_0^{(\eta)}+tVJ+ tV^2 S+tJ\check{W}+tVS\check{W}).
\end{equation}

{{\bf Step 3.} \textit{Estimate of $V(t)$.}} Among all a priori estimates, this step is the most delicate. We shall utilize the bi-characteristic coordinates to cope with various situations. We firstly estimate $w_i(x,t)$ for $i=1,4,7$ and $(x,t)\notin\mathcal{R}_i$. From \eqref{eqw}, it holds for $w_i$ that
\begin{equation}\label{558}
  \frac{\partial w_i}{\partial{s_i}}=O(V^2)+O\Big(\sum_{k\neq i}w_k\Big)V+O\Big(\sum_{m\neq i,k\neq i, m\neq k\atop (m,k)\neq(2,3),(m,k)\neq(5,6)}w_mw_k\Big)+{\underbrace{O\Big((\lambda_2-\lambda_3)\dboxed{w_2w_3}+(\lambda_5-\lambda_6)\dboxed{w_5w_6}\Big)}_{\text{containing weak interaction}}}.
\end{equation}
Note that $\mathcal{C}_i$ starts from $z_i\notin I_0$ and terminates at $(x,t)\notin\mathcal{R}_i$. When $t'\geq t_0^{(\eta)}$, every point $\big(X_i(z_i,t'),t'\big)$ on $\mathcal{C}_i$ falls either in $\big(\mathbb{R}\times[t_0^{(\eta)},t]\big)\setminus\bigcup_{k=1,\cdots,7}\mathcal{R}_k$ or in $\mathcal{R}_k$ for some $k\neq i$.
\begin{figure}[H]
\centering
\begin{tikzpicture}[fill opacity=0.5, text opacity=1, draw opacity=1]
\filldraw[white, fill=gray!40](2.5,0)..controls (3,1) and (3.5,1.5)..(5,3)--(2,3)..controls (2.2,2.5) and (3,2)..(3.5,0);
\filldraw[white, fill=gray!80](2.5,0)..controls (3,1) and (3.5,1.5)..(5,3)--(6,3)..controls (4.5,1.5) and (4,1)..(3.5,0);
\filldraw[white, fill=gray!80](2.5,0)..controls (2,2) and (1.5,2.5)..(1,3)--(2,3)..controls (2.2,2.5) and (3,2)..(3.5,0);
\draw[->](0,0)--(6,0)node[left,above]{$t=0$};
\draw[dashed](0,1.2)--(6,1.2)node[left,above]{$t=t_0^{(\eta)}$};
\node [below]at(3.5,0){$2\eta$};
\node [below]at(2.5,0){$\eta$};

\filldraw [black] (3.5,0) circle [radius=0.01pt]
(4,1) circle [radius=0.01pt]
(4.5,1.5) circle [radius=0.01pt]
(6,3) circle [radius=0.01pt];
\draw [color=black](3.5,0)..controls (4,1) and (4.5,1.5)..(6,3);

\filldraw [black] (2.5,0) circle [radius=0.01pt]
(3,1) circle [radius=0.01pt]
(3.5,1.5) circle [radius=0.01pt]
(5,3) circle [radius=0.01pt];
\draw [color=black](2.5,0)..controls (3,1) and (3.5,1.5)..(5,3);
\node [below] at(5.2,3){$\mathcal{R}_{i}$};

\filldraw [gray] (3.5,0) circle [radius=0.01pt]
(3,2) circle [radius=0.01pt]
(2.2,2.5) circle [radius=0.01pt]
(2,3) circle [radius=0.01pt];
\draw [color=black](3.5,0)..controls (3,2) and (2.2,2.5)..(2,3);
\node [below] at(1.8,3){$\mathcal{R}_{k}$};

\filldraw [gray] (2.5,0) circle [radius=0.01pt]
(2,2) circle [radius=0.01pt]
(1.5,2.5) circle [radius=0.01pt]
(1,3) circle [radius=0.01pt];
\draw [color=black] (2.5,0)..controls (2,2) and (1.5,2.5)..(1,3);

\filldraw [black] (1.3,0) circle [radius=0.01pt]
(1.5,0.7) circle [radius=0.01pt]
(2,1.5) circle [radius=0.01pt]
(3,2.3) circle [radius=0.8pt];
\draw [color=black,thick](1.3,0)..controls (1.5,0.7) and (2,1.5)..(3,2.3);

\node[above]at (3,2.3){$(x,t)$};
\node[below]at (1.3,0){$z_i$};

\end{tikzpicture}
\end{figure}
We denote $I_k^i=\{t'\in[t_0^{(\eta)},t]:(x,t')\in\mathcal{C}_i\bigcap\mathcal{R}_k\}$  for $k\neq i$. Integrating \eqref{558} along $\mathcal{C}_i$, with $w_i^{(\eta)}(z_i,0)=0$, we deduce
\begin{equation}\label{559}
\begin{split}
    w_i(x,t)=&O\Big(tV^2+t\check{W}^2+V\sum_{k\neq i}\int_0^tw_k\big(X_i(z_i,t'),t'\big)dt'\Big)\\
    =&O\Big(tV^2+t\check{W}^2+V\sum_{k\neq i}\int_0^{t_0^{(\eta)}}w_k\big(X_i(z_i,t'),t'\big)dt'\Big)\\
    &+O\Big(V\sum_{k\neq i}\int_{t_0^{(\eta)}}^t w_k\big(X_i(z_i,t'),t'\big)dt'\Big)\\
    =&O\Big(tV^2+t\check{W}^2+\eta[W_0^{(\eta)}]^2+V\sum_{k\neq i}\underbrace{\int_{I_k^i}w_k\big(X_i(z_i,t'),t'\big)dt'}_{M}\Big).
\end{split}
\end{equation}
For the last equality we use the fact $V(t)\leq W(t)=O(W_0^{(\eta)})$ for $t\leq t_0^{(\eta)}$.
\begin{figure}[H]
\centering
\begin{tikzpicture}[fill opacity=0.5, text opacity=1, draw opacity=1]
\filldraw[white, fill=gray!40](2.5,0)..controls (3,1) and (3.5,1.5)..(5,3)--(2,3)..controls (2.2,2.5) and (3,2)..(3.5,0);
\filldraw[white, fill=gray!80](2.5,0)..controls (3,1) and (3.5,1.5)..(5,3)--(6,3)..controls (4.5,1.5) and (4,1)..(3.5,0);
\filldraw[white, fill=gray!80](2.5,0)..controls (2,2) and (1.5,2.5)..(1,3)--(2,3)..controls (2.2,2.5) and (3,2)..(3.5,0);
\draw(0,0)--(2.5,0);
\draw[->](3.1,0)--(6,0)node[left,above]{$t=0$};
\draw[dashed](0,1.2)--(6,1.2)node[left,above]{$t=t_0^{(\eta)}$};
\node [below]at(3.8,0){$2\eta$};
\node [below]at(2.3,0){$\eta$};

\filldraw [black] (3.5,0) circle [radius=0.01pt]
(4,1) circle [radius=0.01pt]
(4.5,1.5) circle [radius=0.01pt]
(6,3) circle [radius=0.01pt];
\draw [color=black](3.5,0)..controls (4,1) and (4.5,1.5)..(6,3);

\filldraw [black] (2.5,0) circle [radius=0.01pt]
(3,1) circle [radius=0.01pt]
(3.5,1.5) circle [radius=0.01pt]
(5,3) circle [radius=0.01pt];
\draw [color=black](2.5,0)..controls (3,1) and (3.5,1.5)..(5,3);
\node [below] at(5.2,3){$\mathcal{R}_{i}$};

\filldraw [gray] (3.5,0) circle [radius=0.01pt]
(3,2) circle [radius=0.01pt]
(2.2,2.5) circle [radius=0.01pt]
(2,3) circle [radius=0.01pt];
\draw [color=gray](3.5,0)..controls (3,2) and (2.2,2.5)..(2,3);
\node [below] at(1.8,3){$\mathcal{R}_{k}$};

\filldraw [gray] (2.5,0) circle [radius=0.01pt]
(2,2) circle [radius=0.01pt]
(1.5,2.5) circle [radius=0.01pt]
(1,3) circle [radius=0.01pt];
\draw [color=gray] (2.5,0)..controls (2,2) and (1.5,2.5)..(1,3);

\filldraw [black] (1.3,0) circle [radius=0.01pt]
(1.5,0.7) circle [radius=0.01pt]
(2,1.5) circle [radius=0.01pt]
(2.3,1.8) circle [radius=0.8pt];
\draw [color=black](1.3,0)..controls (1.5,0.7) and (2,1.5)..(2.3,1.8);

\node[above]at (2.3,1.8){$\big(X_i(z_i,t'),t'\big)$};
\node[below]at (1.3,0){$z_i$};

\filldraw [gray] (3.1,0) circle [radius=0.01pt]
(2.8,1) circle [radius=0.01pt]
(2.5,1.5) circle [radius=0.01pt]
(2.3,1.8) circle [radius=1pt];
\draw [color=black,thick](3.1,0)..controls (2.8,1) and (2.5,1.5)..(2.3,1.8);
\node[below]at (3.1,0){$y_k$};

\draw[line width=2pt,color=black,thick](2.5,0)--(3.1,0);
\end{tikzpicture}
\end{figure}
Now we need to bound the integral $M$ in \eqref{559}. When $\big(X_i(z_i,t'),t'\big)\in I_k^i$ for some $k\neq i$ (See the above figure), we employ the bi-characteristic coordinates to estimate the integral $M$
\begin{equation}\label{560}
  \begin{split}
    &\int_{I_k^i}w_k\big(X_i(z_i,t'),t'\big)dt'\\
    =&O\Big(\int_{y_k\in[\eta,2\eta]}\bigg|\frac{\rho_k\big(y_k,t'(y_i,y_k)\big)}{\lambda_i-\lambda_k}w_k\big(y_k,t'(y_i,y_k)\big)\bigg|dy_k\Big)\\
    =&O(\eta J_k).
  \end{split}
\end{equation}
Here we use the fact that $\lambda_i-\lambda_k\neq 0$ for $i=1,4,7$. Then, together with \eqref{559} and \eqref{560}, for $i=1,4,7$, we obtain
\begin{equation}\label{355}
  w_i(x,t)=O(\eta [W_0^{(\eta)}]^2+tV^2+t\check{W}^2+\eta VJ)\quad\text{for any}\ (x,t)\notin\mathcal{R}_i.
\end{equation}

Next, we estimate other $w_i$, i.e., $\{V_i\}_{i=2,3,5,6}$. Here we provide the details of the estimate for $V_3$. $\{V_j\}_{j=2,5,6}$ can be handled in a similar way. For $w_3(x,t)$, with $(x,t)\notin \mathcal{R}_2\cup\mathcal{R}_3$, we have
\begin{equation}
  \frac{\partial w_3}{\partial{s_3}}=O\bigg(V^2+\Big(\sum_{k\neq 2,k\neq 3}w_k\Big)V+\sum_{m,k\neq 3, m\neq k,\atop(m,k)\neq(5,6)}w_mw_k+\dboxed{(\lambda_5-\lambda_6)w_5w_6}\bigg).
\end{equation}
Integrating $\partial_{s_3}w_3$ along $\mathcal{C}_3$, it is possible that some points $(x',t')\in \mathcal{C}_3$ lie in the region $\mathcal{R}_5\cup\mathcal{R}_6$ for $t_0^{(\eta)}\leq t'<t$. Taking into account of this scenario, we derive the following estimate
\begin{equation}\label{5.63}
\begin{split}
    w_3(x,t)=&O\Big(tV^2+tV\check{W}+t\check{W}^2+V\sum_{k=1,4,7}\int_0^tw_k\big(X_2(z_2,t'),t'\big)dt'\Big)\\
    =&O\Big(tV^2+tV\check{W}+t\check{W}^2+\eta[W_0^{(\eta)}]^2+V\sum_{k=1,4,7}\int_{I_k^3}w_k\big(X_2(z_2,t'),t'\big)dt'\Big).
    \end{split}
\end{equation}
Similar to \eqref{559} and \eqref{560}, it then follows
\begin{equation}\label{bdvr}
  V_3=O\left(\eta [W_0^{(\eta)}]^2+tV^2+tV\check{W}+t\check{W}^2+\eta VJ\right).
\end{equation}
Thus, we conclude from \eqref{355} and \eqref{bdvr}
\begin{equation}\label{bdv}
 V=O\left(\eta [W_0^{(\eta)}]^2+tV^2+tV\check{W}+t\check{W}^2+\eta VJ\right).
\end{equation}

{{\bf Step 4.} \textit{Estimate of $\check{W}(t)$.}} Now we proceed to estimate
$$\check{W}(t)=\max\Big\{\sup_{(x',t')\in\mathcal{R}_2\cup\mathcal{R}_3,\atop 0\leq t'\leq t}\{w_2,w_3\},\sup_{(x',t')\in\mathcal{R}_5\cup\mathcal{R}_6,\atop 0\leq t'\leq t}\{w_5,w_6\}\Big\},$$
i.e., the supremum of $w_2,w_3$ (or $w_5,w_6$) taken in the unions of the corresponding overlapped characteristic strips $\mathcal{R}_2\cup\mathcal{R}_3$ (or $\mathcal{R}_5\cup\mathcal{R}_6)$. Note that the \underline{strong} interaction of $w_2,w_3$ (or $w_5,w_6$) will appear in the following estimates. And their characteristic speeds are almost-the-same. We focus on estimating $w_3(x,t)$ with $(x,t)\in\mathcal{R}_2\cup\mathcal{R}_3$. Recall the equation of $w_3$,
\begin{equation}\label{eqw3}
  \frac{\partial w_3}{\partial{s_3}}=O\Big[w_3^2+\boxed{w_2w_3}+w_2\Big(\sum_{k\neq 2,k\neq 3}w_k\Big)+w_3\Big(\sum_{k\neq 2,k\neq 3}w_k\Big)+\sum_{m,k\neq 2,3, m\neq k}w_mw_k\Big].
\end{equation}
The marked term $\boxed{w_2w_3}$ corresponds to the strong nonlinear interaction. Integrating \eqref{eqw3}, it holds that
\begin{equation}
w_3(x,t)=O\left(\eta[W_0^{(\eta)}]^2+t\check{W}^2+tV\check{W}+tV^2\right),\quad \text{for}\quad (x,t)\in\mathcal{R}_2\cup\mathcal{R}_3.
\end{equation}
The estimates for $w_2,w_5,w_6$ can be obtained in a same manner as for $w_3$. And hence we prove
\begin{equation}
\check{W}(t)=O(\eta[W_0^{(\eta)}]^2+t\check{W}^2+tV\check{W}+tV^2).
\end{equation}

{{\bf Step 5.} \textit{Estimate of $\bar{U}(t)$.}} Finally, the estimate of $\Phi$ can be deduced via the aforementioned estimates. If $(x,t)$ does not belong to any characteristic strip, one can go back to \eqref{y19} and obtain
\begin{equation}\label{u1}
 \bar{U}(t)=O\big((\eta+(\bar{\lambda}_1-\underline{\lambda}_7)t)V\big).
\end{equation}
If $(x,t)\in\mathcal{R}_k$ for some $k$, we have
\begin{equation}\label{u2}
\begin{split}
  |\Phi(x,t)|=&\Big|\int_{X_k(\eta,t)}^x\frac{\partial \Phi(x',t)}{\partial x} dx'\Big|=\Big|\int_{X_k(\eta,t)}^x\sum_{k}w_kr_k(x',t)dx'\Big|\\
  \leq&\int_{X_k(\eta,t)}^{X_k(2\eta,t)}|w_k(x',t)|dx'=O\Big(\int_{\eta}^{2\eta}|w_k(z_k,t)|\rho_kdz_k\Big)=O(\eta J).
  \end{split}
\end{equation}
Combining \eqref{u1} and \eqref{u2}, we therefore obtain
\begin{equation}
   \bar{U}(t)=O(\eta J+\eta V+tV),\quad \text{for any}\ t\in[t_0^{(\eta)},T].
\end{equation}

In conclusion, for all $ t\in[0,T]$, we derive the following a priori estimates
\begin{align}
    S=&O(1+tVS+tJ+tS\check{W}),\label{y25}\\
    J=&O(W_0^{(\eta)}+t VJ+tV^2 S+t\check{W}J+tVS\check{W}),\label{364}\\
    V=&O(\eta [W_0^{(\eta)}]^2+tV^2+\eta VJ+tV\check{W}+t\check{W}^2), \label{366}\\
    \check{W}=&O(\eta[W_0^{(\eta)}]^2+t\check{W}^2+tV\check{W}+tV^2),\label{wc}\\
    \bar{U}=&O(\eta J+\eta V+ tV). \label{y31}
\end{align}

These a priori estimates will be used in a bootstrap argument in the next subsection. After closing the bootstrap argument, we will obtain the upper bounds for all the quantities defined in \eqref{norm1}-\eqref{new norm}.

\subsection{$L^\infty$-estimates and bootstrap argument}\label{bootargu}
In this subsection, we design a bootstrap argument to bound $S$, $J$, $V$, $\check{W}$ and $\bar{U}$. More precisely, employing the desired a priori estimates \eqref{y25}-\eqref{y31}, we obtain the following proposition:
\begin{prop}[{\bf $L^\infty$-estimates}] \label{upperbd}
For $t\in[0,T_\eta^*)$ with $T_\eta^*$ being the (hypothetical) shock formation time satisfying
\begin{equation}\label{368}
  T_\eta^*\leq\frac{C}{W_0^{(\eta)}},
\end{equation}
we have
\begin{align}
  S(t)=&O(1),\label{63}\\
   J(t)=&O(W_0^{(\eta)}),\label{64}\\
   V(t)=&O\left(\eta  [W_0^{(\eta)}]^2\right),\label{66}\\
    \check{W}(t)=&O\left(\eta  [W_0^{(\eta)}]^2\right),\label{667}\\
    \bar{U}(t)=&O(W_0^{(\eta)}).\label{barU}
\end{align}
\end{prop}
\begin{remark} \label{J1 bound}
In this subsection, we establish the above upper-bound estimates via a bootstrap argument. Later, in Subsection \ref{pf1.2}, for proving the shock formation theorem, with these obtained upper-bound estimates, we further derive a lower bound estimate \eqref{vlu} for $v_1$ and show that the corresponding inverse density $\rho_1\to 0$ (see \eqref{711}-\eqref{rho1 zero}) when $t\to T_\eta^*$. This indicates the shock formation happens at $T_\eta^*$ for the first family of the characteristic waves. In addition, in Proposition \ref{wave dynamics}, we also establish improved upper-bound estimates \eqref{other vi} for other $v_i$ within $\mathcal{R}_i$, and show that $\rho_i$ (with $i\neq 1$) is bounded away from $0$. Consequently, the other characteristic waves remain regular up to $T_\eta^*$. 
\end{remark}
\begin{remark} \label{tboot}
A posteriori, we will prove that the shock formation time $T_\eta^*$ verifies \eqref{368} in Subsection \ref{pf1.2}. In \eqref{Tshock}, we will also specify the value of the constant $C$ in \eqref{368}.
\end{remark}

\begin{proof}
When $t=0$, by \eqref{319} and definitions in \eqref{norm1}-\eqref{new norm}, it holds
\begin{equation*}
 S(0)=1,\quad J(0)=W_0^{(\eta)},\quad V(0)=0.
\end{equation*}

To obtain \eqref{63}-\eqref{66}, with $\theta$ small, we impose bootstrap assumptions:
\begin{align}
  tV\leq& \theta^{\frac12},\label{y33}\\
   t\check{W}\leq& \theta^{\frac12},\label{wca}\\
   J\leq&\theta^{\frac12}.\label{Ja}
\end{align}
Applying \eqref{y33} and \eqref{wca} to \eqref{y25}, we have
\begin{align}\label{y34}
 S=O(1+tJ+\theta^{\frac12}S)\quad \text{when $\theta$ is small, this implies}\quad S=O(1+tJ).
\end{align}

To control $J(t)$, we appeal to \eqref{364}. Using \eqref{y33}\eqref{wca}\eqref{y34}, we obtain
\begin{align}
\text{via plugging \eqref{y33}\eqref{wca} into \eqref{364}},\quad J&=O(W_0^{(\eta)}+\theta^{\frac12}J+\theta^{\frac12}V S).\notag\\
\text{Absorb $\theta^{\frac12}J$, we have}\quad J&=O(W_0^{(\eta)}+\theta^{\frac12}V S)\notag\\
\text{by \eqref{y34},} \quad\  &=O(W_0^{(\eta)}+\theta^{\frac12}V +\theta J).\notag\\
\text{Absorb $\theta J$, it holds}\quad J&=O(W_0^{(\eta)}+\theta^{\frac12}V).\label{y35}
\end{align}

We proceed to estimate $V(t)$. Applying \eqref{y33}-\eqref{Ja} to \eqref{366}, we get
\begin{align*}
  V=O\Big(\eta [W_0^{(\eta)}]^2+\theta^{\frac12}V+\eta \theta^{\frac12} V+\theta^{\frac12}\check{W}\Big).
\end{align*}
For $0<\theta\ll 1$, the second and third terms on the right can be absorbed by the left, hence it infers
\begin{equation}\label{Vz}
  V=O\Big(\eta [W_0^{(\eta)}]^2+\theta^{\frac12}\check{W}\Big).
  \end{equation}

We then deal with $\check{W}$, invoking \eqref{y33}\eqref{wca}\eqref{Vz} into \eqref{wc}, we deduce
 \begin{align}
\text{via plugging \eqref{y33}\eqref{wca} into \eqref{wc},}\quad \check{W}&=O(\eta(W_0^{(\eta)})^2+\theta^{\frac12}\check{W}+\theta^{\frac12}V).\notag\\
\text{Absorb $\theta^{\frac12}\check{W}$, we have}\quad \check{W}&=O(\eta(W_0^{(\eta)})^2+\theta^{\frac12}V)\notag\\
\text{by \eqref{Vz},} \quad\quad  &=O(\eta[W_0^{(\eta)}]^2+\theta^{\frac12}\eta (W_0^{(\eta)})^2+\theta\check{W}).\notag\\
\text{Absorb $\theta\check{W}$, it holds}\quad \check{W}&=O(\eta(W_0^{(\eta)})^2). \label{Wcc}
 \end{align}
This finishes the proof of \eqref{667}. We then substitute \eqref{Wcc} into \eqref{Vz}, hence \eqref{66} can be verified.
Furthermore, with \eqref{368}, we have
\begin{equation}\label{tv}
 tV=O(\eta W_0^{(\eta)})=O(\theta)<\theta^{\frac12},
\end{equation}
\begin{equation}\label{tw}
 t\check{W}=O(\eta W_0^{(\eta)})=O(\theta)<\theta^{\frac12}.
\end{equation}
Note that \eqref{tv} and \eqref{tw} improve the bootstrap assumptions \eqref{y33} and \eqref{wca}.

Next, for $J(t)$ and $S(t)$, it follows from \eqref{y35}\eqref{66} that the estimates \eqref{64} and \eqref{63} are valid:
\begin{equation}\label{J}
J=O(W_0^{(\eta)}+\theta^{\frac12}\eta [W_0^{(\eta)}]^2)=O(W_0^{(\eta)})=O(\theta)<\theta^{\frac12},
\end{equation}
\begin{equation}
 S=O(1+tW_0^{(\eta)})=O(1).
\end{equation}
And \eqref{J} improves the bootstrap assumption \eqref{Ja}.

Finally, by \eqref{63}-\eqref{667} we deduce the estimate for $\bar{U}$
\begin{equation}\label{68}
\bar{U}(t)=O(\eta W_0^{(\eta)}+\eta^2 W_0^{(\eta)}+\eta [W_0^{(\eta)}]^2)=O(\eta W_0^{(\eta)})=O(\theta),\ t\in[0,T_\eta^*].
\end{equation}
Furthermore, since $\bar{U}$ is the supremum of $\Phi$ as defined in \eqref{norm3}, with sufficiently small $\theta$, we then conclude that the solution $\Phi$ always stays in $B_{\delta}^7(0)$.

\end{proof}

\subsection{Shock formation} \label{pf1.2}
We will finish the proof of Theorem \ref{shock} in this subsection. In particular, we will show that in $\mathcal{R}_1$ the inverse density $\rho_1\to 0$ and $w_1 \to \infty$ as $t \to T_\eta^*$ while $\Phi$ is uniformly bounded. Hence a shock happens at $T_\eta^*$.
\begin{proof}[{Proof of Theorem \ref{shock}}]
Our first step is to control $v_1$. We integrate \eqref{eqv}
$$
\frac{\partial v_1}{\partial s_1}=O\Big(\sum_{m\neq 1}w_m\Big)v_1+\underbrace{O\Big(\sum_{m,k\neq1, m\neq k}w_mw_k\Big)}_{\text{containing weak interaction}}\rho_1
$$
along $\mathcal{C}_1$ and obtain
\begin{equation}\label{79}
  v_1(z_1,t)=w_1^{(\eta)}(z_1,0)+O(tVJ+tV^2S)=w_1^{(\eta)}(z_1,0)+O(\eta [W_0^{(\eta)}]^2).
\end{equation}
In particular, taking $z_0\in I_0$ such that $w_1^{(\eta)}(z_0,0)=W_0^{(\eta)}$, then for sufficiently small $\theta$, the above formula yields
\begin{equation}\label{vlu}
 (1-\varepsilon)W_0^{(\eta)}\leq v_1(z_0,t)\leq (1+\varepsilon) W_0^{(\eta)}.
\end{equation}

We then estimate $\rho_1$. Using \eqref{eqrho}
$$
  \frac{\partial\rho_1}{\partial s_1}=c_{11}^1(\Phi)v_1+O\Big(\sum_{ k\neq 1}w_k\Big)\rho_1,
$$
and the fact $c_{11}^1(\Phi)<0$, we obtain
\begin{equation}\label{381}
  -|c_{11}^1(\Phi)|v_1-\Big|O\Big(\sum_{ k\neq 1}w_k\Big)\Big|\rho_1\leq\frac{\partial\rho_1}{\partial s_1}\leq -|c_{11}^1(\Phi)|v_1+\Big|O\Big(\sum_{ k\neq 1}w_k\Big)\Big|\rho_1.
\end{equation}
We will then employ the Gr\"{o}nwall's inequality to derive upper and lower bounds of $\rho_1$. Firstly, by \eqref{68}, we have $|\Phi|=O(\eta W_0^{(\eta)})=O(\theta)\leq\delta$. Setting $\theta$ to be sufficiently small, we further have
\begin{equation}\label{y37}
 (1-\varepsilon)|c_{11}^1(0)|\leq |c_{11}^1(\Phi)|\leq (1+\varepsilon)|c_{11}^1(0)|.
\end{equation}
For the coefficient $O\Big(\sum_{ k\neq 1}w_k\Big)$ in \eqref{381}, using bi-characteristic coordinates (as in Subsection \ref{bootargu}), we deduce
\begin{equation*}
\begin{split}
  \int_0^{t}\sum_{k\neq 1}|w_k(X_1(z_1,t'),t')|dt'
 =\sum_{k\neq 1}\int_{-2\eta}^{2\eta}\frac{|v_k|}{\lambda_1-\lambda_k}dy_k
  =O(\eta W_0^{(\eta)}+\eta J)=O(\eta W_0^{(\eta)})=O(\theta).
  \end{split}
\end{equation*}
Here we use the fact that the factor $\frac{1}{\lambda_1-\lambda_k}$ arising from the bi-characteristic transformation is of $O(1)$. For $0<\theta\ll 1$, the above estimate implies
\begin{equation}\label{385}
  1-\varepsilon\leq \exp{\Big(\int_0^{t}\pm|O\big(\sum_{k\neq 1}w_k(X_1(z_1,t'),t')\big)|dt'\Big)}\leq1+\varepsilon.
\end{equation}
Applying the Gr\"{o}nwall's inequality to \eqref{381} and combining with \eqref{y37}-\eqref{385}, we get
\begin{equation}\label{y38}
\begin{split}
  &(1-\varepsilon)(1-(1+\varepsilon)^2|c_{11}^1(0)|\int_0^tv_1(z_1,t')dt')\\
  \leq&\rho_1(z_1,t)\\
\leq& (1+\varepsilon)(1-(1-\varepsilon)^2|c_{11}^1(0)|\int_0^tv_1(z_1,t')dt').
\end{split}
\end{equation}
By the first inequality of \eqref{vlu}, we then arrive at
\begin{equation}\label{711}
  \rho_1(z_0,t)\geq(1-\varepsilon)\Big(1-(1+\varepsilon)^3|c_{11}^1(0)|tW_0^{(\eta)}\Big),
\end{equation}
which shows that $\rho_1(z_0,t)>0$ when $t<\frac{1}{(1+\varepsilon)^3|c_{11}^1(0)|W_0^{(\eta)}}$. Meanwhile, the second inequality of \eqref{vlu} implies
\begin{equation} \label{rho1 upp}
  \rho_1(z_0,t)\leq (1+\varepsilon)\Big(1-(1-\varepsilon)^3|c_{11}^1(0)|tW_0^{(\eta)}\Big).
\end{equation}
Together with \eqref{711}, we conclude that there exists a finite positive time $T_\eta^*(z_0)$ (shock formation time) such that
\begin{equation} \label{rho1 zero}
  \lim_{t\rightarrow T_\eta^*(z_0)} \rho_1(z_0,t)=0\quad \text{and}\quad \lim_{t\rightarrow T_\eta^*(z_0)} w_1(z_0,t)=\lim_{t\rightarrow T_\eta^*(z_0)} \frac{v_1}{\rho_1}(z_0,t)=\infty.
\end{equation}
And $T_\eta^*(z_0)$ obeys
\begin{equation} \label{Tshock}
 \frac{1}{(1+\varepsilon)^3|c_{11}^1(0)|W_0^{(\eta)}}\leq T_\eta^*(z_0)\leq\frac{1}{(1-\varepsilon)^3|c_{11}^1(0)|W_0^{(\eta)}}.
\end{equation}
As we mentioned in Remark \ref{tboot}, this verifies the condition \eqref{368}.
\end{proof}
\begin{remark}[{\bf First shock formation at time $T_\eta^*$}]
In the above proof, we show that along the first family of characteristic issuing from $z_0$, a shock forms at $t=T_\eta^*(z_0)$. We then consider the below neighborhood of $z_0$:
\begin{equation} \label{Tshock4}
\mathcal{C}_{\mathrm{Sing}}(z_0)=\left\{z\,\ \left|\,\frac{(1-\varepsilon)^3}{(1+\varepsilon)^3}<\frac{w_1^{(\eta)}(z,0)}{W_0^{(\eta)}}\right.\right\}.
\end{equation}
With the same argument, for any $z\in \mathcal{C}_{\mathrm{Sing}}(z_0)$, we can also prove that a shock forms at time $t=T_\eta^*(z)$ along the characteristic issuing from $z$. Moreover, similarly as in \eqref{Tshock}, the lifespan $T_\eta^*(z)$ satisfies the following lower and upper bounds estimates:
\begin{equation} \label{Tshock2}
\underline{T}_\eta(z)\triangleq\frac{1}{(1+\varepsilon)^3|c_{11}^1(0)|w_1^{(\eta)}(z,0)}\leq T_\eta^*(z)\leq\frac{1}{(1-\varepsilon)^3|c_{11}^1(0)|w_1^{(\eta)}(z,0)}\triangleq \overline{T}_\eta(z).
\end{equation}
Since $z_0$ is the point where the initial datum $w_1^{(\eta)}(z,0)$ attains its maximum $W_0^{(\eta)}$, in \eqref{Tshock2}, we can see that $\overline{T}_\eta(z)\geq \overline{T}_\eta(z_0)$ and $\underline{T}_\eta(z)\geq \underline{T}_\eta(z_0)$. Outside of $\mathcal{C}_{\mathrm{Sing}}(z_0)$, note that we require $|w_1^{(\eta)}(z,0)|\geq \frac{(1-\varepsilon)^3}{(1+\varepsilon)^3}W_0^{(\eta)}$, hence we have that $\underline{T}_\eta(z)\geq\overline{T}_\eta(z_0)$. This implies that the shock forming along the characteristic starting from $z\notin\mathcal{C}_{\mathrm{Sing}}(z_0)$ is later than the one along the $z_0$-issuing characteristic. Within $\mathcal{C}_{\mathrm{Sing}}(z_0)$, where $|w_1^{(\eta)}(z,0)|\approx W_0^{(\eta)}$, we then find the first shock formation time, and it is 
$$T_\eta^*=\inf \{T_\eta^*(z)|z\in\mathcal{C}_{\mathrm{Sing}}(z_0)\}.$$ 
Since the initial data are smooth and compactly supported, the first shock formation time $T_\eta^*$ can indeed be attained. For notational simplicity, in the following part of this paper, we abuse the notation slightly and still use $z_0$ to label the characteristic along which the first shock forms, and it may be a point close to the previous $z_0$. 
\end{remark}
Recall that the estimates of $V(t)$ and $\check{W}(t)$ in Section \ref{bootargu} show that the unknowns $\{w_i(z,t)\}_{i=2,3,5,6}$ are bounded in the whole space before the shock formation time $T_\eta^*$, and $\{w_i(z,t)\}_{i=4,7}$  are bounded when $(z,t)\notin\{\mathcal{R}_i\}_{i=4,7}$. In fact, now we prove that $\{w_i(z,t)\}_{i=4,7}$ are also regular in $\{\mathcal{R}_i\}_{i=4,7}$. To show this, we introduce
\begin{equation}
\check{W}_i'(t):=\sup_{(z,t')\in\mathcal{R}_i\atop 0\leq t'\leq t}\{w_i(z,t')\}\quad \text{for}\quad i=4,7
\end{equation}
and prove the following:
\begin{prop}
For $t\in[0,T_\eta^*)$ with the shock formation time $T_\eta^*$ satisfying \eqref{Tshock}, we have $\check{W}_i'(t)=O([W_0^{(\eta)}]^2)$ for $i=4,7$.
\end{prop}
\begin{proof}
We study $w_4$ and $w_7$ separately. For $w_4$, since $c^4_{44}=0$, it verifies
\begin{equation}\label{eqw4}
  \frac{\partial w_4}{\partial{s_4}}=O\Big(\sum_{k\neq 4}w_k\Big)w_4+\underbrace{O\Big(\sum_{m\neq 4,k\neq 4\atop m\neq k}w_mw_k\Big)}_{\text{containing weak interaction}}.
\end{equation}
Let $(z,t)\in\mathcal{R}_4$ and $t\in[t_0^{(\eta)},T_\eta^*)$, integrating \eqref{eqw4} along $\mathcal{C}_4$, we obtain
\begin{equation}
w_4(z,t)\leq O(\eta [W_0^{(\eta)}]^2+tV\check{W}_4'+tV^2).
\end{equation}
In view of the facts $tV=O(\theta)$ and $V=O\big(\eta[W_0^{(\eta)}]^2\big)$, we hence get
\begin{equation} \label{Wcc4}
\check{W}_4'\leq O([W_0^{(\eta)}]^2).
\end{equation}

For $w_7$, even though its evolution equation is of Riccati-type, its initial datum is designed to be much more smaller than that of $w_1$. This guarantees that $w_7(z,t)$ is regular when $t<T_\eta^*$. Invoking the estimate of $V$, it holds that
\begin{equation}\label{eqw7}
\begin{split}
  \frac{\partial w_7}{\partial{s_7}}=&-c^7_{77}w_7^2+O\Big(\sum_{k\neq 7}w_k\Big)w_7+\underbrace{O\Big(\sum_{m\neq 7,k\neq 7\atop m\neq k}w_mw_k\Big)}_{\text{containing weak interaction}}\\
  =&-c^7_{77}w_7^2+O(V)w_7+O(V^2)\\
  =&-c^7_{77}w_7^2+O\big(\eta[W_0^{(\eta)}]^2\big)w_7+O\big(\eta^2[W_0^{(\eta)}]^4\big).
  \end{split}
\end{equation}
Equivalently,
\begin{small}
\begin{equation} \label{w7ode}
\frac{d}{ds}\Big[\exp\Big(O\big(\eta[W_0^{(\eta)}]^2\big)s\Big)w_7\Big]=-\exp\Big(O\big(\eta[W_0^{(\eta)}]^2\big)s\Big)c^7_{77}w_7^2+\exp\Big(O\big(\eta[W_0^{(\eta)}]^2\big)s\Big)O\big(\eta^2[W_0^{(\eta)}]^4\big).
\end{equation}
\end{small}
Since
\begin{equation}
\exp\Big(O\big(\eta[W_0^{(\eta)}]^2\big)s\Big)\leq \exp\Big(O\big(\eta[W_0^{(\eta)}]^2\big)T_\eta^*\Big)\leq \exp\Big(O\big(\eta W_0^{(\eta)}\big)\Big)=O(e^\theta),
\end{equation}
by choosing $\theta$ sufficiently small, we have
\begin{equation}
1-\varepsilon\leq\exp\Big(O\big(\eta[W_0^{(\eta)}]^2\big)s\Big)\leq 1+\varepsilon.
\end{equation}
Thus, back to \eqref{w7ode}, it holds
\begin{equation}\label{deqw7}
\begin{split}
\frac{d}{ds}\Big[\exp\Big(O\big(\eta[W_0^{(\eta)}]^2\big)s\Big)w_7\Big]
\leq Cw_7^2+O\big(\eta^2[W_0^{(\eta)}]^4\big).
\end{split}
\end{equation}
Integrating \eqref{deqw7} along $\mathcal{C}_7$, we obtain
\begin{equation}\label{621}
\begin{split}
\check{W}_7'&\leq O\big(w_7(z,0)+t\check{W}_7'^2+t\eta^2[W_0^{(\eta)}]^4\big)\\
&\leq O\big([W_0^{(\eta)}]^2+t\check{W}_7'^2+\eta^2[W_0^{(\eta)}]^3\big)\\
&\leq O\big([W_0^{(\eta)}]^2+t\check{W}_7'^2\big).
\end{split}
\end{equation}
Now we introduce an additional bootstrap assumption
\begin{equation} \label{boottcw}
t\check{W}_7'\leq \theta^{\frac12}.
\end{equation}
Then by \eqref{621}, it holds
\begin{equation} \label{Wcc7}
\check{W}_7'\leq  O\big([W_0^{(\eta)}]^2\big)
\end{equation}
and the bootstrap assumption \eqref{boottcw} can be improved by
\begin{equation}
t\check{W}_7'\leq  O(W_0^{(\eta)})=O(\theta)< \theta^{\frac12}.
\end{equation}
\end{proof}
Moreover, we show that the shock formation \underline{only} occurs in $\mathcal{R}_1$ by deriving positive lower bounds for $\rho_2,\cdots,\rho_7$. This concludes the proof of Theorem \ref{shock}.
\begin{prop}\label{wave dynamics}
 For $t< T_\eta^*$, we have $\{\rho_i\}_{i=2,\cdots,7}$ are all bounded away from zero in the whole $(x,t)$-plane, i.e., $\rho_i>C>0$ for a uniform constant $C$.
\end{prop}
\begin{proof}
To derive a uniform lower bound for $\rho_i$, we start from estimates inside $\mathcal{R}_i$. Similarly to \eqref{y38}, via estimating \eqref{eqrho}, we have
\begin{equation*}
\begin{split}
 \rho_i(z_i,t)\geq(1-\varepsilon)\Big(1-(1+\varepsilon)^2|c_{ii}^i(0)|\int_0^t|v_i(z_i,t')|dt'\Big).
\end{split}
\end{equation*}
Similar as in \eqref{79}\eqref{vlu}, and noting that the initial data of other $w_i$ (with $i\neq 1$) are smaller than that of $w_1$ (see \eqref{Wr}), we derive the following improved estimates: 
\begin{equation}\label{other vi}
 v_i(z_i,t)=w_i^{(\eta)}(z_i,0)+O(tVJ+tV^2S)=O(\eta [W_0^{(\eta)}]^2).   
\end{equation}
Employing the estimates \eqref{other vi} of $v_i$, one can then bound $\rho_i$ from below in $\mathcal{R}_i$.

We then prove that $\{\rho_i\}_{i=1,\cdots,7}$ obey a positive lower bound outside $\mathcal{R}_i$. For the equation \eqref{eqrho} of $\rho_i$, applying the Gr\"{o}nwall's inequality and estimates obtained in Section \ref{pfthm1.2}, one can obtain that for any $z_i\notin [\eta,2\eta],\ t\in[t_0^{(\eta)},T_\eta^*)$,
\begin{equation}
  \rho_i(z_i,t)\geq\exp\Big(-\int_0^t\Big|O\Big(\sum_k w_k(z_i,s)\Big)\Big|ds\Big)\geq 1-\varepsilon.
\end{equation}
Interested readers are also referred to Proposition 8.1 in \cite{an} for a more detailed discussion.
\end{proof}

\section{Proof of $H^2$ ill-posedness (Theorem \ref{3D})} \label{pfill}
In this section, we use the above shock formation to prove $H^2$ ill-posedness for 3D ideal MHD equations \eqref{MHD}. We first construct a family of initial data which will lead to the ill-posedness. In particular, we employ the aforementioned shock-formation argument only in a region where the initial data enjoy plane symmetry. We denote this part of data as $w_i^{(\eta)}(z,0)$. For each $\eta$, we will prove that there is a corresponding shock formation at finite time $T_\eta^*$. The meaning of the ill-posedness is then as follows: for the constructed sequence of initial data with $H^2$ norms tending to $0$ as $\eta\to 0$, the $L^2$ norm of the solution $w_1$ will blow up at $T_\eta^*$ for each $\eta$ due to shock formation.

\subsection{Initial data designed for proving ill-posedness} \label{illdata}
Our first step is to select the initial data. Counterexamples of local well-posedness will evolve from them.

As mentioned in Remark \ref{nonsym data}, the initial datum $\hat{w}_1(x,x_2,x_3)$ is independent of $x_2$, $x_3$ in a small ball $B_{\frac\eta2}^3$, but not planar symmetric in the exterior region. Here, $B_{\frac\eta2}^3$ denotes the 3D ball of radius $\frac\eta2$ centered at $(\frac{3\eta}{2},0,0)$. Let $\mathcal{D}_\eta$ be the development region of $B_{\frac\eta2}^3$ along the $1^{\text{st}}$ family of characteristics. Within $\mathcal{D}_\eta$, owing to finite propagating-speed property of hyperbolic system, the planar symmetry is preserved, and we employ the 1+1 dimensional shock-formation argument to derive the desired ill-posedness. Initial data outside of $B_{\frac\eta2}^3$ are constructed to be compactly supported. We multiply the planar symmetric profile with a suitable cut-off function depending on all the three variables $x,x_2,x_3$. In particular, we prescribe 
	\begin{equation} \label{data}
		w_1^{(\eta)}(x,x_2,x_3,0)=\theta |\ln (x)|^\alpha \mathcal{X}(\frac{x}{\eta})\psi\Big(\frac{|\ln(x)|^\delta x_2}{\sqrt{x}}\Big)\psi\Big(\frac{|\ln(x)|^\delta x_3}{\sqrt{x}}\Big),
	\end{equation}
	where $0<|\delta|<1$ and
	$$\mathcal{X}(x)=\left\{\begin{split}
		&1,\quad x\in[\frac65,\frac95],\\
		&0,\quad x\in(-\infty, 1]\cup [2,+\infty),
	\end{split}\right.
	\qquad \psi(x)=\left\{\begin{split}
		&1,\quad |x|\leq \frac14,\\
		&0,\quad |x|\geq\frac12,
	\end{split}\right.
	$$ 
	and $\mathcal{X}_\eta(x)=\mathcal{X}(\frac{x}{\eta})$. Following the construction in our result \cite{an3} for 2D ideal MHD and Euler systems, by choosing $\theta\sim|\ln\eta|^{2\delta-\alpha} $ and $0<|\delta|<1$, we obtain the desired initial data in $H^1(\mathbb{R}^3)$ such that $\|w_1^{(\eta)}(\cdot,0)\|_{H^1(\mathbb{R}^3)}\to 0$ as $\eta \to 0$. Moreover, regarding the magnetic field, to accommodate the constraint $\nabla \cdot H=0$, the component $H_1$ cannot remain trivial anymore in the exterior region due to the absence of plane symmetry. Instead, $H_1$ must be determined by the constraint equation, together with the obtained initial data of $H_2$ and $H_3$. For more details about the construction, interested readers are referred to our 2D work \cite{an3}. Next, for $k=2,\cdots,7$, we set
\begin{equation} \label{data2}
w_k^{(\eta)}(x,x_2,x_3,0)=\eta\theta^2 \mathcal{X}(\frac{x}{\eta})\psi\Big(\frac{|\ln(x)|^\delta x_2}{\sqrt{x}}\Big)\psi\Big(\frac{|\ln(x)|^\delta x_3}{\sqrt{x}}\Big),
\end{equation}
which are much smaller than initial data of $w^{(\eta)}_1$. One can verify that these initial data satisfy all the requirements in Section \ref{pfthm1.2}.

Recall that
\begin{equation} \label{combine}
\Phi_x(x,t)=\sum_{k=1}^7 w_k(x,t)r_k(\Phi).
\end{equation}
Now we go back to the original system \eqref{pMHD} by reversing the process of wave decomposition. The initial data of the original variables $\Phi$ are denoted as $\Phi(x,0)=\hat\Phi_0(x,x_2,x_3)$. According to \eqref{combine}, the initial data satisfy
\begin{equation} \label{combo2}
\partial_x\hat\Phi_{0}(x,x_2,x_3)=\sum_{k=1}^7 \hat w_k(x,x_2,x_3)r_k\big(\Phi(x,0)\big).
\end{equation}
Since $r_k$ constructed in \eqref{regv} is \underline{Lipschitz continuous} in $\Phi$, by a standard ODE argument, given the aforementioned $\hat w_k(x,x_2,x_3)$, one can solve the ODE system \eqref{combo2} for $\hat\Phi_{0}$. Moreover, because $r_k(\Phi)\in L^\infty( B_\delta^7(0))$ and $\hat w_k\in H^1(\mathbb{R}^3)$, we have $\hat\Phi_0\in \dot{H}^2(\mathbb{R}^3)$. These constitute the initial data for the 3D ideal compressible MHD equations \eqref{pMHD}.

\subsection{Estimate for $\partial_{z_1}\rho_1$}\label{supbound}
In this subsection, we derive a uniform bound for $|\partial_{z_1}\rho_1|$. As we will see in Subsection \ref{ill}, this upper bound plays a crucial role in proving the desired ill-posedness result.

For a fixed $\eta>0$, we employ characteristic and bi-characteristic coordinates introduced in \eqref{flow}-\eqref{dense} for $i\neq j$:
\begin{equation*}
  (x,t)=\big(X_i(z_i,s_i),s_i\big)=\big(X_i(y_i,t'(y_i,y_j)),t'(y_i,y_j)\big)=\big(X_j(y_j,t'(y_i,y_j)),t'(y_i,y_j)\big).
\end{equation*}
As discussed in Section \ref{pre}, the transformations between these coordinates satisfy
\begin{equation}\label{92}
  \partial_{y_i}=\frac{\rho_i}{\lambda_j-\lambda_i}\partial_{s_j}=\partial_{z_i}+\frac{\rho_i}{\lambda_j-\lambda_i}\partial_{s_i}.
\end{equation}
Along $\mathcal{C}_1$ we fix $y_1$ and set $y_7$ as a parameter. From \eqref{92}, we have
\begin{equation} \label{coordtrans}
  \partial_{y_1}=\partial_{z_1}+\frac{\rho_1}{\lambda_7-\lambda_1}\partial_{s_1}=\partial_{z_1}-\frac{\rho_1}{2C_f}\partial_{s_1},
\end{equation}
where $C_f$ is the fast wave speed defined in \eqref{charac speeds}. We now state the main result of this subsection.
\begin{prop} \label{dzrho1}
Fix $\eta$, for all $z_i\in\mathbb{R}$ and any $t<T_\eta^*$, there holds $|\partial_{z_1}\rho_1|\leq C_\eta$ for a positive constant $C_\eta$.
\end{prop}

\begin{proof}
Using \eqref{coordtrans}, we get
\begin{equation}\label{rhod}
 \partial_{z_1}\rho_1=\partial_{y_1}\rho_1+\frac{\rho_1}{2C_f}\partial_{s_1}\rho_1.
\end{equation}

To bound $\partial_{z_1}\rho_1$, we start with controlling $\partial_{y_1}\rho_1$. Let
\begin{equation*}
\tau_{1} ^{(7)}:=\partial_{y_1}\rho_1,\quad\pi_{1}^{(7)}:=\partial_{y_1}v_1.
\end{equation*}
Since
\begin{equation}\label{96}
  \partial_{y_7}=\frac{\rho_7}{\lambda_1-\lambda_7}\partial_{s_1}=\frac{\rho_7}{2C_f}\partial_{s_1},
\end{equation}
we have that $\tau_{1}^{(7)}$ obeys
\begin{equation}\label{97}
\begin{split}
  \partial_{y_7}\tau_{1}^{(7)}=&\partial_{y_1}\partial_{y_7}\rho_1=\partial_{y_1}\Big(\frac{\rho_7}{2C_f}\partial_{s_1}\rho_1\Big)=\partial_{y_1}\Big(\frac{\rho_1\rho_7}{2C_f}\sum_m c_{1m}^1w_m\Big)\\
  =&\frac{\rho_7}{2C_f}\Big(c_{11}^1\partial_{y_1}v_1+\sum_{m\neq1}c_{1m}^1w_m \partial_{y_1}\rho_1\Big)+\frac{\rho_1}{2C_f}\Big(\sum_{m\neq7}c_{1m}^1w_m\partial_{y_1}\rho_7+c_{17}^1\partial_{y_1}v_7\Big)\\
  &+\frac{\rho_1\rho_7}{2C_f}\Big(\sum_{m\neq1,m\neq7}c_{1m}^1\frac1{\rho_m}\partial_{y_1}v_m-\sum_{m\neq1,m\neq7}c_{1m}^1\frac{w_m}{\rho_m}\partial_{y_1}\rho_m\Big)\\
  &-\frac{\partial_{y_1}C_f}{2C_f^2}\rho_1\rho_7\sum_m c_{1m}^1w_m+\frac{\rho_1\rho_7}{2C_f}\sum_m \partial_{y_1}c_{1m}^1w_m.
\end{split}
\end{equation}
We still need to estimate $\partial_{y_1}C_f,\partial_{y_1}c_{1m}^1,\partial_{y_1}\rho_m$ and $\partial_{y_1}v_m$ in \eqref{97}. Note that for $C_f$, its derivative satisfies
\begin{equation}\label{9.11}
  \begin{split}
    \partial_{y_1}C_f=&\nabla_\Phi(\lambda_1-u_1)\cdot\partial_{y_1}\Phi=\nabla_\Phi\lambda_1\cdot[\partial_{s_7}X_7\partial_{y_1}t'\partial_{x}\Phi+\partial_{y_1}t'\partial_{t}\Phi]\\
    =&\nabla_\Phi(\lambda_1-u_1)\cdot\Big[\lambda_7\frac{\rho_1}{\lambda_7-\lambda_1}\sum_kw_kr_k+\frac{\rho_1}{\lambda_7-\lambda_1}(-A(\Phi)\sum_kw_kr_k)\Big]\\
    =&O\Big(v_1+\rho_1\sum_{k\neq1}w_k\Big).
  \end{split}
\end{equation}
Similarly, for $c^1_{1m}$, we also have
\begin{equation}\label{9.12}
  \begin{split}
    \partial_{y_1}c_{1m}^1=&\nabla_\Phi c_{1m}^1\cdot\partial_{y_1}\Phi=\nabla_\Phi c_{1m}^1\cdot[\partial_{s_7}X_7\partial_{y_1}t'\partial_{x}\Phi+\partial_{y_1}t'\partial_{t}\Phi]\\
    =&\nabla_\Phi c_{1m}^1\cdot\Big[\lambda_7\frac{\rho_1}{\lambda_7-\lambda_1}\sum_kw_kr_k+\frac{\rho_1}{\lambda_7-\lambda_1}(-A(\Phi)\sum_kw_kr_k)\Big]\\
    =&O\Big(v_1+\rho_1\sum_{k\neq1}w_k\Big).
  \end{split}
\end{equation}
By \eqref{92}, it holds
\begin{equation}\label{9.14}
  \partial_{y_1}\rho_m=\frac{\rho_1}{\lambda_m-\lambda_1}\partial_{s_m}\rho_m=O\Big(\rho_mv_1+\rho_1\rho_m\sum_{k\neq 1}w_k\Big) \quad \text{when}\quad m\neq 1,
\end{equation}
and
\begin{equation}\label{9.15}
  \partial_{y_1}v_m=\frac{\rho_1}{\lambda_m-\lambda_1}\partial_{s_m}v_m=O\Big(\rho_mv_1\sum_{k\neq1}w_k+\rho_1\rho_m\sum_{j\neq1,k\neq1\atop j\neq k}w_jw_k\Big)\quad \text{when}\quad m\neq 1.
\end{equation}
Together with the estimates obtained in Proposition \ref{upperbd}, \eqref{9.11}-\eqref{9.15} imply that \eqref{97} could be rewritten as
\begin{equation}\label{linear}
\begin{split}
  \partial_{y_7}\tau_{1}^{(7)}=B_{11}^\eta\tau_1^{(7)}+B_{12}^\eta\pi_1^{(7)}+B_{13}^\eta,
\end{split}
\end{equation}
where $B_{11}^\eta,B_{12}^\eta,B_{13}^\eta$ are uniformly bounded constants depending on $\eta$.

In the same fashion, we obtain the evolution equation for $\pi_1^{(7)}$:
\begin{equation}\label{917}
  \begin{split}
  \partial_{y_7}\pi_{1}^{(7)}
  =&\frac{\rho_7 }{2C_f}\Big(\sum_{p\neq 1,q\neq 1\atop p\neq q}\gamma_{pq}^1w_pw_q \tau_{1}^{(7)}+\sum_{p\neq 1}\gamma_{1p}^1w_p\pi_{1}^{(7)}\Big)\\
  &-\frac{\partial_{y_1}C_f}{4C_f^2}\Big(\sum_{p\neq 1}\gamma_{1p}^1 w_p v_1\rho_7+\sum_{p\neq 1,q\neq 1\atop p\neq q}\gamma_{pq}^1w_pw_q\rho_7\rho_1\Big)\\
    &+\frac{\rho_7\rho_1}{2C_f}\Big(\sum_{p\neq 1}\partial_{y_1}\gamma_{1p}^1 w_p w_m+\sum_{p\neq 1,q\neq 1\atop p\neq q}\partial_{y_1}\gamma_{pq}^1w_pw_q\Big)\\
    &+\frac{\rho_1}{2C_f}\Big(\sum_{p\neq 1}\gamma_{1p}^1 w_p w_1+\sum_{p\neq 1,q\neq 1\atop p\neq q}\gamma_{pq}^1w_pw_q\Big)\partial_{y_1}\rho_7\\
    &+\frac{\rho_7\rho_1}{2C_f }\Big(\sum_{p\neq 1,p\neq7}\gamma_{1p}^1\frac{w_1}{\rho_p}+\sum_{p\neq 1,q\neq 1\atop p\neq q}\gamma_{pq}^1\frac{w_q}{\rho_p}\Big)(\partial_{y_1}v_p-w_p\partial_{y_1}\rho_p)\\
    &+\frac{\rho_7\rho_1}{2C_f }\sum_{p\neq 1,q\neq 1\atop p\neq q}\gamma_{pq}^1\frac{w_p}{\rho_q}(\partial_{y_1}v_q-w_q\partial_{y_1}\rho_q)
    +\frac{\rho_1}{2C_f }\gamma_{17}^1 w_1\partial_{y_1}v_7\\
    =&B_{21}^\eta\tau_1^{(7)}+B_{22}^\eta\pi_1^{(7)}+B_{23}^\eta,
  \end{split}
\end{equation}
where one can check as in \eqref{linear} $B_{21}^\eta,B_{22}^\eta,B_{23}^\eta$ are also uniformly bounded.

Next, we check the initial data of $\tau_1^{(7)}$ and $\pi_m^{(7)}$. Since $\rho_1(z_1,0)=1,$ for fixed $\eta$, it follows from \eqref{92} that
\begin{equation*}
  \begin{split}
    \tau_1^{(7)}(z_1,0)=&\partial_{z_1}\rho_1(z_1,0)-\frac{\rho_1(z_1,0)}{2C_f}\partial_{s_1}\rho_1(z_1,0)\\
     =&-\frac{1}{2C_f}\sum_{k}c_{1k}^1 w_k^{(\eta)}(z_1,0)=O(W_0^{(\eta)})<+\infty.
  \end{split}
\end{equation*}
And with $v_1(z_1,0)=w_1^{(\eta)}(z_1,0)$, we similarly deduce
\begin{equation*}
  \begin{split}
    \pi_1^{(7)}(z_1,0)=&\partial_{z_1}v_1(z_1,0)-\frac{\rho_1(z_1,0)}{2C_f}\partial_{s_1}v_1(z_1,0)\\
    =&\partial_{z_1}w_1^{(\eta)}(z_1,0)-\frac{1}{2C_f}\sum_{q\neq 1,q\neq p}\gamma_{pq}^1w_p(z_1,0)w_q(z_1,0)\rho_1(z_1,0)\\
    =&O(\partial_{z_1}w_1^{(\eta)}(z_1,0)+[W_0^{(\eta)}]^2)<+\infty.
  \end{split}
\end{equation*}
Applying Gr\"onwall's inequality to \eqref{linear} and \eqref{917}, for $t< T_\eta^*$ we deduce that $\tau_{1}^{(7)}:=\partial_{y_1}\rho_1$ is bounded by a constant depending on $\eta$.

Back to \eqref{rhod}, we hence prove
\begin{equation*}
  \partial_{z_1}\rho_1=\partial_{y_1}\rho_1+O(v_1+\sum_{m\neq 1}w_m\rho_1).
\end{equation*}
With the bounds for $J(t)$, $S(t)$ and $V(t)$ in Section \ref{bootargu}, consequently, we arrive at the conclusion that $\partial_{z_1}\rho_1$ is bounded by $C_\eta$, which is a constant depending on $\eta$. Note that the dependence on $\eta$ causes no trouble since the blow-up is derived for each fixed $\eta$ in Subsection \ref{ill}.
\end{proof}

\subsection{Ill-posedness mechanism} \label{ill}
We prove our main result Theorem \ref{3D} in this section. For our proof, one can see that the ill-posedness is driven by the shock formation.
\begin{proof}[{Proof of Theorem \ref{3D}}]
 With initial data \eqref{data} and \eqref{data2}, employing the estimates in Subsection \ref{apes} and Subsection \ref{bootargu}, we know that a shock forms in $\mathcal{R}_1$ at time $T_\eta^*$ with
\begin{equation} \label{weakill}
 \frac{1}{(1+\varepsilon)^3|c_{11}^1(0)|W_0^{(\eta)}}\leq T_\eta^*\leq\frac{1}{(1-\varepsilon)^3|c_{11}^1(0)|W_0^{(\eta)}}.
\end{equation}
By \eqref{data}, we see that $w_1^{(\eta)}(z,0)=\theta \cdot O(|\ln z|^\alpha)$ with $0<\alpha<\frac{1}{2}$. 

Furthermore, we will show the $H^1$-norm of solutions to \eqref{MHD} blows up at $t=T_\eta^*$. In particular, we give a lower bound estimate of the $H^2$-norm in the following spatial region
\begin{equation} \label{region}
\Omega_{T_\eta^*}\approx\{(x,x_2,x_3,T_\eta^*): x=X_1(z,T_\eta^*),\ (x-\lambda_1T_\eta^*-\frac{3\eta}2)^2+x_2^2+x_3^2\leq\frac{\eta^2}{4},  \  \text{for}\ \eta\leq z\leq2\eta \}
\end{equation}
where $\Omega_{T_\eta^*}$ denotes the development (at time $T_\eta^*$) evolving from the planar symmetric portion $B_{\frac\eta2}^3$ of the data.

To calculate $\int_{\Omega_{T_\eta^*}}|w_1(x,x_2,x_3,T_\eta^*)|^2dxdx_2dx_3$, we first compute the area element $\int dx_2 dx_3$. By the definition of characteristics, for a characteristic line $\mathcal{C}_1$ starting from $z$, its terminal at time $T_\eta^*$ is $x=z+\lambda_1T_\eta^*$. Thus, one can verify that the area of $\Omega_{T_\eta^*}$ is
\begin{equation}
\begin{split}
\int_{x_2^2+x_3^2\leq\frac{\eta^2}{4}-(x-\lambda_1T_\eta^*-\frac{3\eta}{2})^2} dx_2 dx_3
=&\pi\Big[\frac{\eta^2}{4}-(z-\frac{3\eta}{2})^2\Big]=\pi(z-\eta)(2\eta-z).
\end{split}
\end{equation}
Employing characteristic coordinates, we obtain
\begin{equation} \label{w1l2}
\|w_1\|_{L^2(\Omega_{T_\eta^*})} ^2=\pi\int_{\frac\eta2}^\eta\big|\frac{v_1}{\rho_1}(z,T_\eta^*)\big|^2\rho_1(z,T_\eta^*)(z-\eta)(2\eta-z)dz.
\end{equation}
Recall  the estimate of  $v_1$ in \eqref{79}
\begin{equation} \label{v1low}
  v_1(z,t)=w_1^{(\eta)}(z,0)+O(\eta [W_0^{(\eta)}]^2).
\end{equation}
We integrate in a subinterval $(z_0,z_0^*]\subseteq[\eta,2\eta]$ where $|w_1^{(\eta)}(z,0)|>\frac12W_0^{(\eta)}$ for $z\in (z_0,z_0^*]$. Then in this subinterval, it follows from \eqref{v1low} that $|v_1(z,t)|$ admits a lower bound $|v_1(z,t)|\geq \frac14W_0^{(\eta)}$. Applying this in \eqref{w1l2} and using the fact that $\rho_1(z_0,T_\eta^*)=0$, we get
\begin{equation} \label{normblowup}
\begin{split}
\|w_1\|_{L^2(\Omega_{T_\eta^*})} ^2&\geq C\int_{z_0}^{z_0^*}\big|\frac{v_1}{\rho_1}(z,T_\eta^*)\big|^2\rho_1(z,T_\eta^*)(z-\eta)(2\eta-z)dz\\
&\geq C(W_0^{(\eta)})^2 (z_0-\eta)(2\eta-z_0^*)\underbrace{\int_{z_0}^{z_0^*}\frac{1}{\rho_1(z,T_\eta^*)-\rho_1(z_0,T_\eta^*)}dz}_{\text{by $\rho_1(z_0,T_\eta^*)=0$}}\\
&\geq C(W_0^{(\eta)})^2 (z_0-\eta)(2\eta-z_0^*)\underbrace{\int_{z_0}^{z_0^*}\frac{1}{(z-z_0)\sup|\partial_z\rho_1|}dz}_{\text{by mean value theorem}}\\
&\geq C_\eta\int_{z_0}^{z_0^*}\frac{1}{z-z_0}dz=+\infty.
\end{split}
\end{equation}
This shows the blow-up of the $L^2$-norm of $w_1$. In \eqref{normblowup}, we crucially use the fact that $|\partial_z\rho_1|$ is bounded from Proposition \ref{dzrho1}.

The last step is going back to the original physical unknowns $\Phi$ and show the blow-up of $\|\Phi\|_{H^1}$. Due to the boundedness of $\{\sup_{(x,T_\eta^*)\in\mathcal{R}_i}w_i(x,T_\eta^*)\}_{i=2,\cdots,7}$ in \eqref{Wcc}, \eqref{Wcc4} and \eqref{Wcc7}, it holds
\begin{equation}
\Big\{\sup_{(x,T_\eta^*)\in\mathcal{R}_i}w_i(x,T_\eta^*)\Big\}_{i=2,\cdots,7}\leq C W_0^{(\eta)},
\end{equation}
together with the fact that the first components of $r_2, \ r_4$ and $r_6$ are zero, we deduce
\begin{equation*}
  \begin{split}
    \|\partial_x u_1\|_{L^2(\Omega_{T_\eta^*})}=&\|\sum_{k=1}^7 w_k r_{k}^1\|_{L^2(\Omega_{T_\eta^*})}\geq C\Big(\|w_1\|_{L^2(\Omega_{T_\eta^*})}-\sum_{k=3,5,7}\|w_k \|_{L^2(\Omega_{T_\eta^*})}\Big)\\
    \geq&C\Big(\|w_1\|_{L^2(\Omega_{T_\eta^*})}-3\eta^\frac32W_0^{(\eta)}\Big)\geq +\infty.
  \end{split}
\end{equation*}
Here $r_k^i$ denotes the $i^{\text{th}}$ component of the right eigenvector $r_k$. Hence the blow-up of $\|w_1\|_{L^2(\Omega_{T_\eta^*})}$ implies that $\|u_1\|_{H^1(\Omega_{T_\eta^*})}$ is also infinite. Similarly, we can show $\|\partial_x \varrho\|_{L^2(\Omega_{T_\eta^*})}=+\infty.$
\begin{equation*}
  \begin{split}
    \|\partial_x \varrho\|_{L^2(\Omega_{T_\eta^*})}=&\|\sum_{k=1}^7 w_k r_{k}^4\|_{L^2(\Omega_{T_\eta^*})}\geq C\Big(\|w_1\|_{L^2(\Omega_{T_\eta^*})}-\sum_{k=3,4,5,7}\|w_k \|_{L^2(\Omega_{T_\eta^*})}\Big)\\
    \geq&C\Big(\|w_1\|_{L^2(\Omega_{T_\eta^*})}-4\eta^\frac32W_0^{(\eta)}\Big)\geq +\infty.
  \end{split}
\end{equation*}
This finishes the proof of Theorem \ref{3D}.
\end{proof}

\section{Proof of $\kappa=0$ case} \label{appendix}
In this section, we explain how to deal with the case when $\kappa=0$. We analyze the detailed subtle structures of the system \eqref{decomposed system} under this scenario. And we explain the steps to prove shock formation and ill-posedness when $\kappa=0$.

With $\kappa=0$, $\Phi$ also satisfies \eqref{pMHD}:
\begin{equation*}
\partial_t\Phi+A(\Phi)\partial_x\Phi=0,
\end{equation*}
but with a different coefficient matrix, which is
\begin{equation} \label{0MHD}
A(\Phi)=\begin{pmatrix}
u_1 & 0 & 0& c^2/\varrho & \mu_0H_2/\varrho & \mu_0H_3/\varrho & \partial_Sp/\varrho \\
0 & u_1 & 0& 0 & 0 & 0 & 0 \\
0 & 0  & u_1 & 0 & 0 & 0& 0 \\
\varrho & 0 & 0& u_1 & 0 & 0 & 0 \\
H_2 & 0 & 0& 0 &u_1 & 0 & 0 \\
H_3& 0 & 0& 0 & 0 & u_1 & 0 \\
0 & 0 & 0& 0 & 0 & 0 &  u_1
\end{pmatrix}.
\end{equation}
The eigenvalues of \eqref{0MHD} are
\begin{equation} \label{egv0}
\begin{split}
&\lambda_1=u_1+C_f,\quad \lambda_2=\lambda_3=\lambda_4=\lambda_5=\lambda_6=u_1,\quad \lambda_7=u_1-C_f.
\end{split}
\end{equation}
Here, $C_f$ denotes the fast wave speed. It is defined as
$$C_f=\Big\{\frac{\mu_0}{\varrho}(H_2^2+H_3^2)+c^2\Big\}^{1/2}.$$
And the slow and Alfv\'en wave speeds $C_s,C_a$ vanish. The eigenvalues for $A(\Phi)$ in \eqref{0MHD} satisfy
\begin{equation}\label{negv0}
\begin{split}
\lambda_7<\lambda_6=\lambda_5=\lambda_4=\lambda_3=\lambda_2<\lambda_1.
\end{split}
\end{equation}
There exist seven linearly independent right eigenvectors:
{\begin{equation} \label{rvec0}
\begin{split}
&r_1=\left(\begin{array}{cc}-1\\0\\0\\-\varrho/C_f\\-H_2/C_f\\-H_3/C_f\\0\end{array}\right),
r_2=\left(\begin{array}{cc}0\\1\\0\\0\\0\\0\\0\end{array}\right),
r_3=\left(\begin{array}{cc}0\\0\\1\\0\\0\\0\\0\end{array}\right),
r_4=\left(\begin{array}{cc}0\\0\\0\\-\partial_Sp/c^2\\0\\0\\1\end{array}\right),\\
&r_5=\left(\begin{array}{cc}0\\0\\0\\-\mu_0H_2/c^2\\1\\0\\0\end{array}\right),
r_6=\left(\begin{array}{cc}0\\0\\0\\-\mu_0H_3/c^2\\0\\1\\0\end{array}\right),
r_7=\left(\begin{array}{cc}-1\\0\\0\\\varrho/C_f\\H_2/C_f\\H_3/C_f\\0\end{array}\right).
\end{split}
\end{equation}}
In view of \eqref{egv0}-\eqref{rvec0}, when $\kappa=0$, our system is a non-strictly hyperbolic system with a \underline{quintuple} eigenvalue $\lambda_2=\lambda_3=\lambda_4= \lambda_5=\lambda_6=u_1$.
The dual left eigenvectors are
{\begin{equation}
\begin{split}
&l_1=\Big(-\frac12,0,0,-\frac{c^2}{2\varrho C_f},-\frac{\mu_0H_2}{2\varrho C_f},-\frac{\mu_0H_3}{2\varrho C_f},-\frac{\partial_Sp}{2\varrho C_f}\Big),\\
& l_2=(0,1,0,0,0,0,0),\quad l_3=(0,0,1,0,0,0,0),\quad l_4=(0,0,0,0,0,0,1),\\
& l_5=\Big(0,0,0,-\frac{c^2H_2}{\varrho C_f^2},\frac{\mu_0H_3^2+\varrho c^2}{\varrho C_f^2},-\frac{\mu_0H_2H_3}{\varrho C_f^2},-\frac{H_2\partial_Sp}{\varrho C_f^2}\Big),\\
&l_6=\Big(0,0,0,-\frac{c^2H_3}{\varrho C_f^2},-\frac{\mu_0H_2H_3}{\varrho C_f^2},\frac{\mu_0H_2^2+\varrho c^2}{\varrho C_f^2},-\frac{H_3\partial_Sp}{\varrho C_f^2}\Big),\\
&l_7=\Big(-\frac12,0,0,\frac{c^2}{2\varrho C_f},\frac{\mu_0H_2}{2\varrho C_f},\frac{\mu_0H_3}{2\varrho C_f},\frac{\partial_Sp}{2\varrho C_f}\Big).
\end{split}
\end{equation}}

Since
\begin{equation}
\begin{split}
&\nabla_\Phi\lambda_1=\Big(1,0,0,\frac{(\gamma-1)\varrho c^2-\mu_0(H_2^2+H_3^2)}{2\varrho^2 C_f},\frac{\mu_0H_2}{C_f\varrho},\frac{\mu_0H_3}{C_f\varrho},\frac{c^2}{2C_f}\Big),\\
&\nabla_\Phi\lambda_{2,3,4,5,6}=(1,0,0,0,0,0,0),\\
&\nabla_\Phi\lambda_7=\Big(1,0,0,-\frac{(\gamma-1)\varrho c^2-\mu_0(H_2^2+H_3^2)}{2\varrho^2 C_f},-\frac{\mu_0H_2}{C_f\varrho},-\frac{\mu_0H_3}{C_f\varrho},-\frac{c^2}{2C_f}\Big),\\
\end{split}
\end{equation}
by \eqref{coec} we have $c_{im}^i=0$ for $i,m=2,3,4,5,6.$
Via a direct calculation, it holds that $c_{11}^1(\Phi)=-\big(\frac12+\frac{\gamma c^2}{2C_f^2}+\frac{\mu_0(H_2^2+H_3^2)}{\varrho C_f^2}\big)<0$. Thus, the first characteristic is genuinely nonlinear. For other coefficients, noting that $l_i,r_i,\nabla_\Phi r_i$ are all regular, one can easily verify that all the coefficients $\gamma^i_{im}$, $\gamma^i_{km}$ are of $O(1)$. The structure of the decomposed equations for $H_1=0$ then coincides with \eqref{wsketch1}-\eqref{infmv}. Since $\lambda_2=\lambda_3=\lambda_4= \lambda_5=\lambda_6$, the corresponding characteristic strips $\mathcal{R}_2,\mathcal{R}_3,\mathcal{R}_4,\mathcal{R}_5,\mathcal{R}_6$ are also identical. We consider characteristics that propagate in and out of the following three characteristic strips: $\{\mathcal{R}_1,\mathcal{R}_{2},\mathcal{R}_7\}$. These three strips will be completely separated when $t>t_0^{(\eta)}$. See the following picture.
\begin{figure}[H]
\centering
\begin{tikzpicture}[fill opacity=0.5, draw opacity=1, text opacity=1]
\node [below]at(3.5,0){$2\eta$};
\node [below]at(2.3,0){$\eta$};

\filldraw[white, fill=gray!40](3.5,0)..controls (3,1) and (1.8,2.6)..(1,3.5)--(3.2,3.5)..controls (2.1,1.5) and (2.6,1)..(2.5,0);
\filldraw[white, fill=gray!40](3.5,0)..controls (3.4,1) and (3.2,1.5)..(4.8,3.5)--(6.5,3.5)..controls (3.8,1.5) and (3,1)..(2.5,0);

\filldraw[white,fill=gray!80](2.5,0)..controls (3,1) and (3.8,1.5)..(6.5,3.5)--(8.2,3.5)..controls (5,1.5) and (4.3,1)..(3.5,0);
\filldraw[white,fill=gray!80](2.5,0)..controls (2.6,1) and (2.1,1.5)..(3.2,3.5)--(4.8,3.5)..controls (3.2,1.5) and (3.4,1)..(3.5,0);
\filldraw[white,fill=gray!80](2.5,0)..controls (1.2,2.2) and (0.5,2.8)..(-0.5,3.5)--(1,3.5)..controls (1.8,2.6) and (3,1)..(3.5,0);

\draw[->](-0.5,0)--(8.2,0)node[left,below]{$t=0$};
\draw[dashed](-0.5,1.7)--(8.2,1.7)node[right,below]{$t=t_0^{(\eta)}$};

\draw [color=black](3.5,0)..controls (4.3,1) and (5,1.5)..(8.2,3.5);

\draw [color=black](2.5,0)..controls (3,1) and (3.8,1.5)..(6.5,3.5);
\node [below] at(6.2,3){$\mathcal{R}_1$};

\draw [color=black](3.5,0)..controls (3.4,1) and (3.2,1.5)..(4.8,3.5);

\draw [color=black](2.5,0)..controls (2.6,1) and (2.1,1.5)..(3.2,3.5);
\node [below] at(3.5,3){$\mathcal{R}_2$};

\draw [color=black](3.5,0)..controls (3,1) and (1.8,2.6)..(1,3.5);
\node [below] at(1.1,3){$\mathcal{R}_7$};

\draw [color=black] (2.5,0)..controls (1.2,2.2) and (0.5,2.8)..(-0.5,3.5);
\end{tikzpicture}
\caption{\small\bf Separation of three characteristic strips.}
\end{figure}
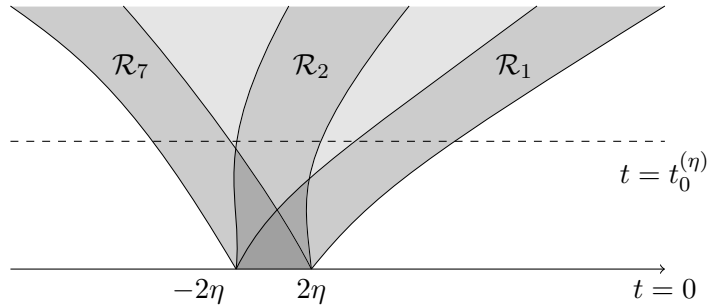
The initial data are given in \eqref{data}-\eqref{data2} and satisfy the following conditions
\begin{equation}\label{aWr}
\begin{split}
&W_0^{(\eta)}=\sup_{z_1}(w_1^{(\eta)}(z_1,0))=O(\theta),\quad\max_{i=2,3,4,5,6,7}\sup_{z_i}|w_i^{(\eta)}(z_i,0)|\leq C\eta( W_0^{(\eta)})^2.
\end{split}
\end{equation}
The main quantities to be estimated are the same as those in \eqref{norm1}-\eqref{new norm} except slight changes in the following two quantities:
\begin{align*}
  &V_{\bar{2}}(t):=\max\sup_{(x',t')\notin\mathcal{R}_{\bar{2}},\atop 0\leq t'\leq t}\{|w_2(x',t')|,|w_3(x',t')|,|w_4(x',t')|,|w_5(x',t')|,|w_6(x',t')|\},\\ &\check{W}(t):=\sup_{(x',t')\in\mathcal{R}_{\bar{2}},\atop 0\leq t'\leq t}\{w_2,w_3,w_4,w_5,w_6\}.
\end{align*}
In the same fashion as in Section \ref{pfthm1.2}, for $t\in[0,T_\eta^*)$ with $T_\eta^*\sim\frac{1}{W_0^{(\eta)}}$, we obtain the following estimates:
 \begin{align*}
  S(t)=&O(1),\quad J(t)=O(W_0^{(\eta)}),\quad V(t)=O\big(\eta  (W_0^{(\eta)})^2\big),\quad \check{W}(t)=O\big(\eta  (W_0^{(\eta)})^2\big).
\end{align*}
Finally, following the same argument as in Section \ref{pfill}, one can deduce that the $H^2$-norm of the solution to \eqref{0MHD} blows up. In conclusion, Theorem \ref{shock}-\ref{3D} also hold when $\kappa=0$.

\section{Proof of $H^2$ ill-posedness for 3D compressible Euler equations (Theorem \ref{sharpeuler})} \label{euler ill}
In this section, we study the case with the absence of magnetic field, i.e., $H=0$. The MHD system \eqref{MHD} then reduces to the 3D compressible Euler equations \eqref{euler}
\begin{equation*}
\left\{\begin{split}
&\partial_t\varrho+\nabla\cdot(\varrho u)=0,\\
&\varrho\{\partial_t+(u\cdot\nabla)\}u+\nabla p=0,\\
&\partial_t S+(u\cdot\nabla)S=0.
\end{split}
\right.
\end{equation*}
To establish ill-posedness via Lindblad-type approach, we study the above system under planar symmetry
\begin{equation}\label{peuler}
\left\{\begin{split}
&\varrho\partial_tu_1+\varrho u_1\partial_x u_1+c^2\partial_x\varrho+\partial_S p\partial_xS=0,\\
&\varrho\partial_tu_2+\varrho u_1\partial_x u_2=0,\\
&\varrho\partial_tu_3+\varrho u_1\partial_x u_3=0,\\
&\partial_t\varrho+u_1\partial_x\varrho+\varrho\partial_x u_1=0,\\
&\partial_t S+u_1\partial_xS=0.
\end{split}
\right.
\end{equation}
Recall the definition of vorticity $\omega:=\nabla \times u$. Under plane symmetry, the vorticity reads $\omega=(0,-\partial_x u_3,\partial_x u_2)$. The system \eqref{peuler} therefore allows non-trivial vorticity and entropy.

In a similar fashion, we can also construct examples of shock formation and $H^2$ ill-posedness for \eqref{peuler}. In particular, according to the low-regularity local well-posedness result in \cite{Wang19} by Wang (See also \cite{Disconzi} by Disconzi-Luo-Mazzone-Speck and \cite{zhang-andersson} by Zhang-Andersson), we will see that the $H^2$ ill-posedness is sharp.

To see the connection, we first review studies on the quasilinear wave equation. The sharp local well-posedness results were obtained by Tataru-Smith \cite{tataru} and Wang \cite{Wang}. And the sharp ill-posedness, as mentioned before, was constructed by Lindblad \cite{Lind98}. For Einstein vacuum equations, the sharp local well-posedness was established by Klainerman-Rodnianski-Szeftel \cite{klainerman-igor-szeftel}.

For the 3D compressible Euler equations, the classical local well-posedness result holds if the initial data for $(u,\varrho)$ are in $H^s$ with $s>\frac52$. In \cite{Wang19}, Wang proved the low-regularity local well-posedness requiring the initial data of velocity $u$, density $\varrho$ and vorticity $\omega$ satisfying $(u,\varrho,\omega)\in H^s\times H^s \times H^{s'}$ with $2<s'<s$. Interested readers are also referred to Disconzi-Luo-Mazzone-Speck \cite{Disconzi} and Zhang-Andersson \cite{zhang-andersson}. In \cite{Wang19}, it is also conjectured that the sharp local well-posedness results would be for initial data satisfying $(u,\varrho,\omega)\in H^s\times H^s\times H^{s-\frac12}$ with $s>2$. For the 3D incompressible case, the corresponding sharp low-regularity ill-posedness was proved by Bourgain-Li \cite{bourgain-li} for initial datum satisfying $(u,\omega) \in H^\frac52 \times H^{\frac32}$ and the ill-posedness is driven by the evolution of $\omega$. We also refer to Bourgain-Li \cite{bourgain-li3}, Elgindi-Jeong \cite{elgindi} for 2D incompressible Euler equations in $H^2$, and Bourgain-Li \cite{bourgain-li2} for 2D and 3D incompressible Euler equations in integer $C^m$ spaces. Bourgain-Li \cite{bourgain-li} provided a sharp result with respect to the regularity of vorticity. Our Theorem \ref{sharpeuler} provides a sharp companion with respect to the regularity of the fluid velocity $u$ and density $\varrho$. In particular, our initial data leading to the ill-posedness satisfy $(u,\varrho) \in H^2\times H^2$ and $\omega \in C^\infty \subseteq H^{\frac32}$. See Figure \ref{euler results}.
\begin{figure}[H]
\centering
\begin{tikzpicture}[fill opacity=0.5, draw opacity=1, text opacity=1,scale=1.2]
\filldraw[white, fill=gray!40](1,0.5)--(1,4)--(5,4)--(1.5,0.5)--(1,0.5);
\filldraw[gray!40, fill=gray!80](1,0.5)--(1,1)--(1.5,1)--(1.5,0.5)--(1,0.5);
\draw[->](0,-0.3)--(5.5,-0.3);
\filldraw(5,-0.3)node[below]{\footnotesize $s$: regularity for the velocity};
\draw[->](0,-0.3)--(0,4)node[right,above]{\footnotesize $s'$: regularity for the vorticity};
\draw(1.5,0.5)--(5,4);
\draw[dashed](1.5,0.5)--(1.5,4);
\draw[dashed](1,1)--(1.5,1);
\draw[dashed](1,4)--(5,4);
\filldraw [thick,black] (1.5,0.5) circle [radius=0.8pt];
\filldraw [thick,black] (1,0.5) circle [radius=0.8pt];
\filldraw [thick,black] (1,-0.3) circle [radius=0.8pt]node[below]{2};
\filldraw [thick,black](1.5,-0.3) circle [radius=0.8pt]node[below]{$\frac{5}{2}$};

\filldraw [thick,black](0,0.5) circle [radius=0.8pt]node[left]{$\frac32$};
\filldraw [thick,black](0,1) circle [radius=0.8pt]node[left]{2};

\draw[thick,black](1,0.5)--(1.5,0.5);
\draw[very thick,black](1,0.5)--(1,4);
\filldraw(1.65,2.2)circle node[right]{\scriptsize LWP proved by Wang(19') (See also};
\filldraw(1.65,1.9)circle node[right]{\scriptsize Disconzi-Luo-Mazzone-Speck(19'),};
\draw[->](1.75,1.9)--(1.25,1.9);
\filldraw(1.65,1.6)circle node[right]{\scriptsize Zhang-Andersson(21'))};
\filldraw(1.65,3)circle node[right]{\scriptsize Classical LWP};
\filldraw(1.65,0.75)circle node[right]{\scriptsize LWP by Conjecture in Wang(19')};
\filldraw (2,2) circle;
\draw[->](1.75,0.75)--(1.25,0.75);
\draw[->](1.75,0.3)--(1.25,0.45);
\filldraw(1.65,0.3)node[right]{\scriptsize Bourgain-Li's ill-posedness(21')};
\filldraw(1.65,0)node[right]{\scriptsize (incompressible 3D Euler)};
\draw[->](-0.1,2.25)--(0.95,2.25);
\filldraw(0,2.5)node[left]{\footnotesize An-Chen-Yin's };
\filldraw(0,2.2)node[left]{\footnotesize ill-posedness};
\end{tikzpicture}
\caption{\small\bf Low regularity local well-posedness (LWP) and ill-posedness for 3D Euler equations.}
\label{euler results}
\end{figure}
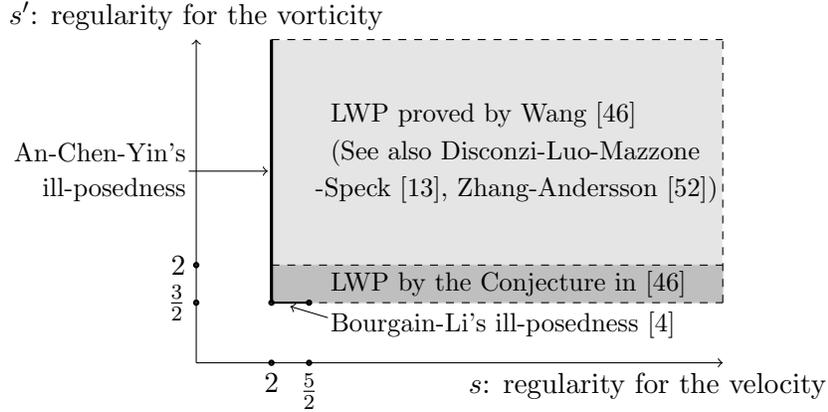

With $H=0$, denoting $\tilde{\Phi}:=(u_1,u_2,u_3,\varrho-1,S)^T(x,t)$, the 3D compressible Euler system \eqref{peuler} reads
\begin{equation} \label{ppeuler}
\partial_t\tilde{\Phi}+A(\tilde{\Phi})\partial_x\tilde{\Phi}=0
\end{equation}
with
\begin{equation} \label{matreuler}
A(\tilde{\Phi})=\begin{pmatrix}
u_1 & 0 & 0& c^2/\varrho & \partial_Sp/\varrho \\
0 & u_1 & 0& 0 & 0 \\
0 & 0  & u_1 & 0 &  0 \\
\varrho & 0 & 0& u_1 &  0 \\
0 & 0 & 0& 0 &  u_1
\end{pmatrix}.
\end{equation}
The eigenvalues of \eqref{matreuler} are
\begin{equation} \label{egveuler}
\lambda_1=u_1+c,\quad\lambda_2=\lambda_3=\lambda_4=u_1,\quad \lambda_5=u_1-c
\end{equation}
with $c=\sqrt{\partial_\varrho p}$ being the sound speed. And we have a \underline{triple} eigenvalue $\lambda_2=\lambda_3=\lambda_4=u_1$. The five linearly independent right eigenvectors are
\begin{equation} \label{rveceuler}
\begin{split}
&r_1=\left(\begin{array}{cc}-1\\0\\0\\-\varrho/c\\0\end{array}\right),
r_2=\left(\begin{array}{cc}0\\1\\0\\0\\0\end{array}\right),
r_3=\left(\begin{array}{cc}0\\0\\1\\0\\0\end{array}\right),
r_4=\left(\begin{array}{cc}0\\0\\0\\-\partial_Sp/c^2\\ 1\end{array}\right),
r_5=\left(\begin{array}{cc}-1\\0\\0\\\varrho/c\\0\end{array}\right).
\end{split}
\end{equation}
The dual left eigenvectors are
\begin{equation}
\begin{split}
&l_1=(-\frac12,0,0,-\frac{c}{2\varrho },-\frac{\partial_Sp}{2\varrho c}),\quad l_2=(0,1,0,0,0),\quad l_3=(0,0,1,0,0),\\
&l_4=(0,0,0,0,1),\quad l_5=(-\frac12,0,0,\frac{c}{2\varrho},\frac{\partial_Sp}{2\varrho c}).
\end{split}
\end{equation}
With these setup, one can then prove the shock formation and the $H^2$ ill-posedness as in Section \ref{appendix}. In particular, with the initial datum of $w_1$ taken as in \eqref{data}, the $H^1$-norms of the velocity $u_1$ and density $\varrho$ will blow up at the shock formation time.

Moreover, the vorticity satisfies
\begin{equation}
    \partial_x u_2=\sum_{k=1}^5 w_k r_{k}^2=w_2,\quad \partial_x u_3=\sum_{k=1}^5 w_k r_{k}^3=w_3.
\end{equation}
Therefore, with the following initial data for $w_i$ (similar to \eqref{data2}):
\begin{equation}
w_i^{(\eta)}(z,0)=\theta^2 \int_\mathbb{R}\zeta_{\frac\eta{10}}(z-y) \chi(y)dy, \quad \text{with}\quad i=2,3,4,5.
\end{equation}
By the smoothness of the test function $\zeta_{\frac\eta{10}}$, we have that the initial datum of $\omega$ is smooth and compactly supported, Hence, it belongs to $H^{s}(\mathbb{R}^3)$ for $s\geq \frac{3}{2}$, which corresponds to the illustration in Figure \ref{euler results}.

\appendix
\section{Proof of Proposition \ref{coeff}} \label{appendixB}
In Proposition \ref{coeff}, we analyze and bound the behaviours of $c^i_{im}$, $\gamma^4_{4m}$ and $\gamma^4_{km}$. In this appendix, we complete the proof of Proposition \ref{coeff} by proving that the coefficients $\gamma^i_{im}$, $\gamma^i_{km}$ with $i \neq 4$ satisfy \eqref{order1} when $H_1=\kappa \neq 0$. This is achieved via a thorough calculation of all the coefficients. As one will see, the formula of certain coefficients can be very complicated. \underline{In spite of these complexities, we show that all these coefficients are of $O(1)$.} This is owing to the careful design of right eigenvectors in \eqref{regv} which \underline{eliminates all potentially} \underline{singular terms}, e.g. ${(C_a^2-C_s^2)}^{-1},{H_2}^{-1},{H_3}^{-1}$.

We begin with calculating the coefficients $\gamma^1_{km}$ as below.
\begin{small}
\begin{align*}
\gamma^1_{12}=&(\lambda_1-\lambda_2)\bigg\{\frac{-C_f}{2(C_f^2-C_s^2)}\bigg[H_3\partial_{H_2}\bigg(\frac{C_a^2}{C_f}-C_f\bigg)-H_2\partial_{H_3}\bigg(\frac{C_a^2}{C_f}-C_f\bigg)\bigg]\\
&-\frac{C_a^2-3C_s^2}{2C_f(C_f^2-C_s^2)}(H_3\partial_{H_2}C_f-H_2\partial_{H_3}C_f)\bigg\},\\
\gamma^1_{13}=&-(\lambda_1-\lambda_3)\bigg\{\frac{C_f}{2(C_f^2-C_s^2)}\bigg[H_2\partial_{H_2}\bigg(\frac{C_a^2}{C_f}-C_f\bigg)+H_3\partial_{H_3}\bigg(\frac{C_a^2}{C_f}-C_f\bigg)-\rho\bigg(\frac{C_a^2}{C_s^2}-1\bigg)\partial_{\rho}\bigg(\frac{C_a^2}{C_f}-C_f\bigg)\\
&-H_2\partial_{H_2}\bigg(\frac{C_a^2}{C_s}-C_s\bigg)-H_3\partial_{H_3}\bigg(\frac{C_a^2}{C_s}-C_s\bigg)+\rho\bigg(\frac{C_a^2}{C_f^2}-1\bigg)\partial_{\rho}\bigg(\frac{C_a^2}{C_s}-C_s\bigg)\bigg]\\
&-\frac{C_f^2C_s^2}{2C_a^2(C_f^2-C_s^2)}\bigg[\bigg(\frac{C_a^2}{C_s^2}-1\bigg)\bigg(\frac{C_a^2}{C_f^2}-1+\rho\partial_\rho \frac{C_a^2}{C_f^2}\bigg)+\frac{2C_a^2H_2}{C_f^3}\partial_{H_2}C_f+\frac{2C_a^2H_3}{C_f^3}\partial_{H_3}C_f\\
&-\bigg(\frac{C_a^2}{C_f^2}-1\bigg)\bigg(\frac{C_a^2}{C_s^2}-1+\rho\partial_\rho \frac{C_a^2}{C_s^2}\bigg)-\frac{2C_a^2H_2}{C_s^3}\partial_{H_2}C_s-\frac{2C_a^2H_3}{C_s^3}\partial_{H_3}C_s\bigg]\\
&+\frac{C_f(C_a^2-C_s^2)}{2(C_f^2-C_s^2)}\bigg[\frac{H_2}{C_f^2}\partial_{H_2}C_f+\frac{H_3}{C_f^2}\partial_{H_3}C_f-\bigg(\frac{1}{C_f}+\frac{\rho\partial_\rho C_f}{C_f^2}\bigg)\bigg(\frac{C_a^2}{C_s^2}-1\bigg)-\frac{1}{C_f}\\
&-\frac{H_2}{C_s^2}\partial_{H_2}C_s-\frac{H_3}{C_s^2}\partial_{H_3}C_s+\bigg(\frac{1}{C_s}+\frac{\rho\partial_\rho C_s}{C_s^2}\bigg)\bigg(\frac{C_a^2}{C_f^2}-1\bigg)+\frac{1}{C_s}\bigg]\bigg\},\\
\gamma^1_{14}=&-(\lambda_1-\lambda_4)\bigg\{\frac{C_f}{2(C_f^2-C_s^2)}\bigg[\frac{\rho}{\gamma}\partial_\rho \bigg(\frac{C_a^2}{C_f}-C_f\bigg)-\partial_S \bigg(\frac{C_a^2}{C_f}-C_f\bigg)\bigg]+\frac{C_f^2 C_s^2}{2C_a^2(C_f^2-C_s^2)}\bigg(\frac{\rho}{\gamma}\partial_\rho \frac{C_a^2}{C_f^2}+\frac{2C_a^2}{C_f^3}\partial_S C_f\bigg)\\
&+\frac{C_a^2-C_s^2}{2(C_f^2-C_s^2)}\bigg(\frac{\rho}{\gamma C_f}\partial_\rho C_f-\frac{\partial_S C_f}{C_f}+\frac{1}{\gamma}\bigg)\bigg\},
\end{align*}

\newpage
\begin{align*}
\gamma^1_{15}=&-(\lambda_1-\lambda_5)\bigg\{\frac{C_f}{2(C_f^2-C_s^2)}\bigg[\rho\bigg(\frac{C_a^2}{C_s^2}-1\bigg)\partial_{\rho}\bigg(\frac{C_a^2}{C_f}-C_f\bigg)-H_2\partial_{H_2}\bigg(\frac{C_a^2}{C_f}-C_f\bigg)-H_3\partial_{H_3}\bigg(\frac{C_a^2}{C_f}-C_f\bigg)\\
&-H_2\partial_{H_2}\bigg(\frac{C_a^2}{C_s}-C_s\bigg)-H_3\partial_{H_3}\bigg(\frac{C_a^2}{C_s}-C_s\bigg)+\rho\bigg(\frac{C_a^2}{C_f^2}-1\bigg)\partial_{\rho}\bigg(\frac{C_a^2}{C_s}-C_s\bigg)\bigg]\\
&+\frac{C_f^2C_s^2}{2C_a^2(C_f^2-C_s^2)}\bigg[\bigg(\frac{C_a^2}{C_s^2}-1\bigg)\bigg(\frac{C_a^2}{C_f^2}-1+\rho\partial_\rho \frac{C_a^2}{C_f^2}\bigg)+\frac{2C_a^2H_2}{C_f^3}\partial_{H_2}C_f+\frac{2C_a^2H_3}{C_f^3}\partial_{H_3}C_f\\
&-\bigg(\frac{C_a^2}{C_f^2}-1\bigg)\bigg(\frac{C_a^2}{C_s^2}-1+\rho\partial_\rho \frac{C_a^2}{C_s^2}\bigg)-\frac{2C_a^2H_2}{C_s^3}\partial_{H_2}C_s-\frac{2C_a^2H_3}{C_s^3}\partial_{H_3}C_s\bigg]\\
&+\frac{C_f(C_a^2-C_s^2)}{2(C_f^2-C_s^2)}\bigg[C_f\bigg(\frac{1}{C_f}+\frac{\rho\partial_\rho C_f}{C_f^2}\bigg)\bigg(\frac{C_a^2}{C_s^2}-1\bigg)+\frac{1}{C_f}-\frac{H_2}{C_f^2}\partial_{H_2}C_f-\frac{H_3}{C_f^2}\partial_{H_3}\\
&-\frac{H_2}{C_s^2}\partial_{H_2}C_s-\frac{H_3}{C_s^2}\partial_{H_3}C_s+\bigg(\frac{1}{C_s}+\frac{\rho\partial_\rho C_s}{C_s^2}\bigg)\bigg(\frac{C_a^2}{C_f^2}-1\bigg)+\frac{1}{C_s}\bigg]\bigg\},\\
\gamma^1_{16}=&(\lambda_1-\lambda_6)\bigg\{\frac{C_f}{2(C_f^2-C_s^2)}\bigg[H_3\partial_{H_2}\bigg(\frac{C_a^2}{C_f}-C_f\bigg)-H_2\partial_{H_3}\bigg(\frac{C_a^2}{C_f}-C_f\bigg)\bigg]\\
&+\frac{C_a^2-3C_s^2}{2C_f(C_f^2-C_s^2)}(H_3\partial_{H_2}C_f-H_2\partial_{H_3}C_f)\bigg\},\\
\gamma^1_{17}=&-\frac{(\lambda_1-\lambda_7)(C_a^2-C_s^2)}{C_f^2-C_s^2}\bigg[\frac{\rho(C_a^2-C_f^2)}{C_f^3}\partial_\rho C_f+\frac{C_a^2}{C_f^2}-\frac{H_2}{C_f}\partial_{H_2}C_f-\frac{H_3}{C_f}\partial_{H_3}C_f\bigg],\\
\gamma^1_{23}=&0,\ \gamma^1_{24}=0,\ \gamma^1_{25}=0,\ \gamma^1_{26}=\frac{(\lambda_2-\lambda_6)C_f(C_f+C_a)(C_a^2-C_s^2)}{2C_a^2(C_f^2-C_s^2)},\ \gamma^1_{27}=0,\\
\gamma^1_{32}=&-(\lambda_3-\lambda_2)\bigg\{\frac{C_f}{2(C_f^2-C_s^2)}\bigg[H_3\partial_{H_2}\bigg(\frac{C_a^2}{C_s}-C_s\bigg)-H_2\partial_{H_3}\bigg(\frac{C_a^2}{C_s}-C_s\bigg)\bigg]\\
&-\frac{C_f^2}{C_s(C_f^2-C_s^2)}(H_3\partial_{H_2}C_s-H_2\partial_{H_3}C_s)+\frac{C_f(C_a^2-C_s^2)}{2C_s^2(C_f^2-C_s^2)} (H_3\partial_{H_2}C_s-H_2\partial_{H_3}C_s)\bigg\},\\
\gamma^1_{34}=&-(\lambda_3-\lambda_4)\bigg\{\frac{C_f}{2(C_f^2-C_s^2)}\bigg[\frac{\rho}{\gamma}\partial_\rho\bigg(\frac{C_a^2}{C_s}-C_s\bigg)-\partial_S\bigg(\frac{C_a^2}{C_s}-C_s\bigg)\bigg]+\frac{C_f(C_a^2-C_s^2)}{2C_s^2(C_f^2-C_s^2)}\bigg(\rho\partial_\rho C_s-\partial_S C_s+\frac{C_s}{\gamma}\bigg)\\
&+\frac{C_f^2C_s^2}{2C_a^2(C_f^2-C_s^2)}\bigg[\frac{2C_a^2}{C_s^3}\partial_S C_s+\frac{1}{\gamma}\bigg(\frac{C_a^2}{C_s^2}-1+\rho\partial_\rho\frac{C_a^2}{C_s^2}\bigg)\bigg]\bigg\},
\end{align*}
\newpage
\begin{align*}
\gamma^1_{35}=&-(\lambda_3-\lambda_5)\bigg\{\frac{C_f}{2(C_f^2-C_s^2)}\bigg[\rho\bigg(\frac{C_a^2}{C_s^2}-1\bigg)\partial_{\rho}\bigg(\frac{C_a^2}{C_s}-C_s\bigg)-H_2\partial_{H_2}\bigg(\frac{C_a^2}{C_s}-C_s\bigg)-H_3\partial_{H_3}\bigg(\frac{C_a^2}{C_s}-C_s\bigg)\bigg]\\
&+\frac{C_f^2C_s^2}{2\rho C_a^2(C_f^2-C_s^2)}\bigg[\frac{1}{2}\partial_\rho \bigg(\rho\big(\frac{C_a^2}{C_s^2}-1\big)\bigg)^2+\frac{2\rho C_a^2H_2}{C_s^3}\partial_{H_2}C_s+\frac{2\rho C_a^2H_3}{C_s^3}\partial_{H_3}C_s\bigg]\\
&+\frac{C_f(C_a^2-C_s^2)}{2(C_f^2-C_s^2)}\bigg[\frac{\rho\partial_\rho C_s}{C_s^2}\bigg(\frac{C_a^2}{C_s^2}-1\bigg)+\frac{C_a^2}{C_s^3}-\frac{H_2}{C_s^2}\partial_{H_2}C_s-\frac{H_3}{C_s^2}\partial_{H_3}C_s\bigg]+\frac{C_f^2(C_a^2-C_s^2)}{2C_a^2(C_f^2-C_s^2)}\bigg\},\\
\gamma^1_{36}=&(\lambda_3-\lambda_6)\bigg\{\frac{C_f}{2(C_f^2-C_s^2)}\bigg[H_3\partial_{H_2}\bigg(\frac{C_a^2}{C_s}-C_s\bigg)-H_2\partial_{H_3}\bigg(\frac{C_a^2}{C_s}-C_s\bigg)\bigg]\\
&-\frac{C_f^2}{C_s(C_f^2-C_s^2)}(H_3\partial_{H_2}C_s-H_2\partial_{H_3}C_s)+\frac{C_f(C_a^2-C_s^2)}{2C_s^2(C_f^2-C_s^2)} (H_3\partial_{H_2}C_s-H_2\partial_{H_3}C_s)\bigg\},\\
\gamma^1_{37}=&-(\lambda_3-\lambda_7)\bigg\{\frac{C_f}{2(C_f^2-C_s^2)}\bigg[\rho\bigg(\frac{C_a^2}{C_f^2}-1\bigg)\partial_{\rho}\bigg(\frac{C_a^2}{C_s}-C_s\bigg)-H_2\partial_{H_2}\bigg(\frac{C_a^2}{C_s}-C_s\bigg)-H_3\partial_{H_3}\bigg(\frac{C_a^2}{C_s}-C_s\bigg)\bigg]\\
&+\frac{C_f^2C_s^2}{2C_a^2(C_f^2-C_s^2)}\bigg[\bigg(\frac{C_a^2}{C_f^2}-1\bigg)\bigg(\frac{C_a^2}{C_s^2}-1+\rho\partial_\rho \frac{C_a^2}{C_s^2}\bigg)+\frac{2C_a^2H_2}{C_s^3}\partial_{H_2}C_s+\frac{2C_a^2H_3}{C_s^3}\partial_{H_3}C_s\bigg]\\
&+\frac{C_f(C_a^2-C_s^2)}{2(C_f^2-C_s^2)}\bigg[\bigg(\frac{1}{C_s}+\frac{\rho\partial_\rho C_s}{C_s^2}\bigg)\bigg(\frac{C_a^2}{C_f^2}-1\bigg)+\frac{1}{C_s}-\frac{H_2}{C_s^2}\partial_{H_2}C_s-\frac{H_3}{C_s^2}\partial_{H_3}C_s\bigg]-\frac{C_f^2(C_a^2-C_s^2)}{2C_a^2(C_f^2-C_s^2)}\bigg\},\\
\gamma^1_{42}=&0,\ \gamma^1_{43}=-\frac{(\lambda_4-\lambda_3)C_f^2C_s^2}{2\gamma C_a^2(C_f^2-C_s^2)}\bigg(\frac{C_a^2}{C_s^2}-1\bigg),\ \gamma^1_{45}=-\frac{(\lambda_4-\lambda_3)C_f^2C_s^2}{2\gamma C_a^2(C_f^2-C_s^2)}\bigg(\frac{C_a^2}{C_s^2}-1\bigg),\ \gamma^1_{46}=0,\\
\gamma^1_{47}=&\frac{(\lambda_4-\lambda_7)C_f^2C_s^2}{2\gamma C_a^2(C_f^2-C_s^2)}\bigg(\frac{C_a^2}{C_f^2}-1\bigg),\\
\gamma^1_{52}=&-(\lambda_5-\lambda_2)\bigg\{\frac{C_f}{2(C_f^2-C_s^2)}\bigg[H_3\partial_{H_2}\bigg(\frac{C_a^2}{C_s}-C_s\bigg)-H_2\partial_{H_3}\bigg(\frac{C_a^2}{C_s}-C_s\bigg)\bigg]\\
&+\frac{C_f^2}{C_s(C_f^2-C_s^2)}(H_3\partial_{H_2}C_s-H_2\partial_{H_3}C_s)+\frac{C_f(C_a^2-C_s^2)}{2C_s^2(C_f^2-C_s^2)} (H_3\partial_{H_2}C_s-H_2\partial_{H_3}C_s)\bigg\},\\
\gamma^1_{53}=&-(\lambda_5-\lambda_3)\bigg\{\frac{C_f}{2(C_f^2-C_s^2)}\bigg[\rho\bigg(\frac{C_a^2}{C_s^2}-1\bigg)\partial_{\rho}\bigg(\frac{C_a^2}{C_s}-C_s\bigg)-H_2\partial_{H_2}\bigg(\frac{C_a^2}{C_s}-C_s\bigg)-H_3\partial_{H_3}\bigg(\frac{C_a^2}{C_s}-C_s\bigg)\bigg]\\
&+\frac{C_f^2C_s^2}{2\rho C_a^2(C_f^2-C_s^2)}\bigg[\frac{1}{2}\partial_\rho \bigg(\rho\big(\frac{C_a^2}{C_s^2}-1\big)\bigg)^2+\frac{2\rho C_a^2H_2}{C_s^3}\partial_{H_2}C_s+\frac{2\rho C_a^2H_3}{C_s^3}\partial_{H_3}C_s\bigg]\\
&-\frac{C_f(C_a^2-C_s^2)}{2(C_f^2-C_s^2)}\bigg[\frac{\rho\partial_\rho C_s}{C_s^2}\bigg(\frac{C_a^2}{C_s^2}-1\bigg)+\frac{C_a^2}{C_s^3}-\frac{H_2}{C_s^2}\partial_{H_2}C_s-\frac{H_3}{C_s^2}\partial_{H_3}C_s\bigg]+\frac{C_f^2(C_a^2-C_s^2)}{2C_a^2(C_f^2-C_s^2)}\bigg\},
\end{align*}
\newpage
\begin{align*}
\gamma^1_{54}=&-(\lambda_5-\lambda_4)\bigg\{\frac{C_f}{2(C_f^2-C_s^2)}\bigg[\frac{\rho}{\gamma}\partial_\rho\bigg(\frac{C_a^2}{C_s}-C_s\bigg)-\partial_S\bigg(\frac{C_a^2}{C_s}-C_s\bigg)\bigg]+\frac{C_f(C_a^2-C_s^2)}{2C_s^2(C_f^2-C_s^2)}\bigg(\rho\partial_\rho C_s-\partial_S C_s+\frac{C_s}{\gamma}\bigg)\\
&-\frac{C_f^2C_s^2}{2C_a^2(C_f^2-C_s^2)}\bigg[\frac{2C_a^2}{C_s^3}\partial_S C_s+\frac{1}{\gamma}\bigg(\frac{C_a^2}{C_s^2}-1+\rho\partial_\rho\frac{C_a^2}{C_s^2}\bigg)\bigg]\bigg\},\\
\gamma^1_{56}=&(\lambda_5-\lambda_6)\bigg\{\frac{C_f}{2(C_f^2-C_s^2)}\bigg[H_3\partial_{H_2}\bigg(\frac{C_a^2}{C_s}-C_s\bigg)-H_2\partial_{H_3}\bigg(\frac{C_a^2}{C_s}-C_s\bigg)\bigg]+\frac{C_f^2}{C_s(C_f^2-C_s^2)}(H_3\partial_{H_2}C_s-H_2\partial_{H_3}C_s)\\
&+\frac{C_f(C_a^2-C_s^2)}{2C_s^2(C_f^2-C_s^2)} (H_3\partial_{H_2}C_s-H_2\partial_{H_3}C_s)\bigg\},\\
\gamma^1_{57}=&-(\lambda_5-\lambda_7)\bigg\{\frac{C_f}{2(C_f^2-C_s^2)}\bigg[\rho\bigg(\frac{C_a^2}{C_f^2}-1\bigg)\partial_{\rho}\bigg(\frac{C_a^2}{C_s}-C_s\bigg)-H_2\partial_{H_2}\bigg(\frac{C_a^2}{C_s}-C_s\bigg)-H_3\partial_{H_3}\bigg(\frac{C_a^2}{C_s}-C_s\bigg)\bigg]\\
&-\frac{C_f^2C_s^2}{2C_a^2(C_f^2-C_s^2)}\bigg[\bigg(\frac{C_a^2}{C_f^2}-1\bigg)\bigg(\frac{C_a^2}{C_s^2}-1+\rho\partial_\rho \frac{C_a^2}{C_s^2}\bigg)+\frac{2C_a^2H_2}{C_s^3}\partial_{H_2}C_s+\frac{2C_a^2H_3}{C_s^3}\partial_{H_3}C_s\bigg]\\
&+\frac{C_f(C_a^2-C_s^2)}{2(C_f^2-C_s^2)}\bigg[\bigg(\frac{1}{C_s}+\frac{\rho\partial_\rho C_s}{C_s^2}\bigg)\bigg(\frac{C_a^2}{C_f^2}-1\bigg)+\frac{1}{C_s}-\frac{H_2}{C_s^2}\partial_{H_2}C_s-\frac{H_3}{C_s^2}\partial_{H_3}C_s\bigg]+\frac{C_f^2(C_a^2-C_s^2)}{2C_a^2(C_f^2-C_s^2)}\bigg\},\\
\gamma^1_{62}=&\frac{(\lambda_6-\lambda_2)C_f(C_f-C_a)(C_a^2-C_s)}{2C_a^2(C_f^2-C_s^2)},\ \gamma^1_{63}=0,\ \gamma^1_{64}=0,\ \gamma^1_{65}=0,\ \gamma^1_{67}=0,\\
\gamma^1_{72}=&-(\lambda_7-\lambda_2)\bigg\{\frac{C_f}{2(C_f^2-C_s^2)}\bigg[H_3\partial_{H_2}\bigg(\frac{C_a^2}{C_f}-C_f\bigg)-H_2\partial_{H_3}\bigg(\frac{C_a^2}{C_f}-C_f\bigg)\bigg]\\
&+\frac{C_a^2+C_s^2}{2C_f(C_f^2-C_s^2)}(H_3\partial_{H_2}C_f-H_2\partial_{H_3}C_f)\bigg\},\\
\gamma^1_{73}=&-(\lambda_7-\lambda_3)\bigg\{\frac{C_f}{2(C_f^2-C_s^2)}\bigg[H_2\partial_{H_2}\bigg(\frac{C_a^2}{C_f}-C_f\bigg)+H_3\partial_{H_3}\bigg(\frac{C_a^2}{C_f}-C_f\bigg)-\rho\bigg(\frac{C_a^2}{C_s^2}-1\bigg)\partial_{\rho}\bigg(\frac{C_a^2}{C_f}-C_f\bigg)\bigg]\\
&+\frac{C_f^2C_s^2}{2C_a^2(C_f^2-C_s^2)}\bigg[\bigg(\frac{C_a^2}{C_s^2}-1\bigg)\bigg(\frac{C_a^2}{C_f^2}-1+\rho\partial_\rho \frac{C_a^2}{C_f^2}\bigg)+\frac{2C_a^2H_2}{C_f^3}\partial_{H_2}C_f+\frac{2C_a^2H_3}{C_f^3}\partial_{H_3}C_f\bigg]\\
&+\frac{C_f(C_a^2-C_s^2)}{2(C_f^2-C_s^2)}\bigg[\frac{H_2}{C_f^2}\partial_{H_2}C_f+\frac{H_3}{C_f^2}\partial_{H_3}C_f-\bigg(\frac{1}{C_f}+\frac{\rho\partial_\rho C_f}{C_f^2}\bigg)\bigg(\frac{C_a^2}{C_s^2}-1\bigg)-\frac{1}{C_f}\bigg]+\frac{C_f^2(C_a^2-C_s^2)}{2C_a^2(C_f^2-C_s^2)}\bigg\},\\
\gamma^1_{74}=&-(\lambda_7-\lambda_4)\bigg\{\frac{C_f}{2(C_f^2-C_s^2)}\bigg[\frac{\rho}{\gamma}\partial_\rho \bigg(\frac{C_a^2}{C_f}-C_f\bigg)-\partial_S \bigg(\frac{C_a^2}{C_f}-C_f\bigg)\bigg]\\
&-\frac{C_f^2 C_s^2}{2 C_a^2(C_f^2-C_s^2)}\bigg[\frac{1}{\gamma}\partial_\rho \rho\bigg(\frac{C_a^2}{C_f^2}-1\bigg)+\frac{2C_a^2}{C_f^3}\partial_S C_f\bigg]+\frac{C_a^2-C_s^2}{2(C_f^2-C_s^2)}\bigg(\frac{\rho}{\gamma C_f}\partial_\rho C_f-\frac{\partial_S C_f}{C_f}+\frac{1}{\gamma}\bigg)\bigg\},
\end{align*}
\newpage
\begin{align*}
\gamma^1_{75}=&(\lambda_7-\lambda_5)\bigg\{\frac{C_f}{2(C_f^2-C_s^2)}\bigg[H_2\partial_{H_2}\bigg(\frac{C_a^2}{C_f}-C_f\bigg)+H_3\partial_{H_3}\bigg(\frac{C_a^2}{C_f}-C_f\bigg)-\rho\bigg(\frac{C_a^2}{C_s^2}-1\bigg)\partial_{\rho}\bigg(\frac{C_a^2}{C_f}-C_f\bigg)\bigg]\\
&+\frac{C_f^2C_s^2}{2C_a^2(C_f^2-C_s^2)}\bigg[\bigg(\frac{C_a^2}{C_s^2}-1\bigg)\bigg(\frac{C_a^2}{C_f^2}-1+\rho\partial_\rho \frac{C_a^2}{C_f^2}\bigg)+\frac{2C_a^2H_2}{C_f^3}\partial_{H_2}C_f+\frac{2C_a^2H_3}{C_f^3}\partial_{H_3}C_f\bigg]\\
&+\frac{C_f(C_a^2-C_s^2)}{2(C_f^2-C_s^2)}\bigg[\frac{H_2}{C_f^2}\partial_{H_2}C_f+\frac{H_3}{C_f^2}\partial_{H_3}C_f-\bigg(\frac{1}{C_f}+\frac{\rho\partial_\rho C_f}{C_f^2}\bigg)\bigg(\frac{C_a^2}{C_s^2}-1\bigg)-\frac{1}{C_f}\bigg]+\frac{C_f^2(C_a^2-C_s^2)}{2C_a^2(C_f^2-C_s^2)}\bigg\},\\
\gamma^1_{76}=&(\lambda_7-\lambda_6)\bigg\{\frac{C_f}{2(C_f^2-C_s^2)}\bigg[H_3\partial_{H_2}\bigg(\frac{C_a^2}{C_f}-C_f\bigg)-H_2\partial_{H_3}\bigg(\frac{C_a^2}{C_f}-C_f\bigg)\bigg]\\
&+\frac{C_a^2+C_s^2}{2C_f(C_f^2-C_s^2)}(H_3\partial_{H_2}C_f-H_2\partial_{H_3}C_f)\bigg\}.
\end{align*}
\end{small}
As one can see, the coefficients $\gamma^1_{km}$ have very complex forms. Nevertheless, thanks to the proper right eigenvectors, there is no singular factors like ${(C_a^2-C_s^2)}^{-1},{H_2}^{-1},{H_3}^{-1}$. Hence, every term in the formula is a simple function of $\Phi$ (or its derivatives), which is clearly of $O(1)$. Moreover, the potentially singular factor $\frac{1}{H_2^2+H_3^2}$ in the left eigenvectors, as we explained in Proposition \ref{coeff}, can always be \underline{cancelled} by a numerator which is either $H_2^2+H_3^2$ or $0$. Then, the combination of these $O(1)$ terms is also regular.

For $i=2$, we obtain
\begin{small}
\begin{align*}
\gamma^2_{21}=&\frac{\lambda_2-\lambda_1}{4}\bigg(\frac{C_a}{C_f}-1\bigg)^2,\ \gamma^2_{23}=\frac{(\lambda_2-\lambda_3)}{4}\bigg
(\frac{C_a}{C_s}-1\bigg)^2,\ \gamma^2_{24}=-\frac{\lambda_2-\lambda_4}{4\gamma},\\
\gamma^2_{25}=&-\frac{(\lambda_2-\lambda_5)}{4}\bigg(\frac{C_a}{C_s}+1\bigg)^2,\ \gamma^2_{26}=0,\ \gamma^2_{27}=-\frac{(\lambda_2-\lambda_7)}{4}\bigg(\frac{C_a}{C_f}+1\bigg)^2,\\
\gamma^2_{13}=&0,\ \gamma^2_{14}=0,\ \gamma^2_{15}=0,\ \gamma^2_{16}=\frac{(\lambda_1-\lambda_6)(C_a+C_f)}{2C_f},\ \gamma^2_{17}=0,\\
\gamma^2_{31}=&0,\ \gamma^2_{34}=0,\ \gamma^2_{35}=0,\ \gamma^2_{36}=-\frac{(\lambda_3-\lambda_6)(C_a+C_s)}{2C_s},\ \gamma^2_{37}=0,\\
\gamma^2_{41}=&0,\ \gamma^2_{43}=0,\ \gamma^2_{45}=0,\ \gamma^2_{46}=0,\ \gamma^2_{47}=0,\\
\gamma^2_{51}=&0,\ \gamma^2_{53}=0,\ \gamma^2_{54}=0,\ \gamma^2_{56}=-\frac{(\lambda_5-\lambda_6)(C_a-C_s)}{2C_s},\ \gamma^2_{57}=0,\\
\gamma^2_{61}=&\frac{\lambda_6-\lambda_1}{4}\bigg(\frac{C_a^2}{C_f^2}+1\bigg),\ \gamma^2_{63}=-\frac{(\lambda_6-\lambda_3)(C_a^2-C_s^2)}{4C_s^2},\ \gamma^2_{64}=-\frac{\lambda_6-\lambda_4}{4\gamma},\\
\gamma^2_{65}=&\frac{(\lambda_6-\lambda_5)(C_a^2-C_s^2)}{4C_s^2},\ \gamma^2_{67}=-\frac{\lambda_6-\lambda_7}{4}\bigg(\frac{C_a^2}{C_f^2}+1\bigg),\\
\gamma^2_{71}=&0,\ \gamma^2_{73}=0,\ \gamma^2_{74}=0,\ \gamma^2_{75}=0,\ \gamma^2_{76}=\frac{(\lambda_7-\lambda_6)(C_f-C_a)}{2C_f}.
\end{align*}
\end{small}
One can easily see that $\gamma^2_{km}$ are all uniformly bounded.

Next, we calculate $\gamma^3_{km}$.
\begin{small}
\begin{align*}
\gamma^3_{31}=&-(\lambda_3-\lambda_1)\bigg\{\frac{C_s}{2(C_f^2-C_s^2)}\bigg[H_2\partial_{H_2}\bigg(\frac{C_a^2}{C_f}-C_f\bigg)+H_3\partial_{H_3}\bigg(\frac{C_a^2}{C_f}-C_f\bigg)-\rho\bigg(\frac{C_a^2}{C_s^2}-1\bigg)\partial_{\rho}\bigg(\frac{C_a^2}{C_f}-C_f\bigg)\\
&-H_2\partial_{H_2}\bigg(\frac{C_a^2}{C_s}-C_s\bigg)-H_3\partial_{H_3}\bigg(\frac{C_a^2}{C_s}-C_s\bigg)+\rho\bigg(\frac{C_a^2}{C_f^2}-1\bigg)\partial_{\rho}\bigg(\frac{C_a^2}{C_s}-C_s\bigg)\bigg]\\
&-\frac{C_f^2C_s^2}{2C_a^2(C_f^2-C_s^2)}\bigg[\bigg(\frac{C_a^2}{C_s^2}-1\bigg)\bigg(\frac{C_a^2}{C_f^2}-1+\rho\partial_\rho \frac{C_a^2}{C_f^2}\bigg)+\frac{2C_a^2H_2}{C_f^3}\partial_{H_2}C_f+\frac{2C_a^2H_3}{C_f^3}\partial_{H_3}C_f\\
&-\bigg(\frac{C_a^2}{C_f^2}-1\bigg)\bigg(\frac{C_a^2}{C_s^2}-1+\rho\partial_\rho \frac{C_a^2}{C_s^2}\bigg)-\frac{2C_a^2H_2}{C_s^3}\partial_{H_2}C_s-\frac{2C_a^2H_3}{C_s^3}\partial_{H_3}C_s\bigg]\\
&+\frac{C_s(C_a^2-C_s^2)}{2(C_f^2-C_s^2)}\bigg[\frac{H_2}{C_f^2}\partial_{H_2}C_f+\frac{H_3}{C_f^2}\partial_{H_3}C_f-\bigg(\frac{1}{C_f}+\frac{\rho\partial_\rho C_f}{C_f^2}\bigg)\bigg(\frac{C_a^2}{C_s^2}-1\bigg)-\frac{1}{C_f}\\
&-\frac{H_2}{C_s^2}\partial_{H_2}C_s-\frac{H_3}{C_s^2}\partial_{H_3}C_s+\bigg(\frac{1}{C_s}+\frac{\rho\partial_\rho C_s}{C_s^2}\bigg)\bigg(\frac{C_a^2}{C_f^2}-1\bigg)+\frac{1}{C_s}\bigg]\bigg\},\\
\gamma^3_{32}=&(\lambda_3-\lambda_2)\bigg\{\frac{C_s}{2(C_f^2-C_s^2)}\bigg[H_3\partial_{H_2}\bigg(\frac{C_a^2}{C_s}-C_s\bigg)-H_2\partial_{H_3}\bigg(\frac{C_a^2}{C_s}-C_s\bigg)\bigg]\\
&+\frac{3C_f^2-C_a^2}{2C_s(C_f^2-C_s^2)}(H_3\partial_{H_2}C_s-H_2\partial_{H_3}C_s)\bigg\},\\
\gamma^3_{34}=&-(\lambda_3-\lambda_4)\bigg\{\frac{C_s}{2(C_f^2-C_s^2)}\bigg[-\frac{\rho}{\gamma}\partial_\rho \bigg(\frac{C_a^2}{C_s}-C_s\bigg)+\partial_S \bigg(\frac{C_a^2}{C_s}-C_s\bigg)\bigg]\\
&-\frac{C_f^2 C_s^2}{2C_a^2(C_f^2-C_s^2)}\bigg(\frac{\rho}{\gamma}\partial_\rho \frac{C_a^2}{C_s^2}+\frac{2C_a^2}{C_s^3}\partial_S C_s\bigg)-\frac{C_a^2-C_f^2}{2(C_f^2-C_s^2)}\bigg(\frac{\rho}{\gamma C_s}\partial_\rho C_s-\frac{\partial_S C_s}{C_s}+\frac{1}{\gamma}\bigg)\bigg\},\\
\gamma^3_{35}=&\frac{(\lambda_3-\lambda_5)(C_a^2-C_f^2)}{C_f^2-C_s^2}\bigg[\frac{\rho(C_a^2-C_s^2)}{C_s^3}\partial_\rho C_s+\frac{C_a^2}{C_s^2}-\frac{H_2}{C_s}\partial_{H_2}C_s-\frac{H_3}{C_s}\partial_{H_3}C_s\bigg],\\
\gamma^3_{36}=&-(\lambda_3-\lambda_6)\bigg\{\frac{C_s}{2(C_f^2-C_s^2)}\bigg[H_3\partial_{H_2}\bigg(\frac{C_a^2}{C_s}-C_s\bigg)-H_2\partial_{H_3}\bigg(\frac{C_a^2}{C_s}-C_s\bigg)\bigg]\\
&+\frac{3C_f^2-C_a^2}{2C_s(C_f^2-C_s^2)}(H_3\partial_{H_2}C_s-H_2\partial_{H_3}C_s)\bigg\},
\end{align*}
\newpage
\begin{align*}
\gamma^3_{37}=&-(\lambda_3-\lambda_7)\bigg\{\frac{C_s}{2(C_f^2-C_s^2)}\bigg[H_2\partial_{H_2}\bigg(\frac{C_a^2}{C_f}-C_f\bigg)+H_3\partial_{H_3}\bigg(\frac{C_a^2}{C_f}-C_f\bigg)-\rho\bigg(\frac{C_a^2}{C_s^2}-1\bigg)\partial_{\rho}\bigg(\frac{C_a^2}{C_f}-C_f\bigg)\\
&+H_2\partial_{H_2}\bigg(\frac{C_a^2}{C_s}-C_s\bigg)+H_3\partial_{H_3}\bigg(\frac{C_a^2}{C_s}-C_s\bigg)-\rho\bigg(\frac{C_a^2}{C_f^2}-1\bigg)\partial_{\rho}\bigg(\frac{C_a^2}{C_s}-C_s\bigg)\bigg]\\
&+\frac{C_f^2C_s^2}{2C_a^2(C_f^2-C_s^2)}\bigg[\bigg(\frac{C_a^2}{C_s^2}-1\bigg)\bigg(\frac{C_a^2}{C_f^2}-1+\rho\partial_\rho \frac{C_a^2}{C_f^2}\bigg)+\frac{2C_a^2H_2}{C_f^3}\partial_{H_2}C_f+\frac{2C_a^2H_3}{C_f^3}\partial_{H_3}C_f\\
&-\bigg(\frac{C_a^2}{C_f^2}-1\bigg)\bigg(\frac{C_a^2}{C_s^2}-1+\rho\partial_\rho \frac{C_a^2}{C_s^2}\bigg)-\frac{2C_a^2H_2}{C_s^3}\partial_{H_2}C_s-\frac{2C_a^2H_3}{C_s^3}\partial_{H_3}C_s\bigg]\\
&+\frac{C_s(C_a^2-C_s^2)}{2(C_f^2-C_s^2)}\bigg[\frac{H_2}{C_f^2}\partial_{H_2}C_f+\frac{H_3}{C_f^2}\partial_{H_3}C_f-\bigg(\frac{1}{C_f}+\frac{\rho\partial_\rho C_f}{C_f^2}\bigg)\bigg(\frac{C_a^2}{C_s^2}-1\bigg)-\frac{1}{C_f}\\
&+\frac{H_2}{C_s^2}\partial_{H_2}C_s+\frac{H_3}{C_s^2}\partial_{H_3}C_s-\bigg(\frac{1}{C_s}+\frac{\rho\partial_\rho C_s}{C_s^2}\bigg)\bigg(\frac{C_a^2}{C_f^2}-1\bigg)-\frac{1}{C_s}\bigg]\bigg\},\\
\gamma^3_{12}=&-(\lambda_1-\lambda_2)\bigg\{\frac{-C_s}{2(C_f^2-C_s^2)}\bigg[H_3\partial_{H_2}\bigg(\frac{C_a^2}{C_f}-C_f\bigg)-H_2\partial_{H_3}\bigg(\frac{C_a^2}{C_f}-C_f\bigg)\bigg]\\
&+\bigg[\frac{C_s^2}{(C_f^2-C_s^2)C_f}-\frac{C_s(C_a^2-C_f^2)}{2C_f^2(C_f^2-C_s^2)}\bigg](H_3\partial_{H_2}C_s-H_2\partial_{H_3}C_s)\bigg\},\\
\gamma^3_{14}=&(\lambda_1-\lambda_4)\bigg\{\frac{C_s}{2(C_f^2-C_s^2)}\bigg[\frac{\rho}{\gamma}\partial_\rho \bigg(\frac{C_a^2}{C_f}-C_f\bigg)-\partial_S \bigg(\frac{C_a^2}{C_f}-C_f\bigg)\bigg]\\
&+\frac{C_f^2 C_s^2}{2C_a^2(C_f^2-C_s^2)}\bigg[\frac{\rho}{\gamma}\partial_\rho \frac{C_a^2}{C_f^2}+\frac{2C_a^2}{C_f^3}\partial_S C_f+\frac{1}{\gamma}\bigg(\frac{C_a^2}{C_f^2}-1\bigg)\bigg]\\
&+\frac{C_s(C_a^2-C_s^2)}{2(C_f^2-C_s^2)}\bigg(\frac{\rho}{\gamma C_f^2}\partial_\rho C_f-\frac{\partial_S C_f}{C_f^2}+\frac{1}{\gamma C_f}\bigg)\bigg\},\\
\gamma^3_{15}=&-(\lambda_1-\lambda_5)\bigg\{\frac{C_s}{2(C_f^2-C_s^2)}\bigg[\rho\bigg(\frac{C_a^2}{C_s^2}-1\bigg)\partial_{\rho}\bigg(\frac{C_a^2}{C_f}-C_f\bigg)-H_2\partial_{H_2}\bigg(\frac{C_a^2}{C_f}-C_f\bigg)-H_3\partial_{H_3}\bigg(\frac{C_a^2}{C_f}-C_f\bigg)\bigg]\\
&+\frac{C_f^2C_s^2}{2C_a^2(C_f^2-C_s^2)}\bigg[\bigg(\frac{C_a^2}{C_s^2}-1\bigg)\bigg(\frac{C_a^2}{C_f^2}-1+\rho\partial_\rho \frac{C_a^2}{C_f^2}\bigg)+\frac{2C_a^2H_2}{C_f^3}\partial_{H_2}C_f+\frac{2C_a^2H_3}{C_f^3}\partial_{H_3}C_f\bigg]\\
&-\frac{C_s(C_a^2-C_s^2)}{2(C_f^2-C_s^2)}\bigg[\frac{H_2}{C_f^2}\partial_{H_2}C_f+\frac{H_3}{C_f^2}\partial_{H_3}C_f-\bigg(\frac{1}{C_f}+\frac{\rho\partial_\rho C_f}{C_f^2}\bigg)\bigg(\frac{C_a^2}{C_s^2}-1\bigg)-\frac{1}{C_f}\bigg]-\frac{C_s^2(C_a^2-C_f^2)}{2C_a^2(C_f^2-C_s^2)}\bigg\},
\end{align*}
\newpage
\begin{align*}
\gamma^3_{16}=&(\lambda_1-\lambda_6)\bigg\{\frac{-C_s}{2(C_f^2-C_s^2)}\bigg[H_3\partial_{H_2}\bigg(\frac{C_a^2}{C_f}-C_f\bigg)-H_2\partial_{H_3}\bigg(\frac{C_a^2}{C_f}-C_f\bigg)\bigg]\\
&+\bigg[\frac{C_s^2}{(C_f^2-C_s^2)C_f}-\frac{C_s(C_a^2-C_f^2)}{2C_f^2(C_f^2-C_s^2)}\bigg](H_3\partial_{H_2}C_s-H_2\partial_{H_3}C_s)\bigg\},\\
\gamma^3_{17}=&-(\lambda_1-\lambda_7)\bigg\{\frac{-C_s}{2(C_f^2-C_s^2)}\bigg[\rho\bigg(\frac{C_a^2}{C_f^2}-1\bigg)\partial_{\rho}\bigg(\frac{C_a^2}{C_f}-C_f\bigg)-H_2\partial_{H_2}\bigg(\frac{C_a^2}{C_f}-C_f\bigg)-H_3\partial_{H_3}\bigg(\frac{C_a^2}{C_f}-C_f\bigg)\bigg]\\
&+\frac{C_f^2C_s^2}{2\rho C_a^2(C_f^2-C_s^2)}\bigg[\frac{1}{2}\partial_\rho \bigg(\rho\big(\frac{C_a^2}{C_f^2}-1\big)\bigg)^2+\frac{2\rho C_a^2H_2}{C_f^3}\partial_{H_2}C_f+\frac{2\rho C_a^2H_3}{C_f^3}\partial_{H_3}C_f\bigg]\\
&-\frac{C_s(C_a^2-C_f^2)}{2(C_f^2-C_s^2)}\bigg[\frac{\rho\partial_\rho C_f}{C_f^2}\bigg(\frac{C_a^2}{C_f^2}-1\bigg)+\frac{C_a^2}{C_f^3}-\frac{H_2}{C_f^2}\partial_{H_2}C_f-\frac{H_3}{C_f^2}\partial_{H_3}C_f\bigg]-\frac{C_s^2(C_a^2-C_f^2)}{2C_a^2(C_f^2-C_s^2)}\bigg\},\\
\gamma^3_{21}=&0,\ \gamma^3_{24}=0,\ \gamma^3_{25}=0,\ \gamma^3_{26}=-\frac{C_s(\lambda_2-\lambda_6)(C_a+C_s)(C_a^2-C_f^2)}{2C_a^2(C_f^2-C_s^2)},\ \gamma^3_{27}=0,\\
\gamma^3_{41}=&\frac{(\lambda_4-\lambda_1)(C_a^2-C_f^2)C_s^2}{2\gamma C_a^2(C_f^2-C_s^2)},\ \gamma^3_{42}=0,\ \gamma^3_{45}=-\frac{(\lambda_4-\lambda_5)C_f^2(C_a^2-C_s^2)}{2\gamma C_a^2(C_f^2-C_s^2)},\ \gamma^3_{46}=0,\\
\gamma^3_{47}=&-\frac{(\lambda_4-\lambda_7)(C_a^2-C_f^2)C_s^2}{2\gamma C_a^2(C_f^2-C_s^2)},\\
\gamma^3_{51}=&-(\lambda_5-\lambda_1)\bigg\{\frac{C_s}{2(C_f^2-C_s^2)}\bigg[\rho\bigg(\frac{C_a^2}{C_f^2}-1\bigg)\partial_{\rho}\bigg(\frac{C_a^2}{C_s}-C_s\bigg)-H_2\partial_{H_2}\bigg(\frac{C_a^2}{C_s}-C_s\bigg)-H_3\partial_{H_3}\bigg(\frac{C_a^2}{C_s}-C_s\bigg)\bigg]\\
&+\frac{C_f^2C_s^2}{2C_a^2(C_f^2-C_s^2)}\bigg[\bigg(\frac{C_a^2}{C_f^2}-1\bigg)\bigg(\frac{C_a^2}{C_s^2}-1+\rho\partial_\rho \frac{C_a^2}{C_s^2}\bigg)+\frac{2C_a^2H_2}{C_s^3}\partial_{H_2}C_s+\frac{2C_a^2H_3}{C_s^3}\partial_{H_3}C_s\bigg]\\
&+\frac{C_s(C_a^2-C_s^2)}{2(C_f^2-C_s^2)}\bigg[\bigg(\frac{1}{C_s}+\frac{\rho\partial_\rho C_s}{C_s^2}\bigg)\bigg(\frac{C_a^2}{C_f^2}-1\bigg)+\frac{1}{C_s}-\frac{H_2}{C_s^2}\partial_{H_2}C_s-\frac{H_3}{C_s^2}\partial_{H_3}C_s\bigg]-\frac{C_s^2(C_a^2-C_f^2)}{2C_a^2(C_f^2-C_s^2)}\bigg\},\\
\gamma^3_{52}=&(\lambda_5-\lambda_2)\bigg\{\frac{C_s}{2(C_f^2-C_s^2)}\bigg[H_3\partial_{H_2}\bigg(\frac{C_a^2}{C_s}-C_s\bigg)-H_2\partial_{H_3}\bigg(\frac{C_a^2}{C_s}-C_s\bigg)\bigg]\\
&-\frac{C_a^2+C_f^2}{2C_s(C_f^2-C_s^2)}(H_3\partial_{H_2}C_s-H_2\partial_{H_3}C_s)\bigg\},\\
\gamma^3_{54}=&(\lambda_5-\lambda_4)\bigg\{\frac{-C_s}{2(C_f^2-C_s^2)}\bigg[-\frac{\rho}{\gamma}\partial_\rho \bigg(\frac{C_a^2}{C_s}-C_s\bigg)+\partial_S \bigg(\frac{C_a^2}{C_s}-C_s\bigg)\bigg]\\
&-\frac{C_f^2 C_s^2}{2C_a^2(C_f^2-C_s^2)}\bigg[\frac{1}{\gamma}\partial_\rho \rho\bigg(\frac{C_a^2}{C_s^2}-1\bigg)+\frac{2C_a^2}{C_s^3}\partial_S C_s\bigg]+\frac{C_a^2-C_f^2}{2(C_f^2-C_s^2)}\bigg(\frac{\rho}{\gamma C_s}\partial_\rho C_s-\frac{\partial_S C_s}{C_s}+\frac{1}{\gamma}\bigg)\bigg\},
\end{align*}
\newpage
\begin{align*}
\gamma^3_{56}=&(\lambda_5-\lambda_6)\bigg\{\frac{-C_s}{2(C_f^2-C_s^2)}\bigg[H_3\partial_{H_2}\bigg(\frac{C_a^2}{C_s}-C_s\bigg)-H_2\partial_{H_3}\bigg(\frac{C_a^2}{C_s}-C_s\bigg)\bigg]\\
&+\frac{C_a^2+C_f^2}{2C_s(C_f^2-C_s^2)}(H_3\partial_{H_2}C_s-H_2\partial_{H_3}C_s)\bigg\},\\
\gamma^3_{57}=&(\lambda_5-\lambda_7)\bigg\{\frac{C_s}{2(C_f^2-C_s^2)}\bigg[\rho\bigg(\frac{C_a^2}{C_f^2}-1\bigg)\partial_{\rho}\bigg(\frac{C_a^2}{C_s}-C_s\bigg)-H_2\partial_{H_2}\bigg(\frac{C_a^2}{C_s}-C_s\bigg)-H_3\partial_{H_3}\bigg(\frac{C_a^2}{C_s}-C_s\bigg)\bigg]\\
&+\frac{C_f^2C_s^2}{2C_a^2(C_f^2-C_s^2)}\bigg[\bigg(\frac{C_a^2}{C_f^2}-1\bigg)\bigg(\frac{C_a^2}{C_s^2}-1+\rho\partial_\rho \frac{C_a^2}{C_s^2}\bigg)+\frac{2C_a^2H_2}{C_s^3}\partial_{H_2}C_s+\frac{2C_a^2H_3}{C_s^3}\partial_{H_3}C_s\bigg]\\
&+\frac{C_s(C_a^2-C_s^2)}{2(C_f^2-C_s^2)}\bigg[\bigg(\frac{1}{C_s}+\frac{\rho\partial_\rho C_s}{C_s^2}\bigg)\bigg(\frac{C_a^2}{C_f^2}-1\bigg)+\frac{1}{C_s}-\frac{H_2}{C_s^2}\partial_{H_2}C_s-\frac{H_3}{C_s^2}\partial_{H_3}C_s\bigg]-\frac{C_s^2(C_a^2-C_f^2)}{2C_a^2(C_f^2-C_s^2)}\bigg\},\\
\gamma^3_{61}=&0,\ \gamma^3_{62}=-\frac{C_s(\lambda_6-\lambda_2)(C_s-C_a)(C_a^2-C_f^2)}{2C_a^2(C_f^2-C_s^2)},\ \gamma^3_{64}=0,\ \gamma^3_{65}=0,\ \gamma^3_{67}=0,\\
\gamma^3_{71}=&-(\lambda_7-\lambda_1)\bigg\{\frac{-C_s}{2(C_f^2-C_s^2)}\bigg[\rho\bigg(\frac{C_a^2}{C_f^2}-1\bigg)\partial_{\rho}\bigg(\frac{C_a^2}{C_f}-C_f\bigg)-H_2\partial_{H_2}\bigg(\frac{C_a^2}{C_f}-C_f\bigg)-H_3\partial_{H_3}\bigg(\frac{C_a^2}{C_f}-C_f\bigg)\bigg]\\
&+\frac{C_f^2C_s^2}{2\rho C_a^2(C_f^2-C_s^2)}\bigg[\frac{1}{2}\partial_\rho \bigg(\rho\big(\frac{C_a^2}{C_f^2}-1\big)\bigg)^2+\frac{2\rho C_a^2H_2}{C_f^3}\partial_{H_2}C_f+\frac{2\rho C_a^2H_3}{C_f^3}\partial_{H_3}C_f\bigg]\\
&+\frac{C_s(C_a^2-C_f^2)}{2(C_f^2-C_s^2)}\bigg[\frac{\rho\partial_\rho C_f}{C_f^2}\bigg(\frac{C_a^2}{C_f^2}-1\bigg)+\frac{C_a^2}{C_f^3}-\frac{H_2}{C_f^2}\partial_{H_2}C_f-\frac{H_3}{C_f^2}\partial_{H_3}C_f\bigg]-\frac{C_s^2(C_a^2-C_f^2)}{2C_a^2(C_f^2-C_s^2)}\bigg\},\\
\gamma^3_{72}=&(\lambda_7-\lambda_2)\bigg\{\frac{C_s}{2(C_f^2-C_s^2)}\bigg[H_3\partial_{H_2}\bigg(\frac{C_a^2}{C_f}-C_f\bigg)-H_2\partial_{H_3}\bigg(\frac{C_a^2}{C_f}-C_f\bigg)\bigg]\\
&+\bigg[\frac{C_s^2}{(C_f^2-C_s^2)C_f}+\frac{C_s(C_a^2-C_f^2)}{2C_f^2(C_f^2-C_s^2)}\bigg](H_3\partial_{H_2}C_s-H_2\partial_{H_3}C_s)\bigg\},\\
\gamma^3_{74}=&(\lambda_7-\lambda_4)\bigg\{\frac{C_s}{2(C_f^2-C_s^2)}\bigg[\frac{\rho}{\gamma}\partial_\rho \bigg(\frac{C_a^2}{C_f}-C_f\bigg)-\partial_S \bigg(\frac{C_a^2}{C_f}-C_f\bigg)\bigg]\\
&-\frac{C_f^2 C_s^2}{2C_a^2(C_f^2-C_s^2)}\bigg[\frac{\rho}{\gamma}\partial_\rho \frac{C_a^2}{C_f^2}+\frac{2C_a^2}{C_f^3}\partial_S C_f+\frac{1}{\gamma}\bigg(\frac{C_a^2}{C_f^2}-1\bigg)\bigg]+\frac{C_s(C_a^2-C_s^2)}{2(C_f^2-C_s^2)}\bigg(\frac{\rho}{\gamma C_f^2}\partial_\rho C_f-\frac{\partial_S C_f}{C_f^2}+\frac{1}{\gamma C_f}\bigg)\bigg\},\\
\gamma^3_{75}=&-(\lambda_7-\lambda_5)\bigg\{\frac{C_s}{2(C_f^2-C_s^2)}\bigg[\rho\bigg(\frac{C_a^2}{C_s^2}-1\bigg)\partial_{\rho}\bigg(\frac{C_a^2}{C_f}-C_f\bigg)-H_2\partial_{H_2}\bigg(\frac{C_a^2}{C_f}-C_f\bigg)-H_3\partial_{H_3}\bigg(\frac{C_a^2}{C_f}-C_f\bigg)\bigg]\\
&-\frac{C_f^2C_s^2}{2C_a^2(C_f^2-C_s^2)}\bigg[\bigg(\frac{C_a^2}{C_s^2}-1\bigg)\bigg(\frac{C_a^2}{C_f^2}-1+\rho\partial_\rho \frac{C_a^2}{C_f^2}\bigg)+\frac{2C_a^2H_2}{C_f^3}\partial_{H_2}C_f+\frac{2C_a^2H_3}{C_f^3}\partial_{H_3}C_f\bigg]\\
&-\frac{C_s(C_a^2-C_s^2)}{2(C_f^2-C_s^2)}\bigg[\frac{H_2}{C_f^2}\partial_{H_2}C_f+\frac{H_3}{C_f^2}\partial_{H_3}C_f-\bigg(\frac{1}{C_f}+\frac{\rho\partial_\rho C_f}{C_f^2}\bigg)\bigg(\frac{C_a^2}{C_s^2}-1\bigg)-\frac{1}{C_f}\bigg]+\frac{C_s^2(C_a^2-C_f^2)}{2C_a^2(C_f^2-C_s^2)}\bigg\},\\
\gamma^3_{76}=&(\lambda_7-\lambda_6)\bigg\{\frac{-C_s}{2(C_f^2-C_s^2)}\bigg[H_3\partial_{H_2}\bigg(\frac{C_a^2}{C_f}-C_f\bigg)-H_2\partial_{H_3}\bigg(\frac{C_a^2}{C_f}-C_f\bigg)\bigg]\\
&-\bigg[\frac{C_s^2}{(C_f^2-C_s^2)C_f}+\frac{C_s(C_a^2-C_f^2)}{2C_f^2(C_f^2-C_s^2)}\bigg](H_3\partial_{H_2}C_s-H_2\partial_{H_3}C_s)\bigg\}.
\end{align*}
\end{small}
Note that the coefficients $\gamma^3_{km}$ take similar forms to that in $\gamma^1_{km}$. In the same fashion, one can verify that $\gamma^3_{km}$ are all regular.

Next, for $i>4$, the remaining coefficients can be estimated in a similar manner based on the above calculations. This can be verified easily by noticing the following two facts: first, the potentially singular $\frac{1}{H_2^2+H_3^2}$ term appears only in the $2^{\text{nd}}$, $3^{\text{rd}}$, $5^{\text{th}}$, $6^{\text{th}}$ components of $l_i$ and $l_i^{1,4,7}$ are all regular. Since all $\nabla_\Phi r_k \cdot r_m$ are of order $1$, it holds that $l_i^j \cdot (\nabla_\Phi r_k \cdot r_m)_j=O(1)$ with $j=1,4,7$. Then, we only need to check
$$
\sum\limits_{j=2,3} l_i^j \cdot (\nabla_\Phi r_k \cdot r_m)_j, \quad \sum\limits_{j=5,6} l_i^j \cdot (\nabla_\Phi r_k \cdot r_m)_j.
$$
Second, the components of $l_i$ satisfy the following symmetry properties
$$
l_i^{1,2,3}=l_{8-i}^{1,2,3},\quad l_i^{4,5,6,7}=-l_{8-i}^{4,5,6,7}.
$$
Therefore, the rest coefficients $\gamma^i_{km}$ with $i>4$ can be regarded as \underline{counterparts} of the known ones and thus can be easily verified to be of $O(1)$. Now we also list all the other coefficients with for $i>4$:
\begin{small}
\begin{align*}
\gamma^5_{51}=&-(\lambda_5-\lambda_1)\bigg\{\frac{-C_s}{2(C_f^2-C_s^2)}\bigg[H_2\partial_{H_2}\bigg(\frac{C_a^2}{C_f}-C_f\bigg)+H_3\partial_{H_3}\bigg(\frac{C_a^2}{C_f}-C_f\bigg)-\rho\bigg(\frac{C_a^2}{C_s^2}-1\bigg)\partial_{\rho}\bigg(\frac{C_a^2}{C_f}-C_f\bigg)\\
&+H_2\partial_{H_2}\bigg(\frac{C_a^2}{C_s}-C_s\bigg)+H_3\partial_{H_3}\bigg(\frac{C_a^2}{C_s}-C_s\bigg)-\rho\bigg(\frac{C_a^2}{C_f^2}-1\bigg)\partial_{\rho}\bigg(\frac{C_a^2}{C_s}-C_s\bigg)\bigg]\\
&-\frac{C_f^2C_s^2}{2C_a^2(C_f^2-C_s^2)}\bigg[\bigg(\frac{C_a^2}{C_s^2}-1\bigg)\bigg(\frac{C_a^2}{C_f^2}-1+\rho\partial_\rho \frac{C_a^2}{C_f^2}\bigg)+\frac{2C_a^2H_2}{C_f^3}\partial_{H_2}C_f+\frac{2C_a^2H_3}{C_f^3}\partial_{H_3}C_f\\
&-\bigg(\frac{C_a^2}{C_f^2}-1\bigg)\bigg(\frac{C_a^2}{C_s^2}-1+\rho\partial_\rho \frac{C_a^2}{C_s^2}\bigg)-\frac{2C_a^2H_2}{C_s^3}\partial_{H_2}C_s-\frac{2C_a^2H_3}{C_s^3}\partial_{H_3}C_s\bigg]\\
&-\frac{C_s(C_a^2-C_s^2)}{2(C_f^2-C_s^2)}\bigg[\frac{H_2}{C_f^2}\partial_{H_2}C_f+\frac{H_3}{C_f^2}\partial_{H_3}C_f-\bigg(\frac{1}{C_f}+\frac{\rho\partial_\rho C_f}{C_f^2}\bigg)\bigg(\frac{C_a^2}{C_s^2}-1\bigg)-\frac{1}{C_f}\\
&+\frac{H_2}{C_s^2}\partial_{H_2}C_s+\frac{H_3}{C_s^2}\partial_{H_3}C_s-\bigg(\frac{1}{C_s}+\frac{\rho\partial_\rho C_s}{C_s^2}\bigg)\bigg(\frac{C_a^2}{C_f^2}-1\bigg)-\frac{1}{C_s}\bigg]\bigg\},\\
\gamma^5_{52}=&(\lambda_5-\lambda_2)\bigg\{\frac{C_s}{2(C_f^2-C_s^2)}\bigg[H_3\partial_{H_2}\bigg(\frac{C_a^2}{C_s}-C_s\bigg)-H_2\partial_{H_3}\bigg(\frac{C_a^2}{C_s}-C_s\bigg)\bigg]\\
&+\frac{3C_f^2-C_a^2}{2C_s(C_f^2-C_s^2)}(H_3\partial_{H_2}C_s-H_2\partial_{H_3}C_s)\bigg\},\\
\gamma^5_{53}=&\frac{(\lambda_3-\lambda_5)(C_a^2-C_f^2)}{C_f^2-C_s^2}\bigg[\frac{\rho(C_a^2-C_s^2)}{C_s^3}\partial_\rho C_s+\frac{C_a^2}{C_s^2}-\frac{H_2}{C_s}\partial_{H_2}C_s-\frac{H_3}{C_s}\partial_{H_3}C_s\bigg],\\
\gamma^5_{54}=&-(\lambda_5-\lambda_4)\bigg\{\frac{C_s}{2(C_f^2-C_s^2)}\bigg[-\frac{\rho}{\gamma}\partial_\rho \bigg(\frac{C_a^2}{C_s}-C_s\bigg)+\partial_S \bigg(\frac{C_a^2}{C_s}-C_s\bigg)\bigg]-\frac{C_f^2 C_s^2}{2C_a^2(C_f^2-C_s^2)}\bigg[\frac{\rho}{\gamma}\partial_\rho \frac{C_a^2}{C_s^2}+\frac{2C_a^2}{C_s^3}\partial_S C_s\bigg]\\
&-\frac{C_a^2-C_f^2}{2(C_f^2-C_s^2)}\bigg(\frac{\rho}{\gamma C_s}\partial_\rho C_s-\frac{\partial_S C_s}{C_s}+\frac{1}{\gamma}\bigg)\bigg\},\\
\gamma^5_{56}=&-(\lambda_5-\lambda_6)\bigg\{\frac{C_s}{2(C_f^2-C_s^2)}\bigg[H_3\partial_{H_2}\bigg(\frac{C_a^2}{C_s}-C_s\bigg)-H_2\partial_{H_3}\bigg(\frac{C_a^2}{C_s}-C_s\bigg)\bigg]\\
&+\frac{3C_f^2-C_a^2}{2C_s(C_f^2-C_s^2)}(H_3\partial_{H_2}C_s-H_2\partial_{H_3}C_s)\bigg\},\\
\gamma^5_{57}=&(\lambda_5-\lambda_7)\bigg\{\frac{C_s}{2(C_f^2-C_s^2)}\bigg[H_2\partial_{H_2}\bigg(\frac{C_a^2}{C_f}-C_f\bigg)+H_3\partial_{H_3}\bigg(\frac{C_a^2}{C_f}-C_f\bigg)-\rho\bigg(\frac{C_a^2}{C_s^2}-1\bigg)\partial_{\rho}\bigg(\frac{C_a^2}{C_f}-C_f\bigg)\\
&-H_2\partial_{H_2}\bigg(\frac{C_a^2}{C_s}-C_s\bigg)-H_3\partial_{H_3}\bigg(\frac{C_a^2}{C_s}-C_s\bigg)+\rho\bigg(\frac{C_a^2}{C_f^2}-1\bigg)\partial_{\rho}\bigg(\frac{C_a^2}{C_s}-C_s\bigg)\bigg]\\
&-\frac{C_f^2C_s^2}{2C_a^2(C_f^2-C_s^2)}\bigg[\bigg(\frac{C_a^2}{C_s^2}-1\bigg)\bigg(\frac{C_a^2}{C_f^2}-1+\rho\partial_\rho \frac{C_a^2}{C_f^2}\bigg)+\frac{2C_a^2H_2}{C_f^3}\partial_{H_2}C_f+\frac{2C_a^2H_3}{C_f^3}\partial_{H_3}C_f\\
&-\bigg(\frac{C_a^2}{C_f^2}-1\bigg)\bigg(\frac{C_a^2}{C_s^2}-1+\rho\partial_\rho \frac{C_a^2}{C_s^2}\bigg)-\frac{2C_a^2H_2}{C_s^3}\partial_{H_2}C_s-\frac{2C_a^2H_3}{C_s^3}\partial_{H_3}C_s]\\
&+\frac{C_s(C_a^2-C_s^2)}{2(C_f^2-C_s^2)}\bigg[\frac{H_2}{C_f^2}\partial_{H_2}C_f+\frac{H_3}{C_f^2}\partial_{H_3}C_f-\bigg(\frac{1}{C_f}+\frac{\rho\partial_\rho C_f}{C_f^2}\bigg)\bigg(\frac{C_a^2}{C_s^2}-1\bigg)-\frac{1}{C_f}\\
&-\frac{H_2}{C_s^2}\partial_{H_2}C_s-\frac{H_3}{C_s^2}\partial_{H_3}C_s+\bigg(\frac{1}{C_s}+\frac{\rho\partial_\rho C_s}{C_s^2}\bigg)\bigg(\frac{C_a^2}{C_f^2}-1\bigg)+\frac{1}{C_s}\bigg]\bigg\},\\
\gamma^5_{12}=&-(\lambda_1-\lambda_2)\bigg\{\frac{-C_s}{2(C_f^2-C_s^2)}\bigg[H_3\partial_{H_2}\bigg(\frac{C_a^2}{C_f}-C_f\bigg)-H_2\partial_{H_3}\bigg(\frac{C_a^2}{C_f}-C_f\bigg)\bigg]\\
&-\bigg[\frac{C_s^2}{(C_f^2-C_s^2)C_f}+\frac{C_s(C_a^2-C_f^2)}{2C_f^2(C_f^2-C_s^2)}\bigg](H_3\partial_{H_2}C_s-H_2\partial_{H_3}C_s)\bigg\},\\
\gamma^5_{13}=&(\lambda_1-\lambda_3)\bigg\{\frac{C_s}{2(C_f^2-C_s^2)}\bigg[\rho\bigg(\frac{C_a^2}{C_s^2}-1\bigg)\partial_{\rho}\bigg(\frac{C_a^2}{C_f}-C_f\bigg)-H_2\partial_{H_2}\bigg(\frac{C_a^2}{C_f}-C_f\bigg)-H_3\partial_{H_3}\bigg(\frac{C_a^2}{C_f}-C_f\bigg)\bigg]\\
&-\frac{C_f^2C_s^2}{2C_a^2(C_f^2-C_s^2)}\bigg[\bigg(\frac{C_a^2}{C_s^2}-1\bigg)\bigg(\frac{C_a^2}{C_f^2}-1+\rho\partial_\rho \frac{C_a^2}{C_f^2}\bigg)+\frac{2C_a^2H_2}{C_f^3}\partial_{H_2}C_f+\frac{2C_a^2H_3}{C_f^3}\partial_{H_3}C_f\bigg]\\
&-\frac{C_s(C_a^2-C_s^2)}{2(C_f^2-C_s^2)}\bigg[\frac{H_2}{C_f^2}\partial_{H_2}C_f+\frac{H_3}{C_f^2}\partial_{H_3}C_f-\bigg(\frac{1}{C_f}+\frac{\rho\partial_\rho C_f}{C_f^2}\bigg)\bigg(\frac{C_a^2}{C_s^2}-1\bigg)-\frac{1}{C_f}\bigg]+\frac{C_s^2(C_a^2-C_f^2)}{2C_a^2(C_f^2-C_s^2)}\bigg\},
\end{align*}
\newpage
\begin{align*}
\gamma^5_{14}=&(\lambda_1-\lambda_4)\bigg\{\frac{C_s}{2(C_f^2-C_s^2)}\bigg[\frac{\rho}{\gamma}\partial_\rho \bigg(\frac{C_a^2}{C_f}-C_f\bigg)-\partial_S \bigg(\frac{C_a^2}{C_f}-C_f\bigg)\bigg]\\
&-\frac{C_f^2 C_s^2}{2C_a^2(C_f^2-C_s^2)}\bigg[\frac{\rho}{\gamma}\partial_\rho \frac{C_a^2}{C_f^2}+\frac{2C_a^2}{C_f^3}\partial_S C_f+\frac{1}{\gamma}\bigg(\frac{C_a^2}{C_f^2}-1\bigg)\bigg]+\frac{C_s(C_a^2-C_s^2)}{2(C_f^2-C_s^2)}\bigg(\frac{\rho}{\gamma C_f^2}\partial_\rho C_f-\frac{\partial_S C_f}{C_f^2}+\frac{1}{\gamma C_f}\bigg)\bigg\},\\
\gamma^5_{16}=&(\lambda_1-\lambda_6)\bigg\{\frac{-C_s}{2(C_f^2-C_s^2)}\bigg[H_3\partial_{H_2}\bigg(\frac{C_a^2}{C_f}-C_f\bigg)-H_2\partial_{H_3}\bigg(\frac{C_a^2}{C_f}-C_f\bigg)\bigg]\\
&-\bigg[\frac{C_s^2}{(C_f^2-C_s^2)C_f}+\frac{C_s(C_a^2-C_f^2)}{2C_f^2(C_f^2-C_s^2)}\bigg](H_3\partial_{H_2}C_s-H_2\partial_{H_3}C_s)\bigg\},\\
\gamma^5_{17}=&-(\lambda_1-\lambda_7)\bigg\{\frac{-C_s}{2(C_f^2-C_s^2)}\bigg[\rho\bigg(\frac{C_a^2}{C_f^2}-1\bigg)\partial_{\rho}\bigg(\frac{C_a^2}{C_f}-C_f\bigg)-H_2\partial_{H_2}\bigg(\frac{C_a^2}{C_f}-C_f\bigg)-H_3\partial_{H_3}\bigg(\frac{C_a^2}{C_f}-C_f\bigg)\bigg]\\
&-\frac{C_f^2C_s^2}{2\rho C_a^2(C_f^2-C_s^2)}\bigg[\frac{1}{2}\partial_\rho \bigg(\rho\big(\frac{C_a^2}{C_f^2}-1\big)\bigg)^2+\frac{2\rho C_a^2H_2}{C_f^3}\partial_{H_2}C_f+\frac{2\rho C_a^2H_3}{C_f^3}\partial_{H_3}C_f\bigg]\\
&-\frac{C_s(C_a^2-C_f^2)}{2(C_f^2-C_s^2)}\bigg[\frac{\rho\partial_\rho C_f}{C_f^2}\bigg(\frac{C_a^2}{C_f^2}-1\bigg)+\frac{C_a^2}{C_f^3}-\frac{H_2}{C_f^2}\partial_{H_2}C_f-\frac{H_3}{C_f^2}\partial_{H_3}C_f\bigg]+\frac{C_s^2(C_a^2-C_f^2)}{2C_a^2(C_f^2-C_s^2)}\bigg\},\\
\gamma^5_{21}=&0,\ \gamma^5_{23}=0,\ \gamma^5_{24}=0,\ \gamma^5_{26}=-\frac{C_s(\lambda_2-\lambda_6)(C_a-C_s)(C_a^2-C_f^2)}{2C_a^2(C_f^2-C_s^2)},\ \gamma^5_{27}=0,\\
\gamma^5_{31}=&-(\lambda_3-\lambda_1)\bigg\{\frac{C_s}{2(C_f^2-C_s^2)}\bigg[\rho\bigg(\frac{C_a^2}{C_f^2}-1\bigg)\partial_{\rho}\bigg(\frac{C_a^2}{C_s}-C_s\bigg)-H_2\partial_{H_2}\bigg(\frac{C_a^2}{C_s}-C_s\bigg)-H_3\partial_{H_3}\bigg(\frac{C_a^2}{C_s}-C_s\bigg)\bigg]\\
&+\frac{C_f^2C_s^2}{2C_a^2(C_f^2-C_s^2)}\bigg[\bigg(\frac{C_a^2}{C_f^2}-1\bigg)\bigg(\frac{C_a^2}{C_s^2}-1+\rho\partial_\rho \frac{C_a^2}{C_s^2}\bigg)+\frac{2C_a^2H_2}{C_s^3}\partial_{H_2}C_s+\frac{2C_a^2H_3}{C_s^3}\partial_{H_3}C_s\bigg]\\
&+\frac{C_s(C_a^2-C_s^2)}{2(C_f^2-C_s^2)}\bigg[\bigg(\frac{1}{C_s}+\frac{\rho\partial_\rho C_s}{C_s^2}\bigg)\bigg(\frac{C_a^2}{C_f^2}-1\bigg)+\frac{1}{C_s}-\frac{H_2}{C_s^2}\partial_{H_2}C_s-\frac{H_3}{C_s^2}\partial_{H_3}C_s\bigg]-\frac{C_s^2(C_a^2-C_f^2)}{2C_a^2(C_f^2-C_s^2)}\bigg\},\\
\gamma^5_{32}=&(\lambda_3-\lambda_2)\bigg\{\frac{C_s}{2(C_f^2-C_s^2)}\bigg[H_3\partial_{H_2}\bigg(\frac{C_a^2}{C_s}-C_s\bigg)-H_2\partial_{H_3}\bigg(\frac{C_a^2}{C_s}-C_s\bigg)\bigg]\\
&-\frac{C_a^2+C_f^2}{2C_s(C_f^2-C_s^2)}(H_3\partial_{H_2}C_s-H_2\partial_{H_3}C_s)\bigg\},\\
\gamma^5_{34}=&-(\lambda_3-\lambda_4)\bigg\{\frac{C_s}{2(C_f^2-C_s^2)}\bigg[-\frac{\rho}{\gamma}\partial_\rho \bigg(\frac{C_a^2}{C_s}-C_s\bigg)+\partial_S \bigg(\frac{C_a^2}{C_s}-C_s\bigg)\bigg]\\
&+\frac{C_f^2 C_s^2}{2C_a^2(C_f^2-C_s^2)}\bigg[\frac{1}{\gamma}\partial_\rho \rho(\frac{C_a^2}{C_s^2}-1)+\frac{2C_a^2}{C_s^3}\partial_S C_s\bigg]-\frac{C_a^2-C_f^2}{2(C_f^2-C_s^2)}\bigg(\frac{\rho}{\gamma C_s}\partial_\rho C_s-\frac{\partial_S C_s}{C_s}+\frac{1}{\gamma}\bigg)\bigg\},
\end{align*}
\newpage
\begin{align*}
\gamma^5_{36}=&-(\lambda_3-\lambda_6)\bigg\{\frac{C_s}{2(C_f^2-C_s^2)}\bigg[H_3\partial_{H_2}\bigg(\frac{C_a^2}{C_s}-C_s\bigg)-H_2\partial_{H_3}\bigg(\frac{C_a^2}{C_s}-C_s\bigg)\bigg]\\
&-\frac{C_a^2+C_f^2}{2C_s(C_f^2-C_s^2)}(H_3\partial_{H_2}C_s-H_2\partial_{H_3}C_s)\bigg\},\\
\gamma^5_{37}=&(\lambda_3-\lambda_7)\bigg\{\frac{C_s}{2(C_f^2-C_s^2)}\bigg[\rho\bigg(\frac{C_a^2}{C_f^2}-1\bigg)\partial_{\rho}\bigg(\frac{C_a^2}{C_s}-C_s\bigg)-H_2\partial_{H_2}\bigg(\frac{C_a^2}{C_s}-C_s\bigg)-H_3\partial_{H_3}\bigg(\frac{C_a^2}{C_s}-C_s\bigg)\bigg]\\
&+\frac{C_f^2C_s^2}{2C_a^2(C_f^2-C_s^2)}\bigg[\bigg(\frac{C_a^2}{C_f^2}-1\bigg)\bigg(\frac{C_a^2}{C_s^2}-1+\rho\partial_\rho \frac{C_a^2}{C_s^2}\bigg)+\frac{2C_a^2H_2}{C_s^3}\partial_{H_2}C_s+\frac{2C_a^2H_3}{C_s^3}\partial_{H_3}C_s\bigg]\\
&+\frac{C_s(C_a^2-C_s^2)}{2(C_f^2-C_s^2)}\bigg[\bigg(\frac{1}{C_s}+\frac{\rho\partial_\rho C_s}{C_s^2}\bigg)\bigg(\frac{C_a^2}{C_f^2}-1\bigg)+\frac{1}{C_s}-\frac{H_2}{C_s^2}\partial_{H_2}C_s-\frac{H_3}{C_s^2}\partial_{H_3}C_s\bigg]-\frac{C_s^2(C_a^2-C_f^2)}{2C_a^2(C_f^2-C_s^2)}\bigg\},\\
\gamma^5_{41}=&-\frac{(\lambda_4-\lambda_1)(C_a^2-C_f^2)C_s^2}{2\gamma C_a^2(C_f^2-C_s^2)},\ \gamma^5_{42}=0,\ \gamma^5_{43}=-\frac{(\lambda_4-\lambda_3)C_f^2(C_a^2-C_s^2)}{2\gamma C_a^2(C_f^2-C_s^2)},\ \gamma^5_{46}=0,\\
\gamma^5_{47}=&\frac{(\lambda_4-\lambda_7)(C_a^2-C_f^2)C_s^2}{2\gamma C_a^2(C_f^2-C_s^2)},\\
\gamma^5_{61}=&0,\ \gamma^5_{62}=\frac{C_s(\lambda_6-\lambda_2)(C_s+C_a)(C_a^2-C_f^2)}{2C_a^2(C_f^2-C_s^2)},\ \gamma^5_{63}=0,\ \gamma^5_{64}=0,\ \gamma^5_{67}=0,\\
\gamma^5_{71}=&-(\lambda_7-\lambda_1)\bigg\{\frac{-C_s}{2(C_f^2-C_s^2)}\bigg[\rho\bigg(\frac{C_a^2}{C_f^2}-1\bigg)\partial_{\rho}\bigg(\frac{C_a^2}{C_f}-C_f\bigg)-H_2\partial_{H_2}\bigg(\frac{C_a^2}{C_f}-C_f\bigg)-H_3\partial_{H_3}\bigg(\frac{C_a^2}{C_f}-C_f\bigg)\bigg]\\
&-\frac{C_f^2C_s^2}{2\rho C_a^2(C_f^2-C_s^2)}\bigg[\frac{1}{2}\partial_\rho \bigg(\rho\big(\frac{C_a^2}{C_f^2}-1\big)\bigg)^2+\frac{2\rho C_a^2H_2}{C_f^3}\partial_{H_2}C_f+\frac{2\rho C_a^2H_3}{C_f^3}\partial_{H_3}C_f\bigg]\\
&+\frac{C_s(C_a^2-C_f^2)}{2(C_f^2-C_s^2)}\bigg[\frac{\rho\partial_\rho C_f}{C_f^2}\bigg(\frac{C_a^2}{C_f^2}-1\bigg)+\frac{C_a^2}{C_f^3}-\frac{H_2}{C_f^2}\partial_{H_2}C_f-\frac{H_3}{C_f^2}\partial_{H_3}C_f\bigg]+\frac{C_s^2(C_a^2-C_f^2)}{2C_a^2(C_f^2-C_s^2)}\bigg\},\\
\gamma^5_{72}=&-(\lambda_7-\lambda_2)\bigg\{\frac{-C_s}{2(C_f^2-C_s^2)}\bigg[H_3\partial_{H_2}\bigg(\frac{C_a^2}{C_f}-C_f\bigg)-H_2\partial_{H_3}\bigg(\frac{C_a^2}{C_f}-C_f\bigg)\bigg]\\
&+\bigg[\frac{C_s^2}{(C_f^2-C_s^2)C_f}-\frac{C_s(C_a^2-C_f^2)}{2C_f^2(C_f^2-C_s^2)}\bigg](H_3\partial_{H_2}C_s-H_2\partial_{H_3}C_s)\bigg\},\\
\gamma^5_{73}=&(\lambda_7-\lambda_3)\bigg\{\frac{C_s}{2(C_f^2-C_s^2)}\bigg[\rho\bigg(\frac{C_a^2}{C_s^2}-1\bigg)\partial_{\rho}\bigg(\frac{C_a^2}{C_f}-C_f\bigg)-H_2\partial_{H_2}\bigg(\frac{C_a^2}{C_f}-C_f\bigg)-H_3\partial_{H_3}\bigg(\frac{C_a^2}{C_f}-C_f\bigg)\bigg]\\
&+\frac{C_f^2C_s^2}{2C_a^2(C_f^2-C_s^2)}\bigg[\bigg(\frac{C_a^2}{C_s^2}-1\bigg)\bigg(\frac{C_a^2}{C_f^2}-1+\rho\partial_\rho \frac{C_a^2}{C_f^2}\bigg)+\frac{2C_a^2H_2}{C_f^3}\partial_{H_2}C_f+\frac{2C_a^2H_3}{C_f^3}\partial_{H_3}C_f\bigg]\\
&-\frac{C_s(C_a^2-C_s^2)}{2(C_f^2-C_s^2)}\bigg[\frac{H_2}{C_f^2}\partial_{H_2}C_f+\frac{H_3}{C_f^2}\partial_{H_3}C_f-\bigg(\frac{1}{C_f}+\frac{\rho\partial_\rho C_f}{C_f^2}\bigg)\bigg(\frac{C_a^2}{C_s^2}-1\bigg)-\frac{1}{C_f}\bigg]-\frac{C_s^2(C_a^2-C_f^2)}{2C_a^2(C_f^2-C_s^2)}\bigg\},\\
\gamma^5_{74}=&(\lambda_7-\lambda_4)\bigg\{\frac{C_s}{2(C_f^2-C_s^2)}\bigg[\frac{\rho}{\gamma}\partial_\rho \bigg(\frac{C_a^2}{C_f}-C_f\bigg)-\partial_S \bigg(\frac{C_a^2}{C_f}-C_f\bigg)\bigg]\\
&+\frac{C_f^2 C_s^2}{2C_a^2(C_f^2-C_s^2)}\bigg[\frac{\rho}{\gamma}\partial_\rho \frac{C_a^2}{C_f^2}+\frac{2C_a^2}{C_f^3}\partial_S C_f+\frac{1}{\gamma}\bigg(\frac{C_a^2}{C_f^2}-1\bigg)\bigg]+\frac{C_s(C_a^2-C_s^2)}{2(C_f^2-C_s^2)}\bigg(\frac{\rho}{\gamma C_f^2}\partial_\rho C_f-\frac{\partial_S C_f}{C_f^2}+\frac{1}{\gamma C_f}\bigg)\bigg\},\\
\gamma^5_{76}=&(\lambda_7-\lambda_6)\bigg\{\frac{-C_s}{2(C_f^2-C_s^2)}\bigg[H_3\partial_{H_2}\bigg(\frac{C_a^2}{C_f}-C_f\bigg)-H_2\partial_{H_3}\bigg(\frac{C_a^2}{C_f}-C_f\bigg)\bigg]\\
&+\bigg[\frac{C_s^2}{(C_f^2-C_s^2)C_f}-\frac{C_s(C_a^2-C_f^2)}{2C_f^2(C_f^2-C_s^2)}\bigg](H_3\partial_{H_2}C_s-H_2\partial_{H_3}C_s)\bigg\},\\
\gamma^6_{61}=&\frac{(\lambda_6-\lambda_1)}{4}\bigg(\frac{C_a}{C_f}+1\bigg)^2,\ \gamma^6_{62}=0,\ \gamma^6_{63}=\frac{(\lambda_6-\lambda_3)}{4}\bigg(\frac{C_a}{C_s}+1\bigg)^2,\\
\gamma^6_{64}=&-\frac{\lambda_6-\lambda_4}{4\gamma},\ \gamma^6_{65}=-\frac{(\lambda_6-\lambda_5)}{4}\bigg(\frac{C_a}{C_s}-1\bigg)^2,\ \gamma^6_{67}=-\frac{\lambda_6-\lambda_7}{4}\bigg(\frac{C_a}{C_f}-1\bigg)^2,\\
\gamma^6_{12}=&-\frac{(\lambda_1-\lambda_2)(C_f-C_a)}{2C_f},\ \gamma^6_{13}=0,\ \gamma^6_{14}=0,\ \gamma^6_{15}=0,\ \gamma^6_{17}=0,\\
\gamma^6_{21}=&\frac{\lambda_2-\lambda_1}{4}\bigg(\frac{C_a^2}{C_f^2}+1\bigg),\ \gamma^6_{23}=-\frac{(\lambda_2-\lambda_3)(C_a^2-C_s^2)}{4C_s^2},\ \gamma^6_{24}=-\frac{\lambda_2-\lambda_4}{4\gamma},\\
\gamma^6_{25}=&\frac{(\lambda_2-\lambda_5)(C_a^2-C_s^2)}{4C_s^2},\ \gamma^6_{27}=-\frac{\lambda_2-\lambda_7}{4}\bigg(\frac{C_a^2}{C_f^2}+1\bigg),\\
\gamma^6_{31}=&0,\ \gamma^6_{32}=\frac{(\lambda_3-\lambda_2)(C_a-C_s)}{2C_s},\ \gamma^6_{34}=0,\ \gamma^6_{35}=0,\ \gamma^6_{37}=0,\\
\gamma^6_{41}=&0,\ \gamma^6_{42}=0,\ \gamma^6_{43}=0,\ \gamma^6_{45}=0,\ \gamma^6_{47}=0,\\
\gamma^6_{51}=&0,\ \gamma^6_{52}=\frac{(\lambda_5-\lambda_2)(C_a+C_s)}{2C_s},\ \gamma^6_{53}=0,\ \gamma^6_{54}=0,\ \gamma^6_{57}=0,\\
\gamma^6_{71}=&0,\ \gamma^6_{72}=-\frac{(\lambda_7-\lambda_2)(C_a+C_f)}{2C_f},\ \gamma^6_{73}=0,\ \gamma^6_{74}=0,\ \gamma^6_{75}=0,\\
\gamma^7_{71}=&-\frac{(\lambda_1-\lambda_7)(C_a^2-C_s^2)}{C_f^2-C_s^2}\bigg[\frac{\rho(C_a^2-C_f^2)}{C_f^3}\partial_\rho C_f+\frac{C_a^2}{C_f^2}-\frac{H_2}{C_f}\partial_{H_2}C_f-\frac{H_3}{C_f}\partial_{H_3}C_f\bigg],\\
\gamma^7_{72}=&-(\lambda_7-\lambda_2)\bigg\{\frac{C_f}{2(C_f^2-C_s^2)}\bigg[H_3\partial_{H_2}\bigg(\frac{C_a^2}{C_f}-C_f\bigg)-H_2\partial_{H_3}\bigg(\frac{C_a^2}{C_f}-C_f\bigg)\bigg]\\
&+\frac{C_a^2-3C_s^2}{2C_f(C_f^2-C_s^2)}(H_3\partial_{H_2}C_f-H_2\partial_{H_3}C_f)\bigg\},
\end{align*}
\newpage
\begin{align*}
\gamma^7_{73}=&(\lambda_7-\lambda_3)\bigg\{\frac{C_f}{2(C_f^2-C_s^2)}\bigg[\rho\bigg(\frac{C_a^2}{C_s^2}-1\bigg)\partial_{\rho}\bigg(\frac{C_a^2}{C_f}-C_f\bigg)-H_2\partial_{H_2}\bigg(\frac{C_a^2}{C_f}-C_f\bigg)-H_3\partial_{H_3}\bigg(\frac{C_a^2}{C_f}-C_f\bigg)\\
&-H_2\partial_{H_2}\bigg(\frac{C_a^2}{C_s}-C_s\bigg)-H_3\partial_{H_3}\bigg(\frac{C_a^2}{C_s}-C_s\bigg)+\rho\bigg(\frac{C_a^2}{C_f^2}-1\bigg)\partial_{\rho}\bigg(\frac{C_a^2}{C_s}-C_s\bigg)\bigg]\\
&+\frac{C_f^2C_s^2}{2C_a^2(C_f^2-C_s^2)}\bigg[\bigg(\frac{C_a^2}{C_s^2}-1\bigg)\bigg(\frac{C_a^2}{C_f^2}-1+\rho\partial_\rho \frac{C_a^2}{C_f^2}\bigg)+\frac{2C_a^2H_2}{C_f^3}\partial_{H_2}C_f+\frac{2C_a^2H_3}{C_f^3}\partial_{H_3}C_f\\
&-\bigg(\frac{C_a^2}{C_f^2}-1\bigg)\bigg(\frac{C_a^2}{C_s^2}-1+\rho\partial_\rho \frac{C_a^2}{C_s^2}\bigg)-\frac{2C_a^2H_2}{C_s^3}\partial_{H_2}C_s-\frac{2C_a^2H_3}{C_s^3}\partial_{H_3}C_s\bigg]\\
&+\frac{C_f(C_a^2-C_s^2)}{2(C_f^2-C_s^2)}\bigg[C_f(\frac{1}{C_f}+\frac{\rho\partial_\rho C_f}{C_f^2})\bigg(\frac{C_a^2}{C_s^2}-1\bigg)+\frac{1}{C_f}-\frac{H_2}{C_f^2}\partial_{H_2}C_f-\frac{H_3}{C_f^2}\partial_{H_3}\\
&-\frac{H_2}{C_s^2}\partial_{H_2}C_s-\frac{H_3}{C_s^2}\partial_{H_3}C_s+\bigg(\frac{1}{C_s}+\frac{\rho\partial_\rho C_s}{C_s^2}\bigg)\bigg(\frac{C_a^2}{C_f^2}-1\bigg)+\frac{1}{C_s}\bigg]\bigg\},\\
\gamma^7_{74}=&-(\lambda_7-\lambda_4)\bigg\{\frac{C_f}{2(C_f^2-C_s^2)}\bigg[\frac{\rho}{\gamma}\partial_\rho \bigg(\frac{C_a^2}{C_f}-C_f\bigg)-\partial_S \bigg(\frac{C_a^2}{C_f}-C_f\bigg)\bigg]+\frac{C_f^2 C_s^2}{2C_a^2(C_f^2-C_s^2)}\bigg(\frac{\rho}{\gamma}\partial_\rho \frac{C_a^2}{C_f^2}+\frac{2C_a^2}{C_f^3}\partial_S C_f\bigg)\\
&+\frac{C_a^2-C_s^2}{2(C_f^2-C_s^2)}\bigg(\frac{\rho}{\gamma C_f}\partial_\rho C_f-\frac{\partial_S C_f}{C_f}+\frac{1}{\gamma}\bigg)\bigg\},\\
\gamma^7_{75}=&(\lambda_7-\lambda_5)\bigg\{\frac{C_f}{2(C_f^2-C_s^2)}\bigg[H_2\partial_{H_2}\bigg(\frac{C_a^2}{C_f}-C_f\bigg)+H_3\partial_{H_3}\bigg(\frac{C_a^2}{C_f}-C_f\bigg)-\rho\bigg(\frac{C_a^2}{C_s^2}-1\bigg)\partial_{\rho}\bigg(\frac{C_a^2}{C_f}-C_f\bigg)\\
&-H_2\partial_{H_2}\bigg(\frac{C_a^2}{C_s}-C_s\bigg)-H_3\partial_{H_3}\bigg(\frac{C_a^2}{C_s}-C_s\bigg)+\rho\bigg(\frac{C_a^2}{C_f^2}-1\bigg)\partial_{\rho}\bigg(\frac{C_a^2}{C_s}-C_s\bigg)\bigg]\\
&-\frac{C_f^2C_s^2}{2C_a^2(C_f^2-C_s^2)}\bigg[\bigg(\frac{C_a^2}{C_s^2}-1\bigg)\bigg(\frac{C_a^2}{C_f^2}-1+\rho\partial_\rho \frac{C_a^2}{C_f^2}\bigg)+\frac{2C_a^2H_2}{C_f^3}\partial_{H_2}C_f+\frac{2C_a^2H_3}{C_f^3}\partial_{H_3}C_f\\
&-\bigg(\frac{C_a^2}{C_f^2}-1\bigg)\bigg(\frac{C_a^2}{C_s^2}-1+\rho\partial_\rho \frac{C_a^2}{C_s^2}\bigg)-\frac{2C_a^2H_2}{C_s^3}\partial_{H_2}C_s-\frac{2C_a^2H_3}{C_s^3}\partial_{H_3}C_s\bigg]\\
&+\frac{C_f(C_a^2-C_s^2)}{2(C_f^2-C_s^2)}\bigg[\frac{H_2}{C_f^2}\partial_{H_2}C_f+\frac{H_3}{C_f^2}\partial_{H_3}C_f-\bigg(\frac{1}{C_f}+\frac{\rho\partial_\rho C_f}{C_f^2}\bigg)\bigg(\frac{C_a^2}{C_s^2}-1\bigg)-\frac{1}{C_f}\\
&-\frac{H_2}{C_s^2}\partial_{H_2}C_s-\frac{H_3}{C_s^2}\partial_{H_3}C_s+\bigg(\frac{1}{C_s}+\frac{\rho\partial_\rho C_s}{C_s^2}\bigg)\bigg(\frac{C_a^2}{C_f^2}-1\bigg)+\frac{1}{C_s}\bigg]\bigg\},\\
\gamma^7_{76}=&(\lambda_7-\lambda_6)\bigg\{\frac{C_f}{2(C_f^2-C_s^2)}\bigg[H_3\partial_{H_2}\bigg(\frac{C_a^2}{C_f}-C_f\bigg)-H_2\partial_{H_3}\bigg(\frac{C_a^2}{C_f}-C_f\bigg)\bigg]\\
&+\frac{C_a^2-3C_s^2}{2C_f(C_f^2-C_s^2)}(H_3\partial_{H_2}C_f-H_2\partial_{H_3}C_f)\bigg\},\\
\gamma^7_{12}=&-(\lambda_1-\lambda_2)\bigg\{\frac{C_f}{2(C_f^2-C_s^2)}\bigg[H_3\partial_{H_2}\bigg(\frac{C_a^2}{C_f}-C_f\bigg)-H_2\partial_{H_3}\bigg(\frac{C_a^2}{C_f}-C_f\bigg)\bigg]\\
&+\frac{C_a^2+C_s^2}{2C_f(C_f^2-C_s^2)}(H_3\partial_{H_2}C_f-H_2\partial_{H_3}C_f)\bigg\},\\
\gamma^7_{13}=&-(\lambda_1-\lambda_3)\bigg\{\frac{C_f}{2(C_f^2-C_s^2)}\bigg[H_2\partial_{H_2}\bigg(\frac{C_a^2}{C_f}-C_f\bigg)+H_3\partial_{H_3}\bigg(\frac{C_a^2}{C_f}-C_f\bigg)-\rho\bigg(\frac{C_a^2}{C_s^2}-1\bigg)\partial_{\rho}\bigg(\frac{C_a^2}{C_f}-C_f\bigg)\bigg]\\
&+\frac{C_f^2C_s^2}{2C_a^2(C_f^2-C_s^2)}\bigg[\bigg(\frac{C_a^2}{C_s^2}-1\bigg)\bigg(\frac{C_a^2}{C_f^2}-1+\rho\partial_\rho \frac{C_a^2}{C_f^2}\bigg)+\frac{2C_a^2H_2}{C_f^3}\partial_{H_2}C_f+\frac{2C_a^2H_3}{C_f^3}\partial_{H_3}C_f\bigg]\\
&+\frac{C_f(C_a^2-C_s^2)}{2(C_f^2-C_s^2)}\bigg[\frac{H_2}{C_f^2}\partial_{H_2}C_f+\frac{H_3}{C_f^2}\partial_{H_3}C_f-\bigg(\frac{1}{C_f}+\frac{\rho\partial_\rho C_f}{C_f^2}\bigg)\bigg(\frac{C_a^2}{C_s^2}-1\bigg)-\frac{1}{C_f}\bigg]+\frac{C_f^2(C_a^2-C_s^2)}{2C_a^2(C_f^2-C_s^2)}\bigg\},\\
\gamma^7_{14}=&-(\lambda_1-\lambda_4)\bigg\{\frac{C_f}{2(C_f^2-C_s^2)}\bigg[\frac{\rho}{\gamma}\partial_\rho \bigg(\frac{C_a^2}{C_f}-C_f\bigg)-\partial_S \bigg(\frac{C_a^2}{C_f}-C_f\bigg)\bigg]\\
&-\frac{C_f^2 C_s^2}{2 C_a^2(C_f^2-C_s^2)}\bigg[\frac{1}{\gamma}\partial_\rho \rho\bigg(\frac{C_a^2}{C_f^2}-1\bigg)+\frac{2C_a^2}{C_f^3}\partial_S C_f\bigg]+\frac{C_a^2-C_s^2}{2(C_f^2-C_s^2)}\bigg(\frac{\rho}{\gamma C_f}\partial_\rho C_f-\frac{\partial_S C_f}{C_f}+\frac{1}{\gamma}\bigg)\bigg\},\\
\gamma^7_{15}=&(\lambda_1-\lambda_5)\bigg\{\frac{C_f}{2(C_f^2-C_s^2)}\bigg[H_2\partial_{H_2}\bigg(\frac{C_a^2}{C_f}-C_f\bigg)+H_3\partial_{H_3}\bigg(\frac{C_a^2}{C_f}-C_f\bigg)-\rho\bigg(\frac{C_a^2}{C_s^2}-1\bigg)\partial_{\rho}\bigg(\frac{C_a^2}{C_f}-C_f\bigg)\bigg]\\
&+\frac{C_f^2C_s^2}{2C_a^2(C_f^2-C_s^2)}\bigg[\bigg(\frac{C_a^2}{C_s^2}-1\bigg)\bigg(\frac{C_a^2}{C_f^2}-1+\rho\partial_\rho \frac{C_a^2}{C_f^2}\bigg)+\frac{2C_a^2H_2}{C_f^3}\partial_{H_2}C_f+\frac{2C_a^2H_3}{C_f^3}\partial_{H_3}C_f\bigg]\\
&+\frac{C_f(C_a^2-C_s^2)}{2(C_f^2-C_s^2)}\bigg[\frac{H_2}{C_f^2}\partial_{H_2}C_f+\frac{H_3}{C_f^2}\partial_{H_3}C_f-\bigg(\frac{1}{C_f}+\frac{\rho\partial_\rho C_f}{C_f^2}\bigg)\bigg(\frac{C_a^2}{C_s^2}-1\bigg)-\frac{1}{C_f}\bigg]+\frac{C_f^2(C_a^2-C_s^2)}{2C_a^2(C_f^2-C_s^2)}\bigg\},\\
\gamma^7_{16}=&(\lambda_1-\lambda_6)\bigg\{\frac{C_f}{2(C_f^2-C_s^2)}\bigg[H_3\partial_{H_2}\bigg(\frac{C_a^2}{C_f}-C_f\bigg)-H_2\partial_{H_3}\bigg(\frac{C_a^2}{C_f}-C_f\bigg)\bigg]\\
&+\frac{C_a^2+C_s^2}{2C_f(C_f^2-C_s^2)}(H_3\partial_{H_2}C_f-H_2\partial_{H_3}C_f)\bigg\},\\
\gamma^7_{21}=&0,\ \gamma^7_{23}=0,\ \gamma^7_{24}=0,\ \gamma^7_{25}=0,\ \gamma^7_{26}=\frac{(\lambda_2-\lambda_6)C_f(C_a-C_f)(C_a^2-C_s^2)}{2C_a^2(C_f^2-C_s^2)},\\
\gamma^7_{31}=&(\lambda_3-\lambda_1)\bigg\{\frac{C_f}{2(C_f^2-C_s^2)}\bigg[\rho\bigg(\frac{C_a^2}{C_f^2}-1\bigg)\partial_{\rho}\bigg(\frac{C_a^2}{C_s}-C_s\bigg)-H_2\partial_{H_2}\bigg(\frac{C_a^2}{C_s}-C_s\bigg)-H_3\partial_{H_3}\bigg(\frac{C_a^2}{C_s}-C_s\bigg)\bigg]\\
&-\frac{C_f^2C_s^2}{2C_a^2(C_f^2-C_s^2)}\bigg[\bigg(\frac{C_a^2}{C_f^2}-1\bigg)\bigg(\frac{C_a^2}{C_s^2}-1+\rho\partial_\rho \frac{C_a^2}{C_s^2}\bigg)+\frac{2C_a^2H_2}{C_s^3}\partial_{H_2}C_s+\frac{2C_a^2H_3}{C_s^3}\partial_{H_3}C_s\bigg]\\
&+\frac{C_f(C_a^2-C_s^2)}{2(C_f^2-C_s^2)}\bigg[\bigg(\frac{1}{C_s}+\frac{\rho\partial_\rho C_s}{C_s^2}\bigg)\bigg(\frac{C_a^2}{C_f^2}-1\bigg)+\frac{1}{C_s}-\frac{H_2}{C_s^2}\partial_{H_2}C_s-\frac{H_3}{C_s^2}\partial_{H_3}C_s\bigg]+\frac{C_f^2(C_a^2-C_s^2)}{2C_a^2(C_f^2-C_s^2)}\bigg\},
\end{align*}
\newpage
\begin{align*}
\gamma^7_{32}=&-(\lambda_3-\lambda_2)\bigg\{\frac{C_f}{2(C_f^2-C_s^2)}\bigg[H_3\partial_{H_2}\bigg(\frac{C_a^2}{C_s}-C_s\bigg)-H_2\partial_{H_3}\bigg(\frac{C_a^2}{C_s}-C_s\bigg)\bigg]\\
&+\frac{C_f^2}{C_s(C_f^2-C_s^2)}\bigg(H_3\partial_{H_2}C_s-H_2\partial_{H_3}C_s\bigg)+\frac{C_f(C_a^2-C_s^2)}{2C_s^2(C_f^2-C_s^2)} \bigg(H_3\partial_{H_2}C_s-H_2\partial_{H_3}C_s\bigg)\bigg\},\\
\gamma^7_{34}=&-(\lambda_3-\lambda_4)\bigg\{\frac{C_f}{2(C_f^2-C_s^2)}\bigg[\frac{\rho}{\gamma}\partial_\rho\bigg(\frac{C_a^2}{C_s}-C_s\bigg)-\partial_S\bigg(\frac{C_a^2}{C_s}-C_s\bigg)\bigg]+\frac{C_f(C_a^2-C_s^2)}{2C_s^2(C_f^2-C_s^2)}\bigg(\rho\partial_\rho C_s-\partial_S C_s+\frac{C_s}{\gamma}\bigg)\\
&-\frac{C_f^2C_s^2}{2C_a^2(C_f^2-C_s^2)}\bigg[\frac{2C_a^2}{C_s^3}\partial_S C_s+\frac{1}{\gamma}\bigg(\frac{C_a^2}{C_s^2}-1+\rho\partial_\rho\frac{C_a^2}{C_s^2}\bigg)\bigg]\bigg\},\\
\gamma^7_{35}=&-(\lambda_3-\lambda_5)\bigg\{\frac{C_f}{2(C_f^2-C_s^2)}\bigg[\rho\bigg(\frac{C_a^2}{C_s^2}-1\bigg)\partial_{\rho}\bigg(\frac{C_a^2}{C_s}-C_s\bigg)-H_2\partial_{H_2}\bigg(\frac{C_a^2}{C_s}-C_s\bigg)-H_3\partial_{H_3}\bigg(\frac{C_a^2}{C_s}-C_s\bigg)\bigg]\\
&-\frac{C_f^2C_s^2}{2\rho C_a^2(C_f^2-C_s^2)}\bigg[\frac{1}{2}\partial_\rho \bigg(\rho(\frac{C_a^2}{C_s^2}-1)\bigg)^2+\frac{2\rho C_a^2H_2}{C_s^3}\partial_{H_2}C_s+\frac{2\rho C_a^2H_3}{C_s^3}\partial_{H_3}C_s\bigg]\\
&+\frac{C_f(C_a^2-C_s^2)}{2(C_f^2-C_s^2)}\bigg[\frac{\rho\partial_\rho C_s}{C_s^2}\bigg(\frac{C_a^2}{C_s^2}-1\bigg)+\frac{C_a^2}{C_s^3}-\frac{H_2}{C_s^2}\partial_{H_2}C_s-\frac{H_3}{C_s^2}\partial_{H_3}C_s\bigg]-\frac{C_f^2(C_a^2-C_s^2)}{2C_a^2(C_f^2-C_s^2)}\bigg\},\\
\gamma^7_{36}=&(\lambda_3-\lambda_6)\bigg\{\frac{C_f}{2(C_f^2-C_s^2)}\bigg[H_3\partial_{H_2}\bigg(\frac{C_a^2}{C_s}-C_s\bigg)-H_2\partial_{H_3}\bigg(\frac{C_a^2}{C_s}-C_s\bigg)\bigg]+\frac{C_f^2}{C_s(C_f^2-C_s^2)}\bigg(H_3\partial_{H_2}C_s-H_2\partial_{H_3}C_s\bigg)\\
&+\frac{C_f(C_a^2-C_s^2)}{2C_s^2(C_f^2-C_s^2)} \bigg(H_3\partial_{H_2}C_s-H_2\partial_{H_3}C_s\bigg)\bigg\},\\
\gamma^7_{41}=&\frac{(\lambda_4-\lambda_1)C_f^2C_s^2}{2\gamma C_a^2(C_f^2-C_s^2)}\bigg(\frac{C_a^2}{C_f^2}-1\bigg),\ \gamma^7_{42}=0,\ \gamma^7_{43}=\frac{(\lambda_4-\lambda_3)C_f^2C_s^2}{2\gamma C_a^2(C_f^2-C_s^2)}\bigg(\frac{C_a^2}{C_s^2}-1\bigg),\\
\gamma^7_{45}=&\frac{(\lambda_4-\lambda_3)C_f^2C_s^2}{2\gamma C_a^2(C_f^2-C_s^2)}\bigg(\frac{C_a^2}{C_s^2}-1\bigg),\ \gamma^7_{46}=0,\\
\gamma^7_{51}=&(\lambda_5-\lambda_1)\bigg\{\frac{C_f}{2(C_f^2-C_s^2)}\bigg[\rho\bigg(\frac{C_a^2}{C_f^2}-1\bigg)\partial_{\rho}\bigg(\frac{C_a^2}{C_s}-C_s\bigg)-H_2\partial_{H_2}\bigg(\frac{C_a^2}{C_s}-C_s\bigg)-H_3\partial_{H_3}\bigg(\frac{C_a^2}{C_s}-C_s\bigg)\bigg]\\
&+\frac{C_f^2C_s^2}{2C_a^2(C_f^2-C_s^2)}\bigg[\bigg(\frac{C_a^2}{C_f^2}-1\bigg)\bigg(\frac{C_a^2}{C_s^2}-1+\rho\partial_\rho \frac{C_a^2}{C_s^2}\bigg)+\frac{2C_a^2H_2}{C_s^3}\partial_{H_2}C_s+\frac{2C_a^2H_3}{C_s^3}\partial_{H_3}C_s\bigg]\\
&+\frac{C_f(C_a^2-C_s^2)}{2(C_f^2-C_s^2)}\bigg[\bigg(\frac{1}{C_s}+\frac{\rho\partial_\rho C_s}{C_s^2}\bigg)\bigg(\frac{C_a^2}{C_f^2}-1\bigg)+\frac{1}{C_s}-\frac{H_2}{C_s^2}\partial_{H_2}C_s-\frac{H_3}{C_s^2}\partial_{H_3}C_s\bigg]-\frac{C_f^2(C_a^2-C_s^2)}{2C_a^2(C_f^2-C_s^2)}\bigg\},\\
\gamma^7_{52}=&-(\lambda_5-\lambda_2)\bigg\{\frac{C_f}{2(C_f^2-C_s^2)}\bigg[H_3\partial_{H_2}\bigg(\frac{C_a^2}{C_s}-C_s\bigg)-H_2\partial_{H_3}\bigg(\frac{C_a^2}{C_s}-C_s\bigg)\bigg]\\
&-\frac{C_f^2}{C_s(C_f^2-C_s^2)}\bigg(H_3\partial_{H_2}C_s-H_2\partial_{H_3}C_s\bigg)+\frac{C_f(C_a^2-C_s^2)}{2C_s^2(C_f^2-C_s^2)} \bigg(H_3\partial_{H_2}C_s-H_2\partial_{H_3}C_s\bigg)\bigg\},\\
\gamma^7_{53}=&-(\lambda_5-\lambda_3)\bigg\{\frac{C_f}{2(C_f^2-C_s^2)}\bigg[\rho\bigg(\frac{C_a^2}{C_s^2}-1\bigg)\partial_{\rho}\bigg(\frac{C_a^2}{C_s}-C_s\bigg)-H_2\partial_{H_2}\bigg(\frac{C_a^2}{C_s}-C_s\bigg)-H_3\partial_{H_3}\bigg(\frac{C_a^2}{C_s}-C_s\bigg)\bigg]\\
&-\frac{C_f^2C_s^2}{2\rho C_a^2(C_f^2-C_s^2)}\bigg[\frac{1}{2}\partial_\rho \bigg(\rho\big(\frac{C_a^2}{C_s^2}-1\big)\bigg)^2+\frac{2\rho C_a^2H_2}{C_s^3}\partial_{H_2}C_s+\frac{2\rho C_a^2H_3}{C_s^3}\partial_{H_3}C_s\bigg]\\
&-\frac{C_f(C_a^2-C_s^2)}{2(C_f^2-C_s^2)}\bigg[\frac{\rho\partial_\rho C_s}{C_s^2}\bigg(\frac{C_a^2}{C_s^2}-1\bigg)+\frac{C_a^2}{C_s^3}-\frac{H_2}{C_s^2}\partial_{H_2}C_s-\frac{H_3}{C_s^2}\partial_{H_3}C_s\bigg]-\frac{C_f^2(C_a^2-C_s^2)}{2C_a^2(C_f^2-C_s^2)}\bigg\},\\
\gamma^7_{54}=&-(\lambda_5-\lambda_4)\bigg\{\frac{C_f}{2(C_f^2-C_s^2)}\bigg[\frac{\rho}{\gamma}\partial_\rho\bigg(\frac{C_a^2}{C_s}-C_s\bigg)-\partial_S\bigg(\frac{C_a^2}{C_s}-C_s\bigg)\bigg]+\frac{C_f(C_a^2-C_s^2)}{2C_s^2(C_f^2-C_s^2)}\bigg(\rho\partial_\rho C_s-\partial_S C_s+\frac{C_s}{\gamma}\bigg)\\
&+\frac{C_f^2C_s^2}{2C_a^2(C_f^2-C_s^2)}\bigg[\frac{2C_a^2}{C_s^3}\partial_S C_s+\frac{1}{\gamma}\bigg(\frac{C_a^2}{C_s^2}-1+\rho\partial_\rho\frac{C_a^2}{C_s^2}\bigg)\bigg]\bigg\},\\
\gamma^7_{56}=&(\lambda_5-\lambda_6)\bigg\{\frac{C_f}{2(C_f^2-C_s^2)}\bigg[H_3\partial_{H_2}\bigg(\frac{C_a^2}{C_s}-C_s\bigg)-H_2\partial_{H_3}\bigg(\frac{C_a^2}{C_s}-C_s\bigg)\bigg]\\
&-\frac{C_f^2}{C_s(C_f^2-C_s^2)}\bigg(H_3\partial_{H_2}C_s-H_2\partial_{H_3}C_s\bigg)+\frac{C_f(C_a^2-C_s^2)}{2C_s^2(C_f^2-C_s^2)} \bigg(H_3\partial_{H_2}C_s-H_2\partial_{H_3}C_s\bigg)\bigg\},\\
\gamma^7_{61}=&0,\ \gamma^7_{62}=-\frac{(\lambda_6-\lambda_2)C_f(C_f+C_a)(C_a^2-C_s^2)}{2C_a^2(C_f^2-C_s^2)},\\
 \gamma^7_{63}=&0,\ \gamma^7_{64}=0,\ \gamma^7_{65}=0.
 \end{align*}
\end{small}
With all above calculations, we now conclude that all the coefficients $c^i_{im}$ and $\gamma^i_{km}$ in the decomposed system \eqref{decomposed system} are of order $1$. This finishes the proof of Proposition \ref{coeff}.

\end{document}